\documentclass[11pt,english,sumlimits,reqno,oneside]{amsart}

\usepackage[T1]{fontenc}
\usepackage[utf8]{inputenc}
\usepackage{lmodern}

\usepackage{babel}

\usepackage[a4paper,textwidth=13.5cm,textheight=23cm]{geometry}

\usepackage{multirow}
\usepackage{hhline}
\usepackage{tabularx}
\usepackage{longtable}
\usepackage{subcaption}
\usepackage{booktabs}

\usepackage{xifthen}
\usepackage{ifthenx}

\usepackage{xargs}

\usepackage{graphicx}
\usepackage[table,svgnames,x11names]{xcolor}

\usepackage[disable,%
colorinlistoftodos,prependcaption,textsize=tiny,textwidth=3.5cm]{todonotes}
\presetkeys{todonotes}{fancyline}{}
\newcommandx{\todoin}[2][1=]{\todo[inline, caption={todo}, #1]{%
    \begin{minipage}{\textwidth-20pt}#2\end{minipage}}}

\newcommandx{\remark}[2][1=]{\todo[linecolor=Plum,backgroundcolor=Plum!25,bordercolor=Plum,#1]{#2}}
\newcommandx{\remarkin}[2][1=]{\remark[inline, caption={remark}, #1]{%
    \begin{minipage}{\textwidth-20pt}#2\end{minipage}}}

\newcommandx{\richard}[2][1=]{\todo[linecolor=Green,backgroundcolor=Green!25,bordercolor=Green,#1]{#2}}
\newcommandx{\richardin}[2][1=]{\richard[inline, caption={richard}, #1]{%
    \begin{minipage}{\textwidth-20pt}#2\end{minipage}}}

\newcommandx{\krystian}[2][1=]{\todo[linecolor=Pink,backgroundcolor=Pink!25,bordercolor=Pink,#1]{#2}}
\newcommandx{\krystianin}[2][1=]{\krystian[inline, caption={krystian}, #1]{%
    \begin{minipage}{\textwidth-20pt}#2\end{minipage}}}

\newcommandx{\issue}[2][1=]{\todo[linecolor=Purple,backgroundcolor=Purple!25,bordercolor=Purple,#1]{#2}}
\newcommandx{\issuein}[2][1=]{\issue[inline, caption={issue}, #1]{%
    \begin{minipage}{\textwidth-20pt}#2\end{minipage}}}

\newcommandx{\future}[2][1=]{\todo[linecolor=Blue,backgroundcolor=Blue!25,bordercolor=Blue,#1]{#2}}
\newcommandx{\futurein}[2][1=]{\future[inline, caption={future}, #1]{%
    \begin{minipage}{\textwidth-20pt}#2\end{minipage}}}

\newcommandx{\ntxt}[2][1=]{\todo[linecolor=Blue,backgroundcolor=Blue!25,bordercolor=Blue,#1]{#2}}

\usepackage{amssymb}
\usepackage{amsmath}
\usepackage{amsthm}
\usepackage{mathtools}
\mathtoolsset{showonlyrefs} 
\usepackage{exscale}
\usepackage{relsize}
\usepackage{bbm}
\usepackage{bm}
\usepackage{amsbsy}
\usepackage{mathdots}

\usepackage{mathrsfs}

\usepackage{stmaryrd}

\usepackage{mdframed}

\usepackage{breqn}

\usepackage[all]{xy}

\usepackage{fp}

\usepackage[section]{placeins}
\usepackage{float}

\usepackage{enumerate}

\newcounter{proof}
\newenvironment{myproof}%
{\stepcounter{proof}\begin{proof}}%
{\end{proof}}%
\newcounter{proofstep}[proof]
\newenvironment{proofstep}[1][]%
{\refstepcounter{proofstep}\bigskip\par\noindent%
  \ifthenelse{\isempty{#1}}
    {\textsc{Step \theproofstep. }}
    {\textsc{#1.}}
  \noindent}%
{\par}%
\newcounter{proofcase}[proof]
\newenvironment{proofcase}[1][]%
{\refstepcounter{proofcase}\bigskip\par\noindent%
  \ifthenelse{\isempty{#1}}
    {\textsc{Case \theproofcase. }}
    {\textsc{#1.}}
  \noindent}%
{\par}%

\usepackage{varioref}
\usepackage{hyperref}
\hypersetup{colorlinks,citecolor=blue}
\hypersetup{colorlinks,allcolors=DarkBlue}
\usepackage[nameinlink]{cleveref}
    
\theoremstyle{plain}
\newtheorem{thm}{Theorem}[section]
\newtheorem*{thm*}{Theorem}

\newcounter{maintheorem}

\newtheorem{mainthm}[maintheorem]{Theorem}

\newtheorem{pro}[thm]{Proposition}
\newtheorem{cor}[thm]{Corollary}
\newtheorem{lem}[thm]{Lemma}
\newtheorem{question}[thm]{Question}
\theoremstyle{definition}
\newtheorem{dfn}[thm]{Definition}
\theoremstyle{remark}

\newtheorem{rem}[thm]{Remark}

\numberwithin{equation}{section}

\AddToHook{env/pro/begin}{\crefalias{thm}{pro}}
\AddToHook{env/cor/begin}{\crefalias{thm}{cor}}
\AddToHook{env/lem/begin}{\crefalias{thm}{lem}}
\AddToHook{env/dfn/begin}{\crefalias{thm}{dfn}}
\AddToHook{env/rem/begin}{\crefalias{thm}{rem}}
\AddToHook{env/clm/begin}{\crefalias{thm}{clm}}
\AddToHook{env/question/begin}{\crefalias{thm}{question}}
\AddToHook{env/mainthm/begin}{\crefalias{maintheorem}{mainthm}}

\crefname{thm}{theorem}{theorems}
\crefname{mainthm}{theorem}{theorems}
\crefname{cor}{corollary}{corollaries}
\crefname{pro}{proposition}{propositions}
\crefname{lem}{lemma}{lemmas}
\crefname{dfn}{definition}{definitions}
\crefname{rem}{remark}{remarks}
\crefname{clm}{claim}{claims}
\crefname{question}{question}{questions}

\newcommandx{\textref}[2][1=]{\hyperref[#2]{#1\ref*{#2}}}
\newcommandx{\textrefp}[2][1=]{(\hyperref[#2]{#1\ref*{#2}})}

\newcommand{\charfun}{\scalebox{1.0}{\ensuremath{\mathbbm 1}}}

\DeclareMathOperator{\Id}{Id}

\DeclareMathOperator{\spn}{span}

\DeclareMathOperator{\cond}{\mathbb{E}}

\DeclareMathOperator{\ord}{ord}
\DeclareMathOperator{\cha}{\mathbbm{1}}
\renewcommand{\subset}{\subseteq}

\begin{document}

\title[Diagonal operators on Janson-Sobolev spaces]%
{Diagonal operators on Janson-Sobolev and Janson-Sobolev-Hardy spaces}%

\author[K.~Kazaniecki]{Krystian Kazaniecki}%
\address{Krystian Kazaniecki, Institute of Mathematics, University of Warsaw, Stefana Banacha 2, 02-097 Warszawa, Poland}%
\email{k.kazaniecki@mimuw.edu.pl}%

\author[R.~Lechner]{Richard Lechner}%
\address{Richard Lechner, Institute of Analysis, Johannes Kepler University Linz, Altenberger
  Strasse 69, A-4040 Linz, Austria}%
\email{Richard.Lechner@jku.at}%

\date{October 1, 2026}%



\thanks{The first author was supported by the Austrian Science Fund (FWF), grant I5231, and the second author was supported by the FWF, grant
  P34414.}%

\begin{abstract}
  We study the Banach space and operator factorization structure of Janson-Sobolev and Janson-Sobolev-Hardy spaces.  This new class of martingale
  spaces is determined by a $q$-adic filtration, a subspace $V\subset \mathbb{R}_0^{l\times q}$, and a rearrangement invariant function space $X$.  Our main result shows
  that, for every bounded diagonal operator $D$, the operator
  \begin{equation*}
    S
    = \sum_{t=1}^s \lambda_{\mathcal U}^{k_t}(D)Q_{\mathcal K_t}^{\mathcal B}
  \end{equation*}
  determined by the linear functionals $\lambda_{\mathcal U}^{k_t}(D)$ and the canonical projections $Q_{\mathcal K_t}^{\mathcal B}$, almost projectionally
  factors through $D$ with constant $1^+$.  As consequences, we obtain factorization results for diagonal operators both under a natural boundedness
  condition on the canonical projections and for all spaces equipped with the $L^1$-norm.
\end{abstract}

\maketitle%

\tableofcontents%

\makeatletter%
\providecommand\@dotsep{5}%
\def\listtodoname{List of Todos}%
\def\listoftodos{\@starttoc{tdo}\listtodoname}%
\makeatother%

\newcounter{mycounter}%

\section{Introduction}
\label{sec:introduction}

It is well established that martingale counterparts of objects from harmonic analysis provide useful models for analytic phenomena. A recurring
strategy is to study the properties of these martingale analogues and subsequently transfer the resulting insights to the classical analytic setting.
In this spirit, R.~Ayoush, D.~M.~Stolyarov, and M.~Wojciechowski~\cite{ASW} introduced Sobolev martingales and studied their fine properties.  Their
martingale model was inspired by S.~Janson's work~\cite{sJ77} on characterizations of martingale $H^1$ in terms of systems of singular integral
transforms.  In the present paper, we take a different perspective and investigate these spaces from the viewpoint of Banach space theory.

Their construction starts with a $q$-adic filtration and a subspace $V$ of the finite-dimensional matrix space $\mathbb{R}^{l\times q}_0$ consisting of matrices
with mean zero rows.  The choice of $V$ determines which martingale differences are admissible.  The resulting martingales reproduce several
characteristic features of functions in the classical Sobolev space $W^{1,1}$ or $\mathrm{BV}$: they established a martingale analogue of Van
Schaftingen's limiting Sobolev estimate~\cite{ASW}, sharp lower bounds for the Hausdorff dimension of terminal distributions~\cite{ASW}, and
martingale trace inequalities~\cite{MR4865002}.  Subsequent work~\cite{zbMATH07961358} has demonstrated that this analogy extends to further fine
properties of Sobolev functions and PDE-constrained measures.  Recent works by
D.~Stolyarov~\cite{zbMATH07733533,stolyarov2026anisotropicbourgainbrezisinequalities} highlight the fruitful transfer of techniques and insights from
discrete martingales to the classical analytic setting.

In this paper, we introduce a broader class of spaces of such martingales and investigate their Banach space structure.  We generalize the models by
replacing the ambient $L^1$ space with a rearrangement invariant space $X$ and their square function variants.

We call the resulting spaces \emph{Janson-Sobolev spaces} or \emph{Janson-Sobolev-Hardy spaces}, respectively, and denote them by
$Z_V = Z_V((\mathcal F_j),X,\boldsymbol{\sigma};\mathbb{R}^l)$.  As described above, the finite-dimensional subspace
$V\subset \mathbb{R}^{l\times q}_0$ specifies the admissible martingale differences, while the sequence of random signs $\boldsymbol{\sigma}$ distinguishes the Janson-Sobolev
and Janson-Sobolev-Hardy variants.  We develop methods that provide a unified approach to this class of spaces whose isomorphic structures exhibit
markedly different behavior.  Indeed, some choices of $V$ lead to spaces isomorphic to $H^1$, others give $L^1$-type spaces (both of which are
primary), and still others produce spaces which are not primary, e.g., $L^1\oplus H^1$.  Thus, these spaces provide a testing ground in which both positive
and negative aspects of primarity and related factorization properties occur naturally.

Recall that a Banach space $E$ is \emph{primary} if every decomposition $E = E_0\oplus E_1$ has a summand isomorphic to $E$.  A strongly related
operator-theoretic principle is the \emph{primary factorization property}: for every bounded operator $T\colon E\to E$, the identity $\Id_E$ factors
either through $T$ or through $\Id_E-T$.  It is well known that in the presence of the usual self-similarity assumptions, such as the accordion
property, the primary factorization property yields primarity.

A common strategy for establishing the primary factorization property is to reduce the problem for general operators to the case of operators that are
diagonal with respect to a Schauder basis of the underlying Banach space.  This reduction underscores the significance of understanding the
corresponding factorization problem for diagonal operators.  We say that a Banach space $E$ has the \emph{diagonal primary factorization property} if
there exists a Schauder basis $(e_n)_{n=1}^{\infty}$ of $E$ such that, for every bounded diagonal operator $D\colon E\to E$ with respect to
$(e_n)_{n=1}^{\infty}$, the identity factors through either $D$ or $\Id_E-D$ (refer to \Cref{dfn:factorization-modes} for a comprehensive list of
definitions).

Although its definition refers to a basis, the diagonal primary factorization property is an isomorphic invariant (see
\Cref{rem:diagonal-primary-factorization-property}).  As one of our examples shows (see \Cref{thm:non-primary-janson-sobolev}), it need not imply the
primary factorization property (which is stronger) or primarity.  For the class of Janson-Sobolev and Janson-Sobolev-Hardy spaces, we also prove that
primarity implies the diagonal primary factorization property (see \Cref{rem:factorization-properties}; see also \Cref{pro:factorization-properties}).

Despite the rich structure of Janson-Sobolev and Janson-Sobolev-Hardy spaces, we show that every bounded diagonal operator can, in general, be reduced
via factorization to fundamental building blocks, which we call canonical projections (see \Cref{mainthm:picewise-constant}).  Based on this result,
we obtain the diagonal primary factorization property for all Janson-Sobolev and Janson-Sobolev-Hardy spaces with Property~I, which is characterized
in terms of canonical projections (see \Cref{mainthm:main-prop-I}).  Finally, we show that all $L^1$-normed Janson-Sobolev and Janson-Sobolev-Hardy
spaces have the diagonal primary factorization property (see \Cref{mainthm:main-L1}).  To introduce the necessary notation, we postpone the detailed
statements of these results to \Cref{sec:main-results}.

%

\section{Preliminaries}
\label{sec:notation}

In this section, we introduce Janson-Sobolev and Janson-Sobolev-Hardy spaces, canonical projections, diagonal operators, and related definitions and
objects.  For convenience, throughout this paper all function spaces are understood to consist of mean zero functions.

\subsection{Janson-Sobolev and Janson-Sobolev-Hardy spaces}
\label{sec:basic-definitions}

Let $\mathbb{R}^n_0$ denote the subspace of all $(a_k)_{k=1}^n\in \mathbb{R}^n$ with mean zero coefficients $(a_k)_{k=1}^n$.  Moreover,
$\mathbb{R}^{m\times n}_0$ denotes the subspace of real $m$ by $n$ matrices with mean zero rows.

Suppose $V$ is a subspace of the matrix space $\mathbb{R}^{l\times q}_0$ which we equip with the scalar product $\langle \cdot ,\cdot \rangle\colon V^2\to \mathbb{R}$ given by
\begin{equation}\label{eq:defscalar}
  \langle u, v \rangle
  = \sum_{i=1}^l \frac{1}{q} \sum_{j=1}^q u_{ij} v_{ij},
  \qquad u = (u_{ij})_{\substack{1\leq i\leq l\\1\leq j\leq q}},\ v = (v_{ij})_{\substack{1\leq i\leq l\\1\leq j\leq q}}\in \mathbb{R}^{l\times q}_0,
\end{equation}
and denote the matrix norm induced by this scalar product by $\|\cdot\|_V$.

The set of atoms of any $\sigma$-algebra $\Sigma$ is denoted by $\mathcal{A}\Sigma$.  Let $(\mathcal{F}_j)_{j=0}^{\infty}$ be a filtration on $[0,1)$ such that for each
$P\in \mathcal{AF}_j$ the set of successors $\mathcal{S}(P)$ consists of $q$ intervals in $\mathcal{AF}_{j+1}$, $j\in \mathbb{N}_0$ of equal Lebesgue measure and that
$\mathcal{AF}_0 = \{[0,1)\}$.  We call such a filtration $(\mathcal{F}_j)_{j=0}^{\infty}$ \emph{$q$-adic} and put
$\mathcal{A} = \bigcup_{j=0}^{\infty} \mathcal{AF}_j$.  By $|\cdot|$, we denote the Lebesgue measure on the unit interval, and for
$1\leq p\leq \infty$, $L^p$ denotes the associated Lebesgue space.

For a $\sigma$-algebra $\mathcal{F}$, we denote as usual the conditional expectation of a function $f$ with respect to $\mathcal{F}$ by
$\cond(f | \mathcal{F})$.  Moreover, for a given filtration $(\mathcal{F}_j)_{j=0}^{\infty}$, we define
\begin{equation}\label{eq:24}
  \Delta_j^{(\mathcal{F}_n)}(f)
  = \cond( f | \mathcal{F}_j) - \cond( f | \mathcal{F}_{j-1}),
  \qquad j\in \mathbb{N}_0,
\end{equation}
where we put $\cond( f | \mathcal{F}_{-1}) = 0$.  If unambiguous, we will omit the superscript, i.e., we write $\Delta_j(f)$.

We now define the following linear orders on $\mathcal{A}$ and $\mathcal{A}\times \{1, \dots, r\}$.  We introduce the order $<_1$ on the set of atoms
$\mathcal{A}$ by viewing $\mathcal{A}$ as the set of vertices of a tree and ordering them according to the natural tree order --- namely, the order induced by a traversal
from left to right and top to bottom.
  
Moreover, we linearly order the set $\mathcal{A}\times\{1, \dots, r\}$ by $<_2$ as follows: $(P_0,k_0) <_2 (P_1,k_1)$ if and only if $P_0 <_1 P_1$, or if
$P_0 = P_1$ and $k_0 < k_1$.

In the following, we will omit the subindices from our orders whenever it is clear from the context.  A depiction of the orders $<_1$ and $<_2$ can be
found in \Cref{fig:order-1} and \Cref{fig:order-2}.

\begin{figure}
  \centering
  \begin{tikzpicture}
    \node[anchor=south west,inner sep=0] (img) {\includegraphics[width=.8\textwidth]{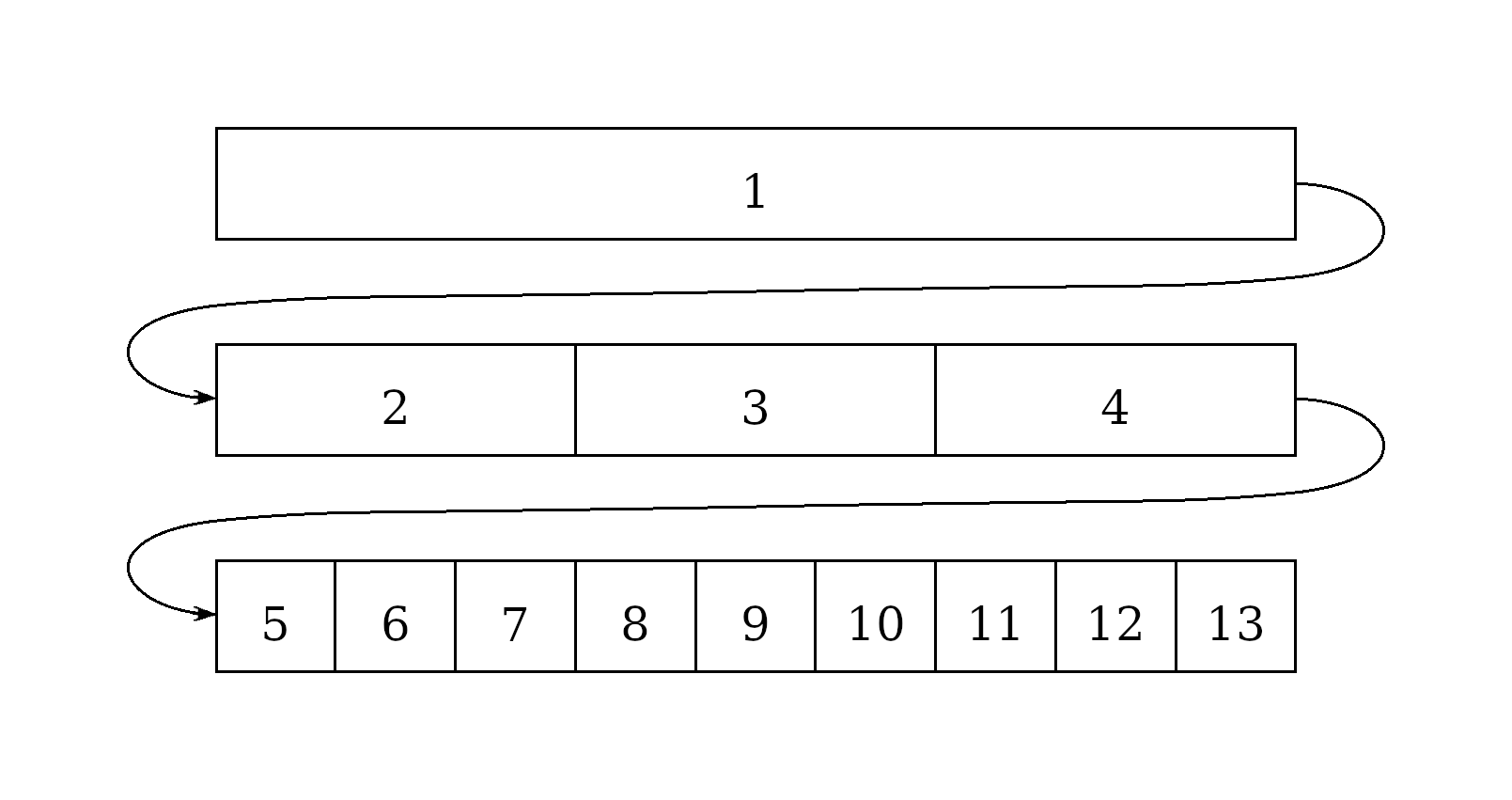}};
    \begin{scope}[x={(img.south east)},y={(img.north west)}]
      \node at (0.02,0.77) {{\small $\mathcal{AF}_0$}};%
      \node at (0.02,0.5) {{\small $\mathcal{AF}_1$}};%
      \node at (0.02,0.22) {{\small $\mathcal{AF}_2$}};%
    \end{scope}
  \end{tikzpicture}
  \caption{This figure illustrates the order $<_1$ on the first three levels of intervals in $\mathcal{A}$ in the triadic setting.}
  \label{fig:order-1}
\end{figure}

\begin{figure}
  \centering
  \begin{tikzpicture}
    \node[anchor=south west,inner sep=0] (img) {\includegraphics[width=\textwidth]{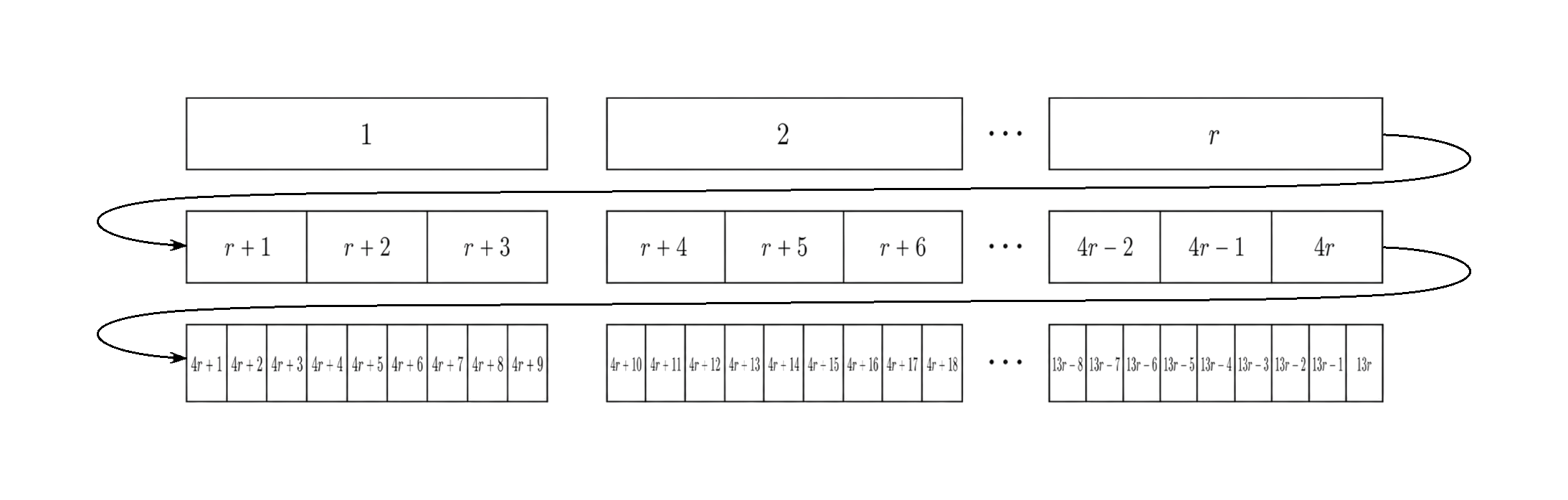}};
    \begin{scope}[x={(img.south east)},y={(img.north west)}]
      \node at (0.02,0.73) {{\tiny $\mathcal{AF}_0$}};%
      \node at (0.02,0.5) {{\tiny $\mathcal{AF}_1$}};%
      \node at (0.02,0.27) {{\tiny $\mathcal{AF}_2$}};%
    \end{scope}
  \end{tikzpicture}
  \caption{Representation of the order $<_2$ on $\mathcal{A}\times\{1,\dots,r\}$ for the first three levels of triadic intervals in $\mathcal{A}$.}
  \label{fig:order-2}
\end{figure}

On a $q$-adic filtration $(\mathcal{F}_j)_{j=0}^{\infty}$, we define $\mathrm{ord}\colon \mathcal{A}\to \{1, \dots, q\}$ such that for any given
$P\in \mathcal{A}$, the map $\ord_{|\mathcal{S}(P)}$ is compatible with the order $<_1$, i.e., $\ord(R_0) < \ord(R_1)$ if and only if $R_0 <_1 R_1$
for all $R_0, R_1\in \mathcal{S}(P)$.

For any orthonormal basis $\mathcal{B}=(v^1,\ldots, v^r)$ of $V$, we define the system of functions
$(\varphi_{P}^{k,\mathcal{B}} : P\in\mathcal{A},\ 1\leq k\leq r)$ by
\begin{equation}\label{eq:119}
  \varphi_{P}^{k,\mathcal{B}}
  = \sum_{R\in \mathcal{S}(P)} v_{\mathrm{ord}(R)}^k \cha_R,
  \qquad P\in \mathcal{A},\ 1\leq k\leq r,
\end{equation}
where $v^k_i$ denotes the $i$-th column of the matrix $v^k$.  Note that $(\varphi_P^{k,\mathcal{B}} : P\in \mathcal{A},\ 1\leq k\leq r)$ is orthogonal
in $L^2([0,1))$.  If unambiguous, we will omit the superscript $\mathcal{B}$ and simply write $\varphi_P^k$ instead of $\varphi_P^{k,\mathcal{B}}$.

\begin{dfn}\label{dfn:r.i.-spaces}
  Let $([0,1),\Sigma,\mu)$ denote a probability space and suppose that the Banach space $(X,\|\cdot\|)$ is a linear subspace of the real-valued
  $\mu$-measurable functions on $[0,1)$.  We say that $(X, \|\cdot\|_X)$ is \emph{rearrangement invariant} if for any two real-valued $\mu$-measurable
  functions $f$, $g$ the following two conditions are satisfied:
  \begin{enumerate}[(a)]
  \item if $|f|\leq |g|$ $\mu$-a.e.\@ and $g\in X$, then $f\in X$ and $\|f\|_X\leq \|g\|_X$;
  \item if $|f|$ and $|g|$ are equi-measurable and $g\in X$, then $f\in X$ and $\|f\|_X = \|g\|_X$.
  \end{enumerate}
\end{dfn}

For more information on rearrangement invariant spaces, we refer to~\cite{zbMATH03776286,MR540367,MR0527010,MR928802}.

\begin{dfn}\label{dfn:scaling}
  Let $X$ denote a rearrangement invariant space and let $0 < s \leq 1$.  Then we define the upscaling and downscaling coefficients $\gamma_X^u(s)$
  and $\gamma_X^d(s)$ by
  \begin{equation}\label{eq:71}
    \begin{aligned}
      \gamma_X^u(s)
      &= \sup\bigl\{ \|g\|_X : \|f\|_X = 1,\ |\{ |f| > \alpha \}| = s\cdot |\{ |g| > \alpha \}|\, (\forall \alpha > 0) \bigr\},\\
      \gamma_X^d(s)
      &= \sup\bigl\{ \|g\|_X^{-1} : \|f\|_X = 1,\ |\{ |f| > \alpha \}| = s\cdot |\{ |g| > \alpha \}|\, (\forall \alpha > 0) \bigr\}.
    \end{aligned}
  \end{equation}
\end{dfn}

\begin{rem}\label{rem:scaling}
  Since $X$ is rearrangement invariant, we can restrict the functions in~\eqref{eq:71} so that $f$ is a scaling of $g$, i.e.,
  \begin{equation}
    \gamma_X^d(s)
    = \sup\left\{ \|g\|_X^{-1} :  \|f\|_X = 1,\ f(x) =
      \begin{cases}
        g(x/s) & x\in [0,s)\\
        0 & x\in [s,1)
      \end{cases}
    \right\},
  \end{equation}
  and similarly for $\gamma_X^u$.  Thus, it is clear that $\gamma_X^d(s)$ and $\gamma_X^u(s)$ are just the norms of the scaling operator and its
  inverse on $X$.  We refer to~\cite[Chapter~II, Theorem~4.4]{zbMATH03776286} to see that $\gamma_X^u(s),\ \gamma_X^d(s) < \infty$ for all
  $0 < s \leq 1$ on all rearrangement invariant spaces.

  From the above observations and the definition itself it is clear that
  \begin{equation}\label{eq:153}
    \gamma_X^d(s)\leq 1\leq \gamma_X^u(s)
    \qquad\text{and}\qquad
    \gamma_X^u(s)\cdot \gamma_X^d(s) \geq 1
  \end{equation}
  for all $0 < s \leq 1$.  The proof of Theorem~4.5 in~\cite[Chapter~II]{zbMATH03776286} gives us that $\gamma_X^u(s) \to 1$ whenever $s\to 1$ from below, and
  thus we also obtain $\gamma_X^d(s)\to 1$ if $s\to 1$ from below.
\end{rem}

\begin{dfn}\label{dfn:W-auxiliary}
  Let $\boldsymbol{\sigma} = (\sigma_j)_{j=0}^{\infty}$, where $\boldsymbol{\sigma}$ is either a sequence of independent Rademacher functions, or
  $\sigma_j = 1$, $j\in \mathbb{N}_0$, and let $\cond_{\boldsymbol{\sigma}}$ denote the expectation.  Given a basis $\mathcal{B}$ for a matrix subspace
  $V$ of $\mathbb{R}_0^{l\times q}$, and a rearrangement invariant space $X$ on $[0,1)$, we define
  $Z_V = Z_V((\mathcal{F}_j),X,\boldsymbol{\sigma};\mathbb{R}^l)$ as the completion of $\spn\{\varphi_P^{k,\mathcal{B}} : P\in \mathcal{A},\ 1\leq k\leq r\}$ under the norm
  \begin{equation}\label{eq:39}
    \Bigl\| \sum_{k=1}^r \sum_{P\in \mathcal{A}} a_P^k \varphi_P^{k,\mathcal{B}} \Bigr\|_{Z_V}
    := \Bigl\|
    x\mapsto \cond_{\boldsymbol{\sigma}} \Bigl\|
    \sum_{j=0}^{\infty} \sigma_j \sum_{P\in \mathcal{AF}_j} \sum_{k=1}^r a_P^k \varphi_P^{k,\mathcal{B}}(x)
    \Bigr\|_{\ell^2_l}
    \Bigr\|_X,
  \end{equation}
  where $a_P^k\in \mathbb{R}$, $P\in \mathcal{A}$, $1\leq k\leq r$.  If $\boldsymbol{\sigma}$ is independent, we call $Z_V$ a \emph{Janson-Sobolev-Hardy space}, and if
  $\boldsymbol{\sigma}$ is constant, $Z_V$ is called \emph{Janson-Sobolev space}.

  Finally we define the $q$-adic Hardy space
  $H^1((\mathcal{F}_j);\mathbb{R}^l) := Z_{\mathbb{R}_{0}^{l\times q}}((\mathcal{F}_j),L^1,\boldsymbol{\sigma};\mathbb{R}^l)$, where
  $\boldsymbol{\sigma}$ denotes a sequence of independent Rademacher functions.  The dual of $H^1((\mathcal{F}_j);\mathbb{R}^l)$ is denoted by
  $\mathrm{BMO}((\mathcal{F}_j);\mathbb{R}^l)$, and its equivalent norm is given by
  \begin{equation}\label{eq:93}
    \|f\|^2
    = \sup_j \bigl\| \cond\bigl( \|f - \cond( f | \mathcal{F}_j)\|_{\ell^2_l}^2 \big| \mathcal{F}_j \bigr) \bigr\|_{L^{\infty}}.
  \end{equation}
  We refer to~\cite[Chapters~7 and~10]{MR3617459} for more information on vector valued $H^1$ and $\mathrm{BMO}$.
\end{dfn}

\begin{rem}\label{rem:H1}
  First, note that by projecting onto each coordinate, we can identify $H^1((\mathcal{F}_j);\mathbb{R}^l)$ with the direct sum of
  $H^1((\mathcal{F}_j);\mathbb{R})$.  Second, by Maurey's isomorphism (see~\cite{MR551766,MR0586594}) and the classification result~\cite{MR0879418}
  by P.~F.~X.~Müller, we know that the space $H^1((\mathcal{F}_j);\mathbb{R})$ does not depend on $(\mathcal{F}_j)$, and therefore, we sometimes
  abbreviate $H^1((\mathcal{F}_j);\mathbb{R}^l)$ by $H^1$.

  A theorem of P.~F.~X.~Müller~\cite[Theorem~6]{MR0955660} states that dyadic $H^1$ has the primary factorization property.  Since this property is an
  isomorphic invariant, the preceding discussion implies that $H^1((\mathcal{F}_j);\mathbb{R}^l)$ also has the primary factorization property.  For a more comprehensive
  overview, we refer the reader to~\cite{MR2157745}.
\end{rem}

\begin{rem}\label{rem:norm}
  Note that for any $f\in Z_V$, we have
  \begin{equation}\label{eq:204}
    \|f\|_{Z_V}
    = \Bigl\| x\mapsto \cond_{\boldsymbol{\sigma}} \Bigl\| \sum_{j=0}^{\infty} \sigma_j \Delta_j^{(\mathcal{F}_n)}f(x) \Bigr\|_{\ell^2_l} \Bigr\|_X,
  \end{equation}
  and therefore $\|\cdot\|_{Z_V}$ does not depend on the choice of $\mathcal{B}$.
\end{rem}

The Janson-Sobolev spaces introduced above extend the spaces that have arisen in the study of fine properties of measures in recent
works~\cite{ASW,MR4865002,zbMATH07961358}.  Specifically, we pass from the classical Lebesgue spaces to rearrangement invariant spaces and their
square function variants.  Implicitly, those spaces already appeared much earlier in S.~Janson's work~\cite{sJ77}.

To avoid certain technical inconveniences, we assume that all functions in the Janson-Sobolev space have mean zero.  Note that using the same method
as in~\cite[Proposition~4.10]{MR4894823} gives us that including constant functions yields a space which is isomorphic to our Janson-Sobolev space
with mean zero functions.

\begin{rem}\label{rem:conditional}
  Suppose that $\mathcal{G}$ is a $\sigma$-algebra generated by a finite subset of $\mathcal{A}$.  Using Jensen's inequality, it is a simple task to verify that
  $\cond(\cdot | \mathcal{G})\colon Z_V\to Z_V$ is a contraction.
\end{rem}

\begin{rem}\label{rem:W-auxiliary}
  For $1\leq p< \infty$, the space $Z_V = Z_V(L^p) = Z_V((\mathcal{F}_j),L^p,\boldsymbol{\sigma};\mathbb{R}^l)$ is isomorphic to $\ell^p(Z_V)$.  To see this, take any sequence
  $(P_j)_{j=1}^{\infty}$ of pairwise disjoint atoms in $\mathcal{A}$, and for each $j\in \mathbb{N}$ define $Z_V((P_j)_{j=1}^{\infty})$ as the closure of
  \begin{equation}
    \spn\{\varphi_P^k : P\in \mathcal{A},\ 1\leq k\leq r,\ \text{$\exists j\in \mathbb{N}$ such that $P\subset P_j$}\}
  \end{equation}
  in $Z_V$.  The space $Z_V((P_j)_{j=1}^{\infty})$ is complemented in $Z_V$ by the projection $\text{Pr}\colon Z_V\to Z_V$ given by
  \begin{equation}
    \text{Pr} f
    = \cha_{\bigcup_{j=1}^{\infty} P_j}\cdot (f - \cond( f | \mathcal{G})),
  \end{equation}
  where $\mathcal{G}$ denotes the $\sigma$-algebra generated by $(P_j)_{j=1}^{\infty}$.  By \Cref{rem:conditional} and since the $\sigma$-algebra
  $\mathcal{G}$ is generated by atoms, $\text{Pr}$ is bounded by $2$.  By scaling, we see that $Z_V$ is $1$-isomorphic to the complemented subspace
  \begin{equation}
    Z_n(L^p)
    = \{ \cha_{P_n}\cdot f : f\in Z_V((P_j)_{j=1}^{\infty}) \}
  \end{equation}
  of $Z_V((P_j)_{j=1}^{\infty})$ for each $n\in \mathbb{N}$.  Therefore, by the pairwise disjointness of $(P_j)_{j=1}^{\infty}$, we obtain
  \begin{equation}\label{eq:121}
    Z_V((P_j)_{j=1}^{\infty})
    \simeq \ell^p(Z_n(L^p))
    \simeq \ell^p(Z_V(L^p)).
  \end{equation}
  Clearly, $\ell^p(Z_V(L^p))$ is isomorphic to its $\ell^p$ direct sum, and thus, by~\eqref{eq:121}, so is $Z_V((P_j)_{j=1}^{\infty})$.  Hence,
  by~Pe{\l}czy\'{n}ski's decomposition technique, $Z_V(L^p)$ is isomorphic to $Z_V((P_j)_{j=1}^{\infty})$.  Using~\eqref{eq:121} again shows that
  $Z_V(L^p)$ is isomorphic to $\ell^p(Z_V(L^p))$.
\end{rem}

\begin{dfn}\label{dfn:janson-operator}
  Given an operator $L\colon \mathbb{R}^q_0\to \mathbb{R}^q_0$, we informally define the corresponding \emph{Janson operator} $T_L$ by
  \begin{equation*}
    T_L\bigl( \sum_{R\in \mathcal{S}(P)} v_{\mathrm{ord}(R)} \cha_R \bigr)
    = \sum_{R\in \mathcal{S}(P)} (Lv)_{\mathrm{ord}(R)} \cha_R,
  \end{equation*}
  for any $P\in \mathcal{A}$, $v = (v_1,\dots, v_q)\in \mathbb{R}^q_0$.
\end{dfn}

The significance of Janson operators is highlighted by the following two key results due to S.~Janson (see~\cite{sJ77}).
\begin{thm}[{\cite[Theorem~2]{sJ77}}]\label{thm:janson-0}
  Janson operators are bounded on $\mathrm{BMO}$ and $H^1$.
\end{thm}

\begin{thm}[{\cite[Theorem~3]{sJ77}}]\label{thm:janson-1}
  Assume that $A_1,\ldots, A_m\colon \mathbb{R}^q_0\to \mathbb{R}^q_0$ are linear operators that do not have a common real eigenvector and suppose that
  $(f_n)_{n=0}^{\infty}$ is a martingale such that $(f_n)_{n=0}^{\infty}$ and $(T_{A_j}f_n)_{n=0}^{\infty}$ are bounded in $L^1$.  Then
  $(f_n)_{n=0}^{\infty}$ and $(T_{A_j}f_n)_{n=0}^{\infty}$ are closed martingales in $H^1$.
\end{thm}

\begin{dfn}\label{dfn:substriction}
  Given a Janson-Sobolev or Janson-Sobolev-Hardy space $Z_V$ and $Q\in \mathcal{AF}_n$ for some $n\in \mathbb{N}_0$, we define the operator $\Theta_Q\colon Z_V\to Z_V$ by
  \begin{equation}\label{eq:205}
    \Theta_Q \bigl( \sum_{P\in \mathcal{A}} \sum_{k=1}^r a_P^k \varphi_P^k \bigr)
    = \sum_{\substack{P\in \mathcal{A}\\P\subset Q}} \sum_{k=1}^r a_P^k \varphi_P^k.
  \end{equation}
\end{dfn}
Note that since $\Theta_Q f = \charfun_Q\cdot (f - \cond(f|\mathcal{F}_n))$, \Cref{rem:conditional}, and the contractivity of multiplication by
$\charfun_Q$ on $(\Id_{Z_V} - \cond(f|\mathcal{F}_n))(Z_V)$, we obtain
\begin{equation}\label{eq:206}
  \|\Theta_Q\colon Z_V\to Z_V\|
  \leq 2.
\end{equation}

\subsection{Canonical projections}
\label{sec:canon-proj}

In this subsection, we introduce the canonical projection operators associated with an orthonormal basis of $V$ and the corresponding classification
of the space $Z_V$.  As we will see in \Cref{sec:stabilization}, every diagonal operator can be reduced (by factorization) to a linear combination of
the canonical projections.  This bears striking similarity to the Capon projection introduced in~\cite{MR4816119}, and prior to that by M.~Capon's
condition $\mathcal{P}_p$ originating in~\cite{MR0688955}.

\begin{dfn}\label{dfn:Q}
  For any orthonormal basis $\mathcal{B}=(v^1,\ldots, v^r)$ of $V$ and any $\mathcal{K}\subset \{1,\dots, r\}$, we define the possibly unbounded
  operator $Q_{\mathcal{K}}^{\mathcal{B}}\colon Z_V\to Z_V$ by
  \begin{equation*}
    \sum_{k=1}^r \sum_{P\in \mathcal{A}} a_P^k \varphi_P^{k,\mathcal{B}}
    \mapsto \sum_{k\in \mathcal{K}} \sum_{P\in \mathcal{A}} a_P^k \varphi_P^{k,\mathcal{B}},
  \end{equation*}
  and put $Q_k^{\mathcal{B}} = Q_{\{k\}}^{\mathcal{B}}$ for $1\leq k\leq r$.  Note that $Q_{\{1, \dots, r\}}^{\mathcal{B}} = \Id_{Z_V}$.  If
  unambiguous, we will omit the superscript $\mathcal{B}$ and simply write $Q_k$ instead of $Q_k^{\mathcal{B}}$.
\end{dfn}

\begin{dfn}\label{dfn:finest-bounded-partition}
  Given an orthonormal basis $\mathcal{B}$ of $V$, we call the uniquely determined partition $\mathcal{K}_1, \dots, \mathcal{K}_s\neq \emptyset$ of
  $\{1, \dots, r\}$ such that
  \begin{itemize}
  \item all the $Q_{\mathcal{K}_t}^{\mathcal{B}}\colon Z_V\to Z_V$, $1\leq t\leq s$ are bounded, and
  \item whenever $\emptyset\neq \mathcal{K}\subsetneq \mathcal{K}_t$ for some $1\leq t\leq s$, then $Q_{\mathcal{K}}^{\mathcal{B}}$ is unbounded
  \end{itemize}
  the \emph{finest bounded partition of $\{1, \dots, r\}$ for $\mathcal{B}$ with respect to $Z_V$}.  If unambiguous, we simply refer to it as the
  \emph{finest bounded partition}.
\end{dfn}

As we will show in \Cref{thm:trichotomy}, depending on our choice of the subspace $V$ and the rearrangement invariant space $X$, the space
$Z_V = Z_V(X) = Z_V((\mathcal{F}_j),X,\boldsymbol{\sigma};\mathbb{R}^l)$ can exhibit at least one of the following three types of behavior.
\begin{dfn}\label{dfn:space-type}
  Let $Z_V = Z_V((\mathcal{F}_j),X,\boldsymbol{\sigma};\mathbb{R}^l)$ denote a Janson-Sobolev space or a Janson-Sobolev-Hardy space.
  \begin{enumerate}
  \item[Property~I] If we can choose an orthonormal basis $\mathcal{B}$ of $V$ such that the only bounded canonical projections
    $Q_{\mathcal{K}}^{\mathcal{B}}\colon Z_V\to Z_V$ with respect to this basis are either $Q_{\mathcal{K}}^{\mathcal{B}} = \Id_{Z_V}$ or
    $Q_{\mathcal{K}}^{\mathcal{B}} = 0$, then we say that $Z_V$ has \emph{Property~I}.
  \item[Property~IIA] If all canonical projections $Q_{\mathcal{K}}^{\mathcal{B}}\colon Z_V\to Z_V$ for all choices of an orthonormal basis
    $\mathcal{B}$ of $V$ are bounded and there exists a non-zero vector $b\in \mathbb{R}^q$ and subspace $W \subset \mathbb{R}^l$ such that
    \begin{equation}\label{eq:defofvectb}
      V
      = \{w b^t \,:\, w\in W\}
    \end{equation}
    we say that $Z_V$ has \emph{Property~IIA}.
  \item[Property~IIB] If all canonical projections $Q_{\mathcal{K}}^{\mathcal{B}}\colon Z_V\to Z_V$ for all choices of an orthonormal basis
    $\mathcal{B}$ of $V$ are bounded and $Z_V$ is not of Property~IIA, we say that $Z_V$ has \emph{Property~IIB}.
  \end{enumerate}
\end{dfn}
Building on the work of S.~Janson and J.-A.Chao~\cite{zbMATH03684627}, R.~Ayoush, D.~Stolyarov, and M.~Wojciechowski introduced in~\cite{ASW} a set of
\emph{structural conditions} on the subspace $V$, formulated in terms of the presence of certain rank-one matrices.  Throughout these works, as well
as in the present paper, whether or not $V$ contains rank one matrices plays an important role.

\begin{rem}\label{rem:dimension-1}
  Note that if $\dim V = 1$, then $Z_V$ has not only Property~I, but also either Property~IIA or Property~IIB.  However, for $\dim V > 1$, Property~I,
  Property~IIA and Property~IIB are mutually distinct.
\end{rem}

\begin{rem}\label{rem:space-type}
  We claim that for spaces with Property~IIA, the boundedness of the canonical projections $Q_{\mathcal{K}}^{\mathcal{B}}$ follows automatically from
  the structure of the subspace $V$ (see~\eqref{eq:defofvectb}), and therefore need not be assumed.

  To see this, suppose that $\mathcal{B} = (v^1, \dots, v^{r})$ is any orthonormal basis for the subspace $V$, which has the form given
  by~\eqref{eq:defofvectb}, i.e., $v^k = w^kb^t$, $1\leq k\leq r$.  By scaling $b\neq 0$ we can assume that $(w^k : 1\leq k\leq r)$ is orthonormal in
  $\mathbb{R}^l$.  Then clearly, for fixed $P\in \mathcal{A}$, $1\leq k\leq r$ and $x\in [0,1)$ we have
  \begin{equation}\label{eq:68}
    \begin{aligned}
      \varphi_P^{k,\mathcal{B}}(x)
      = \sum_{R\in \mathcal{S}(P)} v_{\mathrm{ord}(R)}^k \cha_R(x)
      = \sum_{R\in \mathcal{S}(P)} w^k b_{\mathrm{ord}(R)} \cha_R(x)
      = a^{k}(x) w^k,
    \end{aligned}
  \end{equation}
  for a suitable scalar $a^k(x)$.  For fixed $1\leq j\leq r$, we obtain
  \begin{equation}\label{eq:69}
    \begin{aligned}
      \bigl((w^j)^t\cdot \varphi_P^{k,\mathcal{B}}(x)\bigr) w^j
      = a^{k}(x) \bigl( (w^j)^t\cdot w^k \bigr) w^j.
    \end{aligned}
  \end{equation}
  Exploiting orthonormality of $(w^k : 1\leq k\leq r)$ and appealing to~\eqref{eq:68} yields
  \begin{equation}\label{eq:70}
    \bigl(Q_j^{\mathcal{B}} \varphi_P^{k,\mathcal{B}}\bigr) (x)
    = \bigl((w^j)^t\cdot \varphi_P^{k,\mathcal{B}}(x)\bigr) w^j.
  \end{equation}
  By linearity, we obtain
  \begin{equation}\label{eq:210}
    \bigl(Q_j^{\mathcal{B}} f\bigr) (x)
    = \bigl((w^j)^t\cdot f(x)\bigr) w^j,
    \qquad x\in [0,1),
  \end{equation}
  for all $f\in Z_V$.  Clearly, the right hand side of the equation determines a bounded operator on $Z_V$, i.e., all $Q_j^{\mathcal{B}}$,
  $1\leq j\leq r$ are bounded.
\end{rem}

In \Cref{sec:hardy-iso} we will give a simpler characterization of Janson-Sobolev spaces with Property~IIA and Property~IIB in the case $X=L^1$.  If a
Janson-Sobolev space (i.e., $\boldsymbol{\sigma}$ is the constant sequence) has Property~IIB, then
$Z_V((\mathcal{F}_j),L^1,\boldsymbol{\sigma};\mathbb{R}^l)$ is isomorphic to $H^1$.  If it has Property~IIA then it is isomorphic to
$Z_{\operatorname{span}\{b\}}((\mathcal{F}_j),L^1,\boldsymbol{\sigma};\mathbb{R}^l)$.  Moreover, we show that Janson-Sobolev-Hardy spaces (i.e.,
$\boldsymbol{\sigma}$ is a sequence of independent Rademacher functions) are isomorphic to $H^1$.

\subsection{Diagonal operators}
\label{sec:mart-mult}

Let $D = (d_P^k : P\in \mathcal{A},\ 1\leq k\leq r)\colon Z_V\to Z_V$ denote a bounded \emph{diagonal operator (with respect to $(\varphi_P^k)$)},
i.e., $D\colon Z_V\to Z_V$ is given by
\begin{equation}\label{eq:219}
  D\bigl( \sum_{P\in \mathcal{A}} \sum_{k=1}^r a_P^k \varphi_P^k \bigr)
  = \sum_{P\in \mathcal{A}} \sum_{k=1}^r d_P^k a_P^k \varphi_P^k,
\end{equation}
where $d_P^k\in \mathbb{R}$, $1\leq k\leq r$.  Note that
\begin{equation}\label{eq:74}
  |d_P^k|
  \leq \|D\|,
  \qquad P\in \mathcal{A},\ 1\leq k\leq r.
\end{equation}

We relate to $D$ the \emph{associated adapted processes} $(A_n^{D,k})_{n=0}^{\infty}$, $1\leq k\leq r$ given by
\begin{equation}\label{eq:1}
  A_n^{D,k}(x)
  = d_P^k,
  \qquad x\in P\in \mathcal{AF}_n,\ n\in \mathbb{N}_0,\ 1\leq k\leq r.
\end{equation}
It follows from~\eqref{eq:74} that the sequence $(d_P^k : P \in \mathcal{A},\ 1 \leq k \leq r)$ is bounded. Consequently, for each $1 \leq k \leq r$,
the sequence $(A_n^{D,k})_{n=0}^{\infty}$ is bounded in $L^{\infty}$. Therefore, there exists a subsequence $(n_j)_{j=0}^{\infty}$ such that, for all
$1 \leq k \leq r$, we have
\begin{equation}\label{eq:2}
  A_{n_j}^{D,k} \to M^{D,k}
  \quad \text{in the weak$^{*}$ topology of $L^{\infty}$}.
\end{equation}
The pointwise multipliers $M^{D,k}$, $1\leq k\leq r$ give rise to the closed martingales $(M_n^{D,k})_{n=0}^{\infty}$, $1\leq k\leq r$ given by
\begin{equation}\label{eq:3}
  M_n^{D,k} := \cond(M^{D,k}|\mathcal{F}_n),
  \qquad n\in \mathbb{N}_0,\ 1\leq k\leq r.
\end{equation}
We record the fact that $M_n^{D,k} \to M^{D,k}$ almost everywhere and in $L^1$, for all $1\leq k\leq r$.

Now, we relate to the pointwise multiplier $M^{D,k}$, $1\leq k\leq r$ the diagonal operator
$\widetilde{D} = (\widetilde{d}_P^k : P\in \mathcal{A},\ 1\leq k\leq r)$ defined by
\begin{equation}\label{eq:4}
  \widetilde{d}_P^k
  := M_n^{D,k}(x),
  \qquad x\in P\in \mathcal{AF}_n,\ n\in \mathbb{N}_0,\ 1\leq k\leq r.
\end{equation}
Finally, we note that $A_n^{\widetilde{D},k} = M_n^{D,k}$, $n\in \mathbb{N}_0$, $1\leq k\leq r$.  In particular,
$(A_n^{\widetilde{D},k})_{n=0}^{\infty}$ is a closed martingale.

The bi-parameter dyadic counterparts of the following linear functionals $\lambda_{\mathcal{U}}^k$ were first introduced in~\cite{MR4816119}.
\begin{dfn}\label{dfn:functionals}
  Let $D = (d_P^k : P\in \mathcal{A},\ 1\leq k\leq r)\colon Z_V\to Z_V$ denote a bounded diagonal operator, and let $\mathcal{U}$ denote a non-principal ultrafilter on
  $\mathbb{N}_0$.  Then we define the linear functional $\lambda_{\mathcal{U}}^k$ on the space of bounded diagonal operators by
  \begin{equation}\label{eq:173}
    \lambda_{\mathcal{U}}^k(D)
    = \lim_{n\to \mathcal{U}} \int A_n^{D,k}
    = \lim_{n\to \mathcal{U}} \sum_{P\in \mathcal{AF}_n} |P| d_P^k,
    \qquad 1\leq k\leq r.
  \end{equation}
\end{dfn}

\begin{figure}
  \centering
  \begin{tikzpicture}
    \node[anchor=south west,inner sep=0] (img) {\includegraphics[width=.8\textwidth]{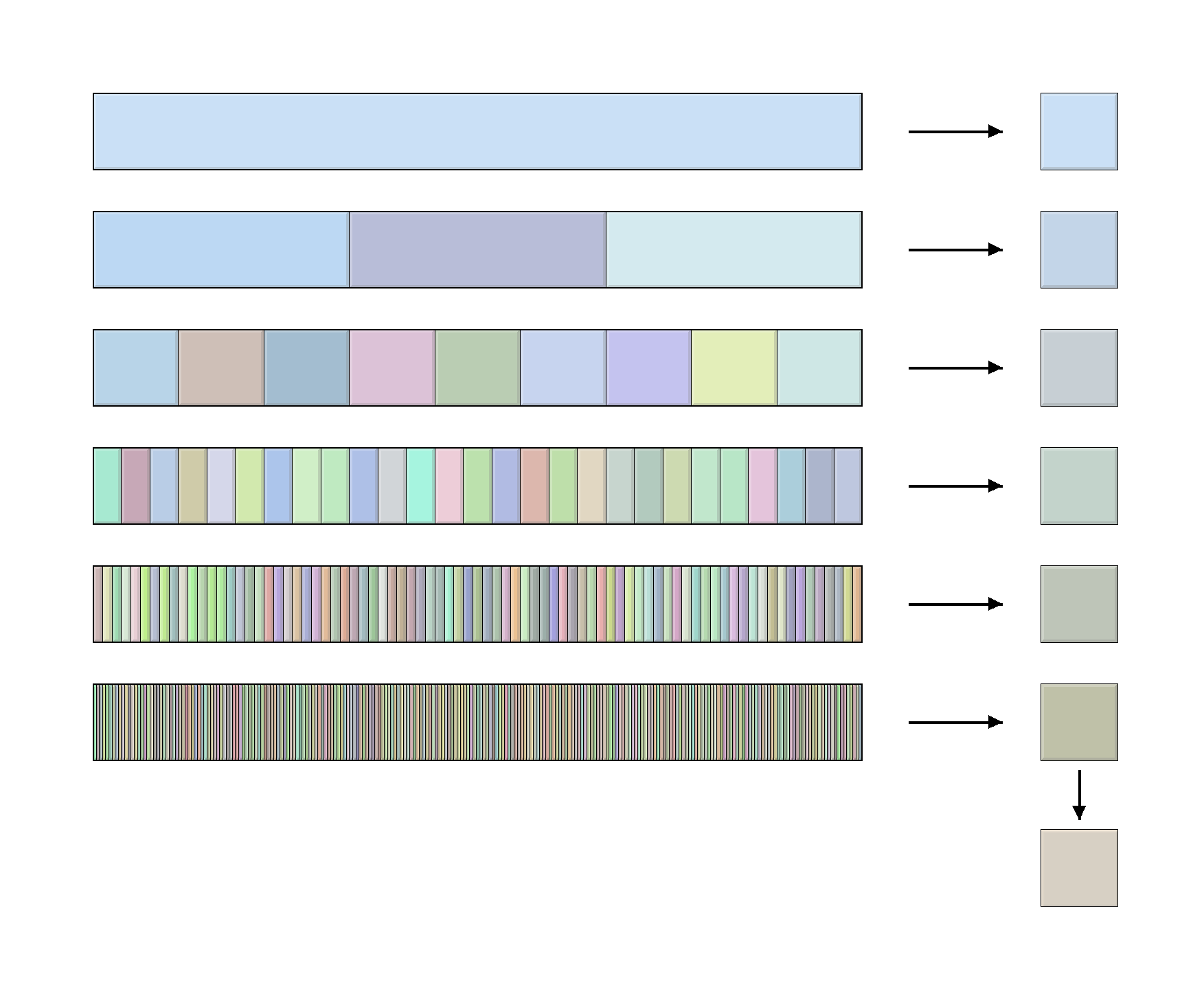}};
    \begin{scope}[x={(img.south east)},y={(img.north west)}]
      \node at (0.02,0.865) {{\small $\mathcal{AF}_0$}};%
      \node at (0.02,0.75) {{\small $\mathcal{AF}_1$}};%
      \node at (0.02,0.63) {{\small $\mathcal{AF}_2$}};%
      \node at (0.02,0.51) {{\small $\mathcal{AF}_3$}};%
      \node at (0.02,0.39) {{\small $\mathcal{AF}_4$}};%
      \node at (0.02,0.27) {{\small $\mathcal{AF}_5$}};%
      \node at (0.79,0.95) {{\small Avg}};%
      \node at (0.9875,0.2) {{\small $n\to \mathcal{U}$}};%
      \node at (0.9875,0.125) {{\small $\lambda_{\mathcal{U}}^k$}};%
    \end{scope}
  \end{tikzpicture}
  \caption{The colors on the triadic grid represent the different values of a diagonal operator $D = (d_P^k : P\in \mathcal{A},\ 1\leq k\leq r)$.  The squares to the
    right are the averages of the values $(d_P^k : P\in \mathcal{AF}_n)$ for each row $n$ for some fixed $1\leq k\leq r$, and
    $\lambda_{\mathcal{U}}^k$ is the ultrafilter limit of those averages.}
\end{figure}

\begin{rem}\label{rem:functionals}
  Note that if $(A_n^{D,k})_{n=0}^{\infty}$ converges pointwise to $M^{D,k}$ almost everywhere on $[0,1)$, then
  \begin{equation}\label{eq:191}
    \lambda_{\mathcal{U}}^k(D)
    = \lim_{n\to \infty} \int A_n^{D,k}
    = \int M^{D,k},
    \qquad 1\leq k\leq r.
  \end{equation}
\end{rem}

\begin{lem}\label{lem:functionals-classical-limit}
  Let $D = (d_P^k : P\in \mathcal{A},\ 1\leq k\leq r)\colon Z_V\to Z_V$ denote a bounded diagonal operator, and let $\mathcal{U}$ denote a non-principal ultrafilter on
  $\mathbb{N}_0$.  Then there exists an increasing sequence $(n_i)_{i=0}^{\infty}$ in $\mathbb{N}_0$ such that
  \begin{equation}\label{eq:66}
    \lambda_{\mathcal{U}}^k(D)
    = \lim_{i\to \infty} \int A_{n_i}^{D,k}
    = \lim_{i\to \infty} \sum_{P\in \mathcal{AF}_{n_i}} |P| d_P^k,
    \qquad 1\leq k\leq r.
  \end{equation}
\end{lem}

\begin{proof}
  By definition of the ultrafilter limit, for any $i\in \mathbb{N}_0$ and $1\leq k\leq r$ there exists $\Lambda_i^k\in \mathcal{U}$ such that
  $|\lambda_{\mathcal{U}}^k(D) - \int A_n^{D,k}|\leq \frac{1}{i}$ for all $n\in \Lambda_i^k$.  Since finite intersections of sets in the ultrafilter are infinite sets in
  $\mathcal{U}$, we can choose an increasing sequence $(n_i)_{i=0}^{\infty}$ in $\mathbb{N}_0$ such that $|\lambda_{\mathcal{U}}^k(D) - \int A_{n_i}^{D,k}|\leq \frac{1}{i}$, for all $i\in \mathbb{N}_0$.
\end{proof}

\subsection{Projectional reduction}
\label{sec:proj-reduct}

Here we provide the basic factorization tools that will be used throughout this work.

\begin{dfn}\label{dfn:factorization-modes}
  Given a Banach space $E$, bounded linear operators $S, T\colon E\to E$ and constants $C,\eta \ge 0$, we define the following:
  \begin{enumerate}[(a)]
  \item\label{enu:dfn:diagonal-operator} Given a Schauder basis $(e_n)_{n=1}^{\infty}$ of $E$ with biorthogonal functionals
    $(e_n^{*})_{n=1}^{\infty}$, we say that an operator $D\colon E\to E$ is \emph{diagonal (with respect to $(e_n)_{n=1}^{\infty}$)} if
    $\langle e_m^{*}, D e_n \rangle = 0$ whenever $n\neq m$.  We also say $D$ is a diagonal operator (with respect to $(e_n)_{n=1}^{\infty}$).
    
  \item\label{enu:dfn:factorization-modes:a} We say that \emph{$S$ factors through $T$ with constant $C$ and error $\eta$} if there exist linear
    operators $A,B\colon E\to E$ with $\|A\|\|B\|\leq C$ such that $\|S - ATB\|\leq \eta$.
    
  \item\label{enu:dfn:factorization-modes:a2} If~\eqref{enu:dfn:factorization-modes:a} holds and we additionally have $AB = \Id_E$, then we say that
    $S$ \emph{projectionally} factors through $T$ with constant $C$ and error $\eta$.
    
  \item\label{enu:dfn:factorization-modes:b} We say that $S$ (projectionally) factors through $T$ \emph{with constant $C^+$} and error $\eta$ if for
    every $\gamma > 0$, the operator $S$ (projectionally) factors through $T$ with constant $C + \gamma$ and error $\eta$.
    
  \item\label{enu:dfn:factorization-modes:e} Given $C \geq 1$, we say that $E$ has the \emph{$C^{+}$-diagonal primary factorization property} if there
    exists a Schauder basis $(e_n)_{n=1}^{\infty}$ of $E$ such that, for every diagonal operator $D$ with respect to $(e_n)_{n=1}^{\infty}$, the identity
    operator $\Id_E$ either factors through $D$ or $\Id_E - D$ with constant $C^+$.  Moreover, $E$ has the diagonal primary factorization property if
    it has the ${C'}^{+}$-diagonal primary factorization property for some $C' > 0$.
  \end{enumerate}
  If we omit the phrase ``with error $\eta$'' in~\eqref{enu:dfn:factorization-modes:a}, \eqref{enu:dfn:factorization-modes:a2}
  or~\eqref{enu:dfn:factorization-modes:b}, then we take that to mean that the error is~$0$.  Finally, if one of the properties holds for every error
  $\eta > 0$, then we prepend the modifier ``almost'', e.g., if for every $\eta > 0$ the operator $S$ projectionally factors through $T$ with constant
  $C$ and error $\eta$, we say that $S$ almost projectionally factors through $T$ with constant $C$.
\end{dfn}

Projectional factorization was explicitly introduced in~\cite{MR4430957}, where it was employed to circumvent the sometimes rather cumbersome criteria
for identifying large subspaces.  Indeed, the algebraic properties of projectional factorization allow the same factorization operators to be used to
simultaneously factor an operator $S$ through $D$ and $\Id-S$ through $\Id-D$.

\begin{rem}\label{rem:diagonal-primary-factorization-property}
  Note that the diagonal primary factorization property is an isomorphic invariant.  To see this, let $(e_n)_{n=1}^{\infty}$ denote a Schauder basis for
  the Banach space $E$, and suppose that $T\colon E\to F$ is an isomorphism.  Then clearly, $f_n = Te_n$, $n\in \mathbb{N}$ is a Schauder basis for
  $F$.  Now, let $D\colon F\to F$ denote any diagonal operator with respect to $(f_n)_{n=1}^{\infty}$.  Then it is simple to check that
  $\widetilde{D}\colon E\to E$ defined by $\widetilde{D} = T^{-1}DT$ is a diagonal operator with respect to $(e_n)_{n=1}^{\infty}$.  Thus, $\Id_E$ factors
  either through $\widetilde{D}$ or through $\Id_E - \widetilde{D}$, i.e.,
  \begin{equation}\label{eq:201}
    \Id_E = \widetilde{A}_1\widetilde{D}\widetilde{B}_1
    \qquad\text{or}\qquad \Id_E = \widetilde{A}_2(\Id_E - \widetilde{D})\widetilde{B}_2
  \end{equation}
  for some operators $\widetilde{A}_i, \widetilde{B}_i\colon E\to E$, $i=1,2$.  Consequently, putting
  \begin{equation}\label{eq:72}
    A_i
    = T\widetilde{A}_iT^{-1}
    \qquad{and}\qquad
    B_i
    = T\widetilde{B}_iT^{-1}
  \end{equation}
  for $i=1,2$ gives us
  \begin{equation}\label{eq:203}
    \Id_F = A_1DB_1
    \qquad\text{or}\qquad
    \Id_F = A_2(\Id_F - D)B_2.
  \end{equation}
\end{rem}

\begin{rem}\label{rem:factorization-properties}
  Clearly, the primary factorization property implies the diagonal primary factorization property whenever the underlying Banach space has a Schauder
  basis.  Moreover, under the usual self-similarity assumptions, Pe{\l}czyński's decomposition technique can be used to deduce primarity from the
  primary factorization property.

  The Janson-Sobolev space we will construct later in \Cref{thm:non-primary-janson-sobolev} is neither primary nor does it have the primary
  factorization property.  Nevertheless, it has Property~I and hence, by \Cref{mainthm:main-prop-I}, possesses the diagonal primary factorization
  property.  Finally, we note that any primary Janson-Sobolev or Janson-Sobolev-Hardy space necessarily has the diagonal primary factorization
  property.

  This follows from the more general result \Cref{pro:factorization-properties} below, together with \Cref{mainthm:picewise-constant}.
\end{rem}

\begin{pro}\label{pro:factorization-properties}
  Suppose that $(e_n)_{n=1}^{\infty}$ is a Schauder basis for the primary Banach space $E$ such that for every bounded diagonal operator
  $D\colon E\to E$ with respect to $(e_n)_{n=1}^{\infty}$ there exist scalars $\lambda_1, \dots, \lambda_n$ and bounded projections $Q_1, \dots, Q_n\colon E\to E$ with
  \begin{equation}\label{eq:11}
    \sum_{j=1}^n Q_j = \Id_E
    \qquad\text{and}\qquad
    Q_iQ_j = 0,\ i\neq j
  \end{equation}
  such that $\sum_{j=1}^n \lambda_j Q_j$ almost projectionally factors through $D$.  Then $E$ has the diagonal primary factorization property.
\end{pro}

\begin{proof}
  In order to verify this claim, we need to show that for a given bounded diagonal operator $D\colon E\to E$, the identity $\Id_E$ either factors
  through $D$ or $\Id_E - D$.  First, we pick $\lambda_1, \dots, \lambda_n$ and the bounded projections $Q_1, \dots, Q_n$ as in~\eqref{eq:11}.  By the primarity
  of $E$ and since $\sum_{j=1}^n Q_j = \Id_E$, we know that at least one of the $Q_j(E)$, $1\leq j\leq n$ is isomorphic to $E$ itself.  We assume without
  restriction that $Q_1(E)$ is $C$-isomorphic to $E$ for some $C > 0$, i.e., there exists an isomorphism $T\colon E\to Q_1(E)$ with
  \begin{equation}\label{eq:76}
    \|T\|\cdot \|T^{-1}\|
    \leq C.
  \end{equation}

  Next, we put $\eta = \frac{1}{4C\|Q_1\|}$ and observe that since $\sum_{j=1}^n \lambda_j Q_j$ almost projectionally factors through $D$, we can find bounded
  operators $A,B\colon E\to E$ with $AB = \Id_E$ such that
  \begin{equation}\label{eq:6}
    \bigl\| ADB - \sum_{j=1}^n \lambda_j Q_j \bigr\|
    < \eta
    \quad\text{and}\quad
    \bigl\| A(\Id_E - D)B - \sum_{j=1}^n (1 - \lambda_j) Q_j \bigr\|
    < \eta.
  \end{equation}
  Note that either $|\lambda_1| \geq 1/2$ or $|1 - \lambda_1| \geq 1/2$.  We assume without restriction $|\lambda_1| \geq 1/2$.  If not, we would use
  $\Id_E - D$ instead of $D$.

  Now, using~\eqref{eq:6}, gives us $\bigl\| Q_1ADB - \lambda_1 Q_1 \bigr\| < \eta \|Q_1\|$, and thus
  \begin{equation}\label{eq:40}
    \|Q_1 - \lambda_1^{-1}Q_1ADB\|
    < 2\eta \|Q_1\|.
  \end{equation}
  Combining $T^{-1}Q_1T = \Id_E$ with \eqref{eq:76} and~\eqref{eq:40} yields
  \begin{equation}\label{eq:73}
    \|\Id_E - T^{-1}\lambda_1^{-1}Q_1ADBT\|
    < 2C\eta \|Q_1\| = 1/2,
  \end{equation}
  i.e., $\Id_E$ factors through $D$ with error $1/2$.  Thus, $T^{-1}\lambda_1^{-1}Q_1ADBT$ is invertible, and hence, $\Id_E$ factors through $D$.
\end{proof}

%

\section{Main results}
\label{sec:main-results}

Over the years, numerous strategies and frameworks have been developed to investigate the factorization properties of classes of bounded operators
acting on a Banach space $E$, with the aim of obtaining structural results for operator ideals and complemented subspaces.  The factorization problems
considered in this work concern the question of whether the identity operator can be factored through certain classes of bounded operators.  A
particularly important subclass of bounded operators in this context consists of diagonal operators with respect to a Schauder basis of $E$.  Most
closely related to the present work is~\cite{MR4816119}, which studies multipliers on the biparameter Haar system acting on rearrangement invariant
spaces and their square function analogues, known as Haar system Hardy spaces.

The main results in~\cite{MR4816119} show that every Haar multiplier on a Haar system Hardy space is a projectional factor of a linear combination of
the Capon projection $\mathcal{C}$ and $\Id - \mathcal{C}$.  The Capon projection is the natural projection onto the lower triangular part of the biparameter Haar system,
and the coefficients of the linear combination are predetermined and can be computed via bounded linear functionals from the given diagonal operator.
If the Capon projection is unbounded, then the coefficients in the linear combination must coincide.  Consequently, the identity operator on a Haar
system Hardy space factors through either a given Haar multiplier $D$ or $\Id - D$.  In particular, this establishes that Haar system Hardy spaces
possess the diagonal primary factorization property, which is an isomorphic invariant introduced here in \Cref{sec:proj-reduct}.

Our main results establish factorization properties for operators acting on Janson-Sobolev and Janson-Sobolev-Hardy spaces $Z_V$ that are diagonal
with respect to their natural Schauder bases $(\varphi_P^{k,\mathcal{B}} : P\in\mathcal{A},\ 1\leq k\leq r)$.  A central role in our analysis is played by the canonical projections
$Q_{\mathcal{K}}^{\mathcal{B}}$, $\mathcal{K}\subset\{1,\dots,r\}$, which map a linear combination of the Schauder basis
$(\varphi_P^{k,\mathcal{B}} : P\in\mathcal{A},\ 1\leq k\leq r)$ onto the components corresponding to the indices $k\in\mathcal{K}$.  These projections play a role analogous to that of the Capon
projection $\mathcal{C}$ in~\cite{MR4816119}.  This is reflected in \Cref{mainthm:picewise-constant}, where it is shown that diagonal operators $D$ on
$Z_V$ are almost a projectional factor of a linear combination of canonical projections.  The coefficients of this linear combination are likewise
obtained from the diagonal operator via bounded linear functionals.

The occurring canonical projections are determined by the finest bounded partition $\mathcal{K}_t$, $1\leq t\leq s$ of $\{1, \dots, r\}$, that is the finest partition
such that every $Q_{\mathcal{K}_t}^{\mathcal{B}}$, $1\leq t\leq s$ is bounded.  Analogous to the results for the unbounded Capon projection above, if none of the nontrivial
canonical projections $Q_{\mathcal{K}}^{\mathcal{B}}$, $\emptyset\neq \mathcal{K}\subsetneq \{1, \dots, r\}$ is bounded (i.e., $Z_V$ has Property~I), then the coefficients of the linear combination
are forced to coincide.  This rigidity is instrumental in the proof of \Cref{mainthm:main-prop-I}.

From the perspective of factorization, the Capon projection in~\cite{MR4816119} and the canonical projections considered in the present work serve as
fundamental building blocks for diagonal operators.  Consequently, their behavior, as characterized by Properties~I, IIA, and IIB (see
\Cref{dfn:space-type}), plays a decisive role in determining the factorization properties of the entire class of diagonal operators.  Specializing to
the case $Z_V(L^1)$, we establish the diagonal primary factorization property in \Cref{mainthm:main-L1}.  Moreover, if $Z_V(L^1)$ has Property~IIB,
then it is isomorphic to $H^1$ and hence even possesses the primary factorization property.

We will now state our first main theorem and describe the ideas for its proof in more detail.
\begin{mainthm}\label{mainthm:picewise-constant}
  Let $Z_V = Z_V((\mathcal{F}_j),X,\boldsymbol{\sigma};\mathbb{R}^l)$ denote a Janson-Sobolev space or a Janson-Sobolev-Hardy space, and let
  $\mathcal{K}_1, \dots, \mathcal{K}_s\neq \emptyset$ denote the finest bounded partition of $\{1, \dots, r\}$ for the orthonormal basis $\mathcal{B}$ of
  $V$ with respect to $Z_V$.  Suppose that $D\colon Z_V\to Z_V$ is a bounded diagonal operator with respect to
  $(\varphi_P^{k,\mathcal{B}} : P\in \mathcal{A},\ 1\leq k\leq r)$, let $k_t\in \mathcal{K}_t$, $1\leq t\leq s$ and pick any non-principal ultrafilter $\mathcal{U}$ on $\mathbb{N}_0$.  Then $S\colon Z_V\to Z_V$ given by
  \begin{equation}\label{eq:21}
    S
    = \sum_{t=1}^{s} \lambda_{\mathcal{U}}^{k_t}(D) Q_{\mathcal{K}_t}^{\mathcal{B}}
  \end{equation}
  almost projectionally factors through $D$ with constant~$1^{+}$.
\end{mainthm}

The proof proceeds through a systematic, stepwise reduction of the factorization problem using projectional factorization operators
$A_1,B_1,\dots,A_n,B_n$, each induced by a carefully constructed block basis of $(\varphi_P^{k,\mathcal{B}} : P\in\mathcal{A},\ 1\leq k\leq r)$.  In the first reduction step, we obtain
a new diagonal operator $D_1$ by setting $D_1=A_1DB_1$.  We then proceed inductively, defining $D_j=A_jD_{j-1}B_j$ for $2\leq j\leq n$, leading to
$D_n=S$ as in~\eqref{eq:21}.  We refer to the procedure by which this reduction is achieved as \emph{stabilization}.  It is described in detail in
\Cref{sec:stabilization}.

For a general bounded diagonal operator $D$, the associated adapted processes $(A_n^{D,k})_{n=0}^{\infty}$, $1\leq k\leq r$, are bounded but need not possess any
useful form of convergence, such as pointwise convergence.  In previous works, stabilization has been achieved through a variety of methods.  For
example, in his work on the James tree space~\cite{MR709696}, A.~D.~Andrew passes to a subtree on which a given bounded function converges along every
branch.  More recently, in~\cite{MR4430957,MR4816119,MR4894823,MR4884827}, a stochastic selection procedure was introduced that modulates
Gamlen-Gaudet block bases by random signs, balancing the two sides of the dyadic tree at each node and thereby ensuring convergence along every
branch.  This convergence is uniform except, possibly, at points belonging to the dyadic grid.

This modulation of the Haar functions by signs can be viewed as a local rearrangement: the sets on which a Haar function takes the values $+1$ and
$-1$ are swapped.  In the present setting, however, we must account not only for vector-valued $q$-adic martingale difference sequences in place of
scalar-valued Haar functions, but also for the additional subspace restriction imposed by $V$.  Consequently, a randomization procedure based on
rearranging the vector-valued martingale difference sequences would necessarily restrict the class of admissible subspaces $V$.  Achieving
stabilization in full generality therefore requires a fundamentally different approach.

Rather than balancing the nodes of the tree, our approach forces pointwise convergence by splitting off a martingale.  Since the processes
$(A_n^{D,k})_{n=0}^{\infty}$, $1\leq k\leq r$, are bounded in $L^\infty$, we can extract weak$^*$ convergent subsequences with limits $M^{D,k}$,
$1\leq k\leq r$, in $L^\infty$.  Each of these limits induces a closed martingale $(M_n^{D,k})_{n=0}^{\infty}$ converging pointwise to $M^{D,k}$.  We then apply
Egorov's theorem to improve pointwise convergence to uniform convergence, at the cost of losing faithfulness of the resulting block bases.  Details of
the argument are given in \Cref{sec:stab-via-mart}.

Once convergence has been established, the crucial \Cref{thm:agreement} shows that the limits $M^{D,k}$, $1\leq k\leq r$ agree almost everywhere on indices
determined by the finest bounded partition.  This is analogous to the corresponding results for the Capon projection obtained in~\cite{MR4816119} (see
Theorems~2.32, 2.38, and 2.39).  The proof of \Cref{thm:agreement} relies essentially on the Vandermonde technique, which we develop in
\Cref{sec:vandermonde-approach} and apply to the following setting.  Given bounded diagonal operators $\widehat{D}$ and $D_t$, $1\leq t\leq s$, such that
$\widehat{D} = \sum_{t=1}^s D_tQ_{\mathcal{K}_t}$ and $D_{t_0}-D_{t_1}$ is invertible for all $t_0\neq t_1$, we show that all canonical projections
$Q_{\mathcal{K}_t}$, $1\leq t\leq s$ are necessarily bounded.  The final step in the proof of \Cref{mainthm:picewise-constant} is to translate the agreement of the
limits $M^{D,k}$, $1\leq k\leq r$ on the finest bounded partition back to the corresponding diagonal operators.  We refer to the first part of
\Cref{sec:reduct-canon-proj} for details.

Finally, we emphasize that the reduction of an arbitrary bounded diagonal operator $D$ to the operator $S$ given by~\eqref{eq:21} is achieved in such a way that the coefficients in the resulting linear combination of canonical projections are determined a priori by the linear functionals $\lambda_{\mathcal{U}}^{k_t}(D)$.
If the space $Z_V$ in \Cref{mainthm:picewise-constant} has Property~I, then all coefficient functionals $\lambda_{\mathcal{U}}^k(D)$, $1\leq k\leq r$, must agree. This
agreement is the key ingredient in the proof of our second main theorem, stated below.
\begin{mainthm}\label{mainthm:main-prop-I}
  Every Janson-Sobolev space or Janson-Sobolev-Hardy space $Z_V$ with Property~I has the $2^+$-diagonal primary factorization property.
\end{mainthm}

\Cref{mainthm:main-prop-I} concerns all Janson-Sobolev and Janson-Sobolev-Hardy spaces equipped with a general rearrangement invariant norm and having
Property~I.  More can be said about the remaining cases, Properties~IIA and IIB, when the underlying rearrangement invariant norm is the $L^1$-norm.
In this case, we can exploit the isomorphic classification of such spaces established in \Cref{sec:hardy-iso}.
\begin{mainthm}\label{mainthm:main-L1}
  Let $Z_V(L^1) = Z_V((\mathcal{F}_j),L^1,\boldsymbol{\sigma};\mathbb{R}^l)$ denote a Janson-Sobolev space or a Janson-Sobolev-Hardy space.  Then
  $Z_V(L^1)$ has the diagonal primary factorization property.  Moreover, if $Z_V(L^1)$ has Property~IIB, then $Z_V(L^1)$ is isomorphic to $H^1$ and
  therefore has the primary factorization property.
\end{mainthm}

In \Cref{sec:an-example-non}, we construct a Janson-Sobolev space $Z_V(L^1)$ with Property~I that is neither primary nor has the primary factorization
property.  This shows that, in general, the diagonal primary factorization property in \Cref{mainthm:main-L1} cannot be upgraded to either the primary
factorization property or primarity.  This result suggests a broader question.
\begin{question}\label{que:diagonal-primary-factorization}
  Which Banach spaces with a Schauder basis have the diagonal primary factorization property?
\end{question}

%

\section{Schauder basis}
\label{sec:schauder-basis}

We establish quantitative estimates for the basis constant of the Schauder basis $(\varphi_P^k : P\in \mathcal{A},\ 1\leq k\leq r)$ for Janson-Sobolev and
Janson-Sobolev-Hardy spaces.

\begin{pro}\label{pro:schauder-Z}
  The system $(\varphi_P^k : P\in \mathcal{A},\ 1\leq k\leq r)$ ordered with respect to the linear order on $\mathcal{A}\times \{1, \dots, r\}$ defined in \Cref{sec:basic-definitions} is a
  Schauder basis in $Z_V(X) = Z_V((\mathcal{F}_j),X,\boldsymbol{\sigma};\mathbb{R}^l)$ with basis constant $\lambda_{Z_V(X)}\leq 1+2\sqrt{q}$.
\end{pro}

\begin{myproof}[Proof of \Cref{pro:schauder-Z}]
  \begin{proofstep}
    Let $f = \sum_{P\in \mathcal{A}} \sum_{k=1}^r a_P^k \varphi_P^k$ denote a finite linear combination of
    $(\varphi_P^k : P\in \mathcal{A},\ 1\leq k\leq r)$ and let $(P_0,k_0)\in \mathcal{A}\times\{1, \dots, r\}$.  Now define $\pi\colon Z_V(X)\to Z_V(X)$ by
    \begin{equation*}
      \pi(f)
      = \sum_{(P,k)\leq (P_0,k_0)} a_P^k \varphi_P^k.
    \end{equation*}
    We need to show that $\|\pi f\|_{Z_V(X)}\leq (1 + 2\sqrt{q})\cdot \|f\|_{Z_V(X)}$.

    To this end, we define
    \begin{equation*}
      g_1 = \sum_{P< P_0} \sum_{k=1}^r a_P^k \varphi_P^k
      \qquad\text{and}\qquad
      g_2 = \sum_{k=1}^{k_0} a_{P_0}^k \varphi_{P_0}^k,
    \end{equation*}
    and let $\mathcal{G}$ denote the $\sigma$-algebra generated by $\{P\in \mathcal{A} : P < P_0\}$.  Note that by \Cref{rem:conditional}, we have
    $\|g_1\|_{Z_V(X)}\leq \|f\|_{Z_V(X)}$, and hence
    \begin{equation}\label{eq:117}
      \|\pi f\|_{Z_V(X)}
      = \|g_1 + g_2\|_{Z_V(X)}
      \leq \|f\|_{Z_V(X)} + \|g_2\|_{Z_V(X)}.
    \end{equation}
    Below, we will prove that
    \begin{equation}\label{eq:104}
      \|g_2\|_{Z_V(X)}
      = \Bigl\| \sum_{k=1}^{k_0} a_{P_0}^k \varphi_{P_0}^k \Bigr\|_{Z_V(X)}
      \leq \sqrt{q} \Bigl\| \sum_{k=1}^r a_{P_0}^k \varphi_{P_0}^k \Bigr\|_{Z_V(X)}
      =: \sqrt{q} \|G_2\|_{Z_V(X)},
    \end{equation}
    Assuming for the moment that~\eqref{eq:104} is true, we can complete the proof as follows: Let $\mathcal{H}$ denote the $\sigma$-algebra generated by
    $\{P\in \mathcal{A} : P \leq P_0\}$ and note that $G_2 = \cond(f | \mathcal{H}) - \cond(f | \mathcal{G})$.  We conclude the proof appealing to \Cref{rem:conditional} and
    combining~\eqref{eq:117} and \eqref{eq:104} to obtain
    \begin{equation*}
      \|\pi f\|_{Z_V(X)}
      \leq (1 + 2\sqrt{q})\|f\|_{Z_V(X)}.
    \end{equation*}
  \end{proofstep}
  
  \begin{proofstep}
    Here, we prove that~\eqref{eq:104} is true, and thereby we complete the proof of \Cref{pro:schauder-Z}.  Since~\eqref{eq:104} concerns only
    functions supported on a single atom, we observe that by~\eqref{eq:39} for any sequence $(b_k)_{k=1}^r$ of vectors in $\mathbb{R}^l$ and
    $P_0\in\mathcal{AF}_s$, $s\in\mathbb{N}$ we have
    \begin{equation}\label{eq: anysigmanormonA}
      \Bigl\|\sum_{k=1}^{r} b_k \varphi_{P_0}^k\Bigr\|_{Z_V}
      =\Bigl\|
      \bigl\|\sum_{k=1}^r b_k \varphi_{P_0}^k(x)
      \bigr\|_{\ell^2_l}
      \Bigr\|_X,
    \end{equation}
    for any choice of $\boldsymbol{\sigma}$.  There are q-tuples $(u_i)_{i=1}^q$ and $(U_i)_{i=1}^q$ of vectors in $\mathbb{R}^l$ such that
    \begin{equation}\label{eq: characteristicfuncg2G2}
      g_2 = \sum_{R\in\mathcal{S}(P_0)} u_{\operatorname{ord}(R)} \cha_R
      \quad\mbox{and}\quad 
      G_2 = \sum_{R\in\mathcal{S}(P_0)} U_{\operatorname{ord}(R)} \cha_R.
    \end{equation}
    Note that by~\eqref{eq: anysigmanormonA} we have
    \begin{equation}\label{eq: onlymodcounts}
      \|g_2\|_{Z_V}
      = \Bigl\|
      \sum_{R\in\mathcal{S}(P_0)} \|u_{\operatorname{ord}(R)}\|_{\ell^2_l} \cha_R
      \Bigr\|_X
      \quad\mbox{and}\quad
      \|G_2\|_{Z_V}
      = \Bigl\|
      \sum_{R\in\mathcal{S}(P_0)} \|U_{\operatorname{ord}(R)}\|_{\ell^2_l} \cha_R
      \Bigr\|_X.
    \end{equation}

    In view of~\eqref{eq: onlymodcounts}, it is natural to consider the norm $\rho\colon \mathbb{R}^q \rightarrow [0,\infty)$ given by
    \begin{equation*}
      \rho\bigl((z_i)_{i=1}^{q}\bigr)
      =  \Bigl\|
      \sum_{R\in\mathcal{S}(P_0)} |z_{\operatorname{ord}(R)}| \cha_R
      \Bigr\|_X.
    \end{equation*}
    Clearly, by the above definition and~\eqref{eq: onlymodcounts}, we have
    \begin{equation}\label{eq: norminothernorm}
      \|g_2\|_{Z_V}
      = \rho\bigl((\|u_i\|_{\ell^2_l})_{i=1}^q\bigr)
      \quad\mbox{and}\quad
      \|G_2\|_{Z_V}
      = \rho\bigl((\|U_i\|_{\ell^2_l})_{i=1}^q\bigr)
    \end{equation}     

    Since $X$ is a rearrangement invariant space and all the $R\in \mathcal{S}(P_0)$ have the same measure, it follows that $\rho$ is $1$-symmetric, i.e.,
    \begin{equation*}
      \rho\bigl((z_i)_{i=1}^{q}\bigr)
      = \rho\bigl((\varepsilon_i z_{\tau(i)})_{i=1}^{q}\bigr),
    \end{equation*}
    for every permutation $\tau$ of the set $\{1,2,\ldots, q\}$ and every choice of signs $\varepsilon_i\in\{-1,1\}$, $1\leq i\leq q$.
    By~\cite[Proposition~4.8, Exercise~4.24]{AliceBobBanach} the John ellipsoid is a Euclidean ball.  Therefore, by a classical theorem of
    F.~John~\cite{JohnEl} there exists a constant $C>0$ such that
    \begin{equation}\label{eq: johnel}
      C\|(z_i)_{i=1}^{q}\|_{\ell^2_q}
      \leq \rho\bigl((z_i)_{i=1}^{q}\bigr)
      \leq C\sqrt{q}\|(z_i)_{i=1}^{q}\|_{\ell^2_q}.
    \end{equation}
    Recall that $|P_0| = q |R|$ for every $R\in \mathcal{S}(P_0)$.  Hence, by~\eqref{eq: johnel} and~\eqref{eq: characteristicfuncg2G2}, we obtain
    \begin{equation}\label{eq:115}
      \begin{split}
        \rho\bigl((\|u_i\|_{\ell^2_l})_{i=1}^q\bigr)
        &\leq C\sqrt{q}\cdot \|(\|u_i\|_{\ell^2_l})_{i=1}^q\|_{\ell^2_q}
          = C \frac{q}{\sqrt{|P_0|}}\cdot \sqrt{\sum_{i=1}^q\frac{|P_0|}{q}\|u_i\|^2_{\ell^2_l} }\\
        &= C \frac{q}{\sqrt{|P_0|}}\cdot \sqrt{\sum_{R\in\mathcal{S}(P_0)} |R|\, \|u_{\ord(R)}\|^2_{\ell^2_l}}
          = C \frac{q}{\sqrt{|P_0|}}\cdot \|g_2\|_{L^2}.
      \end{split}
    \end{equation}
    The orthogonality of the functions $\varphi^k_{P_0}$, $1\leq k\leq r$ in $L^2$ (see \Cref{sec:basic-definitions}) yields
    \begin{equation*}
      \|g_2\|_{L^2}
      = \Bigl\|\sum_{k=1}^{k_0} a_{P_0}^k \varphi_{P_0}^k\Bigr\|_{L^2}
      \leq \Bigl\|\sum_{k=1}^{r} a_{P_0}^k \varphi_{P_0}^k\Bigr\|_{L^2}
      = \|G_2\|_{L^2}.
    \end{equation*}
    Inserting the above estimate into~\eqref{eq:115} and using~\eqref{eq: characteristicfuncg2G2} gives us
    \begin{equation}\label{eq: estP01}
      \begin{split}
        \rho\bigl((\|u_i\|_{\ell^2_l})_{i=1}^q\bigr)
        &\leq C \frac{q}{\sqrt{|P_0|}}\cdot \|G_2\|_{L^2}
          = C \frac{q}{\sqrt{|P_0|}}\cdot \sqrt{\sum_{R\in\mathcal{S}(P_0)} |R|\, \|U_{\ord(R)}\|^2_{\ell^2_l}}\\
        &= C\sqrt{q} \|(\|U_i\|_{\ell^2_l})_{i=1}^q\|_{\ell^2_q}.
      \end{split}
    \end{equation}
    By~\eqref{eq: norminothernorm}, \eqref{eq: estP01} and \eqref{eq: johnel}, we have
    \begin{align}\label{testxyz}
      \|g_2\|_{Z_V}
      &= \rho\bigl((\|u_i\|_{\ell^2_l})_{i=1}^q\bigr)
        \leq C\sqrt{q} \|(\|U_i\|_{\ell^2_l})_{i=1}^q\|_{\ell^2_q}\\
      &\leq \sqrt{q} \rho\bigl((\|U_i\|_{\ell^2_l})_{i=1}^q\bigr)
        = \sqrt{q} \|G_2\|_{Z_V},
    \end{align}
    which concludes the proof of \Cref{pro:schauder-Z}.\qedhere
  \end{proofstep}
\end{myproof}

\begin{lem}\label{lem:schauder}
  Let $D = (d_P^k : P\in \mathcal{A},\ 1\leq k\leq r)$ denote a diagonal operator with respect to $(\varphi_P^k : P\in \mathcal{A},\ 1\leq k\leq r)$.  Then
  \begin{equation*}
    \|D\colon Z_V\to Z_V\|
    \leq 2 \lambda_{Z_V} \sum_{P\in \mathcal{A}} \sum_{k=1}^{r} |d_P^k|,
  \end{equation*}
  where $\lambda_{Z_V}$ is the basis constant of $(\varphi_P^k : P\in \mathcal{A},\ 1\leq k\leq r)$ in $Z_V$.
\end{lem}

\begin{proof}
  Let $f = \sum_{P\in \mathcal{A}} \sum_{k=1}^r a_P^k \varphi_P^k$ be a finite sum. Then
  \begin{equation*}
    \|Df\|_{Z_V}
    \leq \sum_{P\in \mathcal{A}} \sum_{k=1}^r |d_P^k|\cdot \|a_P^k\varphi_P^k\|_{Z_V}
    \leq \sum_{P\in \mathcal{A}} \sum_{k=1}^r |d_P^k|\cdot 2\lambda_{Z_V} \|f\|_{Z_V}
  \end{equation*}
  as claimed.
\end{proof}

%
\section{A three-way classification for $Z_V$ spaces}
\label{sec:three-way-class}

In this section we show that there are in fact only two possibilities:
\begin{itemize}
\item we can choose a basis $(\varphi_P^k : P\in \mathcal{A},\ 1\leq k\leq r)$ such that the only bounded canonical projection
  $Q_{\mathcal{K}}\colon Z_V\to Z_V$ on this basis is $Q_{\mathcal{K}} = \Id_{Z_V}$, or
\item all canonical projections on all bases are bounded.
\end{itemize}

Throughout this section, we will make use of different bases of $Z_V$ by choosing different bases for $V$.  Recall that for any orthonormal basis
$\mathcal{B}=(v^1,\ldots, v^r)$ of $V$ we denote a corresponding basis of $Z_V$ by $(\varphi_{P}^{k,\mathcal{B}} : P\in\mathcal{A},\ 1\leq k\leq r)$, and the associated family of canonical
projections by $Q_{\mathcal{K}}^{\mathcal{B}}$, $\mathcal{K}\subset\{1, \dots, r\}$ (we refer to~\eqref{eq:119} and \Cref{dfn:Q}).

We are now ready to state and prove the classification result.
\begin{thm}\label{thm:trichotomy}
  Let $Z_V$ be a Janson-Sobolev or Janson-Sobolev-Hardy space.  Then $Z_V$ either has Property~I, Property~IIA or Property~IIB.
\end{thm}

\begin{myproof}
  First, we assume that the space has neither Property~IIA nor Property~IIB, i.e., there exists an orthonormal basis
  $\mathcal{B}_0=\{v^{1,0},\ldots, v^{r,0}\}$ of $V$ such that $Q_1^{\mathcal{B}_0}$ is unbounded.  Let $\mathcal{K}_1^0,\ldots \mathcal{K}^0_{d_0}$ be the finest bounded partition of
  $\{1,\ldots,r\}$ for the basis $\mathcal{B}_0$ (see \Cref{dfn:finest-bounded-partition}).  For simplicity of notation we assume that
  $1\in\mathcal{K}^0_{1}$.  Since $Q_1^{\mathcal{B}_0}$ is unbounded
  \begin{equation*}
    m
    := \#\mathcal{K}_1^0
    \geq 2.
  \end{equation*}
  By relabeling we can assume that $\mathcal{K}_1^0=\{1,\ldots,m\}$ and $m+1\in\mathcal{K}_2^0$.  We now define a new basis $\mathcal{B}_1=(v^{1,1},\ldots, v^{r,1})$ given by the formula:
  \begin{equation}\label{eq:75}
    v^{k,1}
    =\left\{
      \begin{array}{cl} v^{k,0}&\quad\mbox{for }k\in\{1,\ldots, r\}\setminus\{m,m+1\},\\
        \frac{v^{m,0}+v^{m+1,0}}{\sqrt{2}}&\quad\mbox{for }k=m+1,\\
        \frac{v^{m,0}-v^{m+1,0}}{\sqrt{2}}&\quad\mbox{for }k=m.
      \end{array}
    \right.
  \end{equation}
  Let $\mathcal{K}_1^1,\ldots \mathcal{K}^1_{d_1}$ be the finest bounded partition of $\{1,\ldots,r\}$ for the basis $\mathcal{B}_1$.  For simplicity of notation we again assume that
  $1\in\mathcal{K}_1^1$.  We will prove that $d_1=d_0-1$ and $\mathcal{K}_1^1=\mathcal{K}_1^0\cup\mathcal{K}_2^0$.  First we observe that
  \begin{equation*}
    \varphi_{P}^{k,\mathcal{B}_0}
    = \varphi_{P}^{k,\mathcal{B}_1},
    \qquad k\notin\{m,m+1\},
  \end{equation*}
  and therefore, we obtain
  \begin{equation}\label{eq:eqlQproj}
    Q_k^{\mathcal{B}_0}
    = Q_k^{\mathcal{B}_1},
    \qquad k\notin\{m,m+1\}.
  \end{equation}

  We now examine the behavior of $Q_k^{\mathcal{B}_0}$ and $Q_k^{\mathcal{B}_1}$ for $k\in \mathcal{K}_j^0$, $j > 2$.  Let $j > 2$, and observe that since
  $\{m,m+1\}\cap \mathcal{K}_j^0=\emptyset$, we have for any $1\leq i\leq d_1$ that
  \begin{equation}\label{eq:intersoldgen}
    \left(\mathcal{K}_i^1\setminus\{m,m+1\}\right)\cap \mathcal{K}_j^0=\mathcal{K}_i^1\cap \mathcal{K}_j^0
  \end{equation} 
  and by~\eqref{eq:75}
  \begin{equation}\label{eq:orthooldm0}
    Q_{\mathcal{K}_j^0}^{\mathcal{B}_0} Q_{m}^{\mathcal{B}_1}=Q_{\mathcal{K}_j^0}^{\mathcal{B}_0} Q_{m+1}^{\mathcal{B}_1}=0.
  \end{equation}
  Recall that the $Q_k^{\mathcal{B}_0}$, $1\leq k\leq r$ are orthogonal projections.  Thus, for any $1\leq i\leq d_1$ and $j>2$ we obtain by~\eqref{eq:orthooldm0},
  \eqref{eq:eqlQproj} and~\eqref{eq:intersoldgen} that
  \begin{equation}\label{eq:othepartitioneqb0}
    \begin{split}
      Q_{\mathcal{K}_j^0}^{\mathcal{B}_0} Q_{\mathcal{K}_i^1}^{\mathcal{B}_1}
      &= Q_{\mathcal{K}_j^0}^{\mathcal{B}_0}\sum_{l\in \mathcal{K}_i^1} Q_{l}^{\mathcal{B}_1}
        = \sum_{l\in \mathcal{K}_i^1} Q_{\mathcal{K}_j^0}^{\mathcal{B}_0} Q_{l}^{\mathcal{B}_1}
        = \sum_{\substack{l\in \mathcal{K}_i^1\\l\notin \{m,m+1\}}} Q_{\mathcal{K}_j^0}^{\mathcal{B}_0} Q_{l}^{\mathcal{B}_1}\\
      &= \sum_{\substack{l\in \mathcal{K}_i^1\\l\notin \{m,m+1\}}} Q_{\mathcal{K}_j^0}^{\mathcal{B}_0} Q_{l}^{\mathcal{B}_0}
      = \sum_{\substack{l\in \mathcal{K}_i^1\cap \mathcal{K}_j^0\\l\notin \{m,m+1\}}} Q_{l}^{\mathcal{B}_0}
      = \sum_{l\in \mathcal{K}_i^1\cap \mathcal{K}_j^0 } Q_{l}^{\mathcal{B}_0}\\
      &= Q_{\mathcal{K}_i^1\cap \mathcal{K}_j^0}^{\mathcal{B}_0}.
    \end{split}
  \end{equation}  
  In particular, $Q_{\mathcal{K}_i^1\cap \mathcal{K}_j^0}^{\mathcal{B}_0}$ is bounded for $1\leq i\leq d_1$, $j>2$.

  Fix $j > 2$, and assume that
  \begin{equation}\label{eq:80}
    \mathcal{K}_i^1\cap \mathcal{K}_j^0\neq\emptyset
    \quad \text{for some $1\leq i\leq d_1$}.
  \end{equation}
  Since $ \mathcal{K}_1^0, \ldots, \mathcal{K}_{d_0}^0$ is the finest bounded partition of $\{1,\ldots,r\}$ for the basis $\mathcal{B}_0$ we have
  \begin{equation}\label{eq:81}
    \mathcal{K}_i^1\cap \mathcal{K}_j^0
    = \mathcal{K}_j^0.
  \end{equation}   
  On the other hand, \eqref{eq:othepartitioneqb0}, \eqref{eq:intersoldgen} and~\eqref{eq:eqlQproj} yield
  \begin{equation*}
    Q_{\mathcal{K}_j^0}^{\mathcal{B}_0} Q_{\mathcal{K}_i^1}^{\mathcal{B}_1}
    = Q_{\mathcal{K}_i^1\cap \mathcal{K}_j^0}^{\mathcal{B}_0}
    = \sum_{l\in \mathcal{K}_i^1\cap \mathcal{K}_j^0 } Q_{l}^{\mathcal{B}_0}
    = \sum_{l\in \mathcal{K}_i^1\cap \mathcal{K}_j^0 } Q_{l}^{\mathcal{B}_1}
    = Q_{\mathcal{K}_i^1\cap \mathcal{K}_j^0}^{\mathcal{B}_1}.
  \end{equation*}
  Hence, the operator $Q_{\mathcal{K}_i^1\cap \mathcal{K}_j^0}^{\mathcal{B}_1}$ is also bounded.  However, by~\eqref{eq:80} and since
  $\mathcal{K}_1^1, \ldots, \mathcal{K}_{d_1}^1$ is the finest bounded partition of $\{1,\ldots,r\}$ for the basis $\mathcal{B}_1$, it follows that
  \begin{equation}\label{eq:86}
    \mathcal{K}_i^1\cap \mathcal{K}_j^0
    = \mathcal{K}_i^1.
  \end{equation}
  Combining~\eqref{eq:80}, \eqref{eq:81} and~\eqref{eq:86}, we proved that whenever $\mathcal{K}_i^1\cap \mathcal{K}_j^0\neq \emptyset$ for some
  $1\leq i\leq d_1$ and $2 < j \leq d_0$, we have
  \begin{equation}\label{eq:stabbigind}
    \mathcal{K}_i^1
    = \mathcal{K}_j^0.
  \end{equation}
  Since $\mathcal{K}_i^1$, $1\leq i\leq d_1$ and $\mathcal{K}_j^0$, $1\leq j\leq d_0$ are both partitions, we obtain that for each $2 < j\leq d_0$,
  there exists a unique $1\leq i(j)\leq d_1$ such that $\mathcal{K}_{i(j)}^1 = \mathcal{K}_j^0$.  This leaves
  $\mathcal{K}_1^0\cup \mathcal{K}_2^0\neq \emptyset$ uncovered by $\mathcal{K}_{i(j)}^1$, $2 < j\leq d_0$, and since $\mathcal{K}_i^1$,
  $1\leq i\leq d_1$ is a partition, we must have $d_1\geq d_0-1$.

  We will distinguish the following three cases: $d_1 > d_0$, $d_1 = d_0$ and $d_1 = d_0 - 1$.

  \begin{proofcase}[Case A $(d_1 > d_0)$]
    In this case define $J = \{1, \dots, r\}\setminus\{i(j) : 2 < j\leq d_0\}$ and note that
    \begin{equation}\label{eq:91}
      \#J \geq 3
      \qquad\text{as well as}\qquad
      \mathcal{K}_1^0\cup \mathcal{K}_2^0
      = \bigcup_{j\in J}\mathcal{K}_j^1.
    \end{equation}
    Recall that $m\in \mathcal{K}_1^0$ and $m+1\in \mathcal{K}_2^0$, hence, by~\eqref{eq:91} there exists a $j\in J$ such that
    \begin{equation}\label{eq:92}
      \mathcal{K}_j^1 \cap\{m,m+1\}
      = \emptyset,
    \end{equation}
    and the intersection of this $\mathcal{K}_j^1$ with at least one of the sets $\mathcal{K}_1^0$, $\mathcal{K}_2^0$ is non-empty.  Without loss of
    generality we assume that $\mathcal{K}_1^0\cap\mathcal{K}_j^1 \neq \emptyset$.  Using~\eqref{eq:92}, we repeat the argument
    from~\eqref{eq:othepartitioneqb0} and obtain
    \begin{equation*}
      Q_{\mathcal{K}_1^0}^{\mathcal{B}_0} Q_{\mathcal{K}_j^1}^{\mathcal{B}_1}=Q_{\mathcal{K}_j^1\cap \mathcal{K}_1^0}^{\mathcal{B}_0}.
    \end{equation*}
    We get a contradiction with the fact that $\mathcal{K}_1^0,\ldots,\mathcal{K}_{d_0}^0$ is the finest bounded partition.
  \end{proofcase}

  \begin{proofcase}[Case B $(d_1 = d_0)$]\label{test}
    If $d_1=d_0$ by \eqref{eq:stabbigind} and relabeling we may assume that
    \begin{equation}\label{eq:subsum12}
      \mathcal{K}_j^1
      = \mathcal{K}_j^0,\ 2 < j\leq d_0,
      \qquad
      \mathcal{K}_1^1\cup \mathcal{K}_2^1
      =\mathcal{K}_1^0\cup \mathcal{K}_2^0.
    \end{equation}
    
    Now we consider two subcases:
    \begin{description}
    \item[Subcase B1] There exists $s\in\{1,2\}$ such that $\{m,m+1\}\subset\mathcal{K}_s^1$.
    \item[Subcase B2] There exists $s\in\{1,2\}$ such that $m\in\mathcal{K}_s^1$ and $m+1\notin\mathcal{K}_s^1$.
    \end{description}

    \paragraph{\underline{ad Subcase B1}}\hfill\\
    First, we observe that by \Cref{dfn:Q} and~\eqref{eq:75}
    \begin{equation*}
      \left(Q_{m}^{\mathcal{B}_1}+Q_{m+1}^{\mathcal{B}_1}\right)Q_{m}^{\mathcal{B}_0} =Q_{m}^{\mathcal{B}_0}.
    \end{equation*} 
    Hence, exploiting the orthogonality of the canonical projections yields
    \begin{equation}\label{eq:idenm0}
      Q_{\mathcal{K}_s^1}^{\mathcal{B}_1}Q_{m}^{\mathcal{B}_0}=Q_{m}^{\mathcal{B}_0}.
    \end{equation}
    By~\eqref{eq:idenm0} and~\eqref{eq:eqlQproj}, we obtain
    \begin{equation}\label{eq:othepartitioneqm0m1}
      \begin{split}
        Q_{\mathcal{K}_s^1}^{\mathcal{B}_1} Q_{\mathcal{K}_1^0}^{\mathcal{B}_0}
        &= Q_{\mathcal{K}_s^1}^{\mathcal{B}_1} \sum_{l\in \mathcal{K}_1^0} Q_{l}^{\mathcal{B}_0}
          = \sum_{l\in \mathcal{K}_1^0} Q_{\mathcal{K}_s^1}^{\mathcal{B}_1} Q_{l}^{\mathcal{B}_0}
          = Q_{m}^{\mathcal{B}_0}+\sum_{l\in \mathcal{K}_1^0\setminus\{m\}} Q_{\mathcal{K}_s^1}^{\mathcal{B}_1} Q_{l}^{\mathcal{B}_0}\\
        &= Q_{m}^{\mathcal{B}_0}+\sum_{l\in \mathcal{K}_1^0\setminus\{m\}} Q_{\mathcal{K}_s^1}^{\mathcal{B}_1} Q_{l}^{\mathcal{B}_1}
          = Q_{m}^{\mathcal{B}_0} + \sum_{\substack{l\in \mathcal{K}_1^0\cap \mathcal{K}_s^1\\l\neq m }}  Q_{l}^{\mathcal{B}_1}\\
        &= Q_{m}^{\mathcal{B}_0} + \sum_{\substack{l\in \mathcal{K}_1^0\cap \mathcal{K}_s^1\\l\neq m }} Q_{l}^{\mathcal{B}_0}
        = \sum_{l\in \mathcal{K}_1^0\cap \mathcal{K}_s^1}  Q_{l}^{\mathcal{B}_0}
        = Q_{\mathcal{K}_s^1\cap \mathcal{K}_1^0}^{\mathcal{B}_0}.
      \end{split}
    \end{equation} 
    Since $m\in\mathcal{K}_s^1\cap \mathcal{K}_1^0$ and $ \mathcal{K}_1^0, \ldots, \mathcal{K}_{d_0}^0$ is the finest bounded partition of
    $\{1,\ldots,r\}$ for the basis $\mathcal{B}_0$
    \begin{equation}\label{eq:98}
      \mathcal{K}_s^1\cap \mathcal{K}_1^0
      = \mathcal{K}_1^0.
    \end{equation}
    Similarly we get
    \begin{equation}\label{eq:101}
      \mathcal{K}_s^1\cap \mathcal{K}_2^0=\mathcal{K}_2^0.
    \end{equation}
    Recall that by $\eqref{eq:subsum12}$ we have $\mathcal{K}_s^1 \subset \mathcal{K}_1^0\cup \mathcal{K}_2^0$.  Hence, by~\eqref{eq:98}
    and~\eqref{eq:101}
    \begin{equation*}
      \mathcal{K}_s^1 \cap\left(\mathcal{K}_1^0\cup \mathcal{K}_2^0\right)=\left(\mathcal{K}_s^1\cap\mathcal{K}_1^0\right) \cup\left(\mathcal{K}_s^1 \cap \mathcal{K}_2^0\right)=\mathcal{K}_1^0\cup \mathcal{K}_2^0.
    \end{equation*}
    We arrive at a contradiction since $\mathcal{K}_{s'}^1=\emptyset$, where $s'\in\{1,2\}\setminus\{s\}$.

    \paragraph{\underline{ad Subcase B2}}\hfill\\
    If $s=1$, we define
    \begin{equation*}
      T
      = 4 Q_{\mathcal{K}_1^0}^{\mathcal{B}_0} Q_{\mathcal{K}_1^1}^{\mathcal{B}_1}Q_{\mathcal{K}_1^0}^{\mathcal{B}_0}Q_{\mathcal{K}_2^1}^{\mathcal{B}_1}Q_{\mathcal{K}_1^0}^{\mathcal{B}_0},
    \end{equation*}
    and if $s=2$, then
    \begin{equation*}
      T
      = 4 Q_{\mathcal{K}_1^0}^{\mathcal{B}_0} Q_{\mathcal{K}_2^1}^{\mathcal{B}_1}Q_{\mathcal{K}_1^0}^{\mathcal{B}_0}Q_{\mathcal{K}_1^1}^{\mathcal{B}_1}Q_{\mathcal{K}_1^0}^{\mathcal{B}_0}.
    \end{equation*}
    In both cases $T$ is clearly a bounded operator, and we will show that $T=Q_{m}^{\mathcal{B}_0}$.  In order to do so, we check the values of $T$
    on the Schauder basis $(\varphi_{P}^{j,\mathcal{B}_0})$.  For a fixed $P\in \mathcal{A}$, we will make case distinctions for $j$ according to
    \Cref{fig:cases}.
    \begin{table}[H]
      \begin{center}
        \begin{tabular}{|c|c|c|c|c|}
          \hhline{~|-|-|-|-|}
          \multicolumn{1}{l|}{} & Case 1 & Case 2 & Case 3 & Case 4\\
          \hline
          $s=1$ & $j\in\{1,\ldots,r\}\setminus\mathcal{K}_1^0$ & $j\in (\mathcal{K}_1^0\cap \mathcal{K}_2^1)$ & $j\in \mathcal{K}_1^0\cap \mathcal{K}_1^1\setminus\{m\}$ & $j=m$\\
          \hline
          $s=2$ & $j\in\{1,\ldots,r\}\setminus\mathcal{K}_1^0$ & $j\in (\mathcal{K}_1^0\cap \mathcal{K}_1^1)$ & $j\in \mathcal{K}_1^0\cap \mathcal{K}_2^1\setminus\{m\}$ & $j=m$\\
          \hline
        \end{tabular}
      \end{center}
      \caption{Cases}
      \label{fig:cases}
    \end{table}
    Since for both values of $s$ the arguments are analogous, we will only present the proof for $s=1$ below.
    \begin{enumerate}[{Case} 1.]
    \item Let $j\in\{1,\ldots,r\}\setminus\mathcal{K}_1^0$.  By \Cref{dfn:Q} we have
      $Q_{\mathcal{K}_1^0}^{\mathcal{B}_0}\varphi_{P}^{j,\mathcal{B}_0} = 0$, and hence
      $T\varphi_{P}^{j,\mathcal{B}_0} = 0 = Q_{m}^{\mathcal{B}_0}\varphi_{P}^{j,\mathcal{B}_0}$.

    \item Let $j\in (\mathcal{K}_1^0\cap \mathcal{K}_2^1)$.  Since we assumed $s=1$ and by the restriction of Subcase~B2, we have
      $m+1\notin \mathcal{K}_1^0$ and $m\notin \mathcal{K}_2^1$.  Consequently, $j\notin \{m,m+1\}$, and thus~\eqref{eq:75} yields
      $\varphi_P^{j,\mathcal{B}_0} = \varphi_P^{j,\mathcal{B}_1}$.  Referring to \Cref{dfn:Q}, we obtain
      \begin{equation*}
        Q_{\mathcal{K}_1^0}^{\mathcal{B}_0}Q_{\mathcal{K}_2^1}^{\mathcal{B}_1}Q_{\mathcal{K}_1^0}^{\mathcal{B}_0}\varphi_{P}^{j,\mathcal{B}_0}
        = \varphi_{P}^{j,\mathcal{B}_0}
      \end{equation*}
      as well as
      \begin{equation*}
        T \varphi_{P}^{j,\mathcal{B}_0}
        = 4 Q_{\mathcal{K}_1^0}^{\mathcal{B}_0} Q_{\mathcal{K}_1^1}^{\mathcal{B}_1}\varphi_{P}^{j,\mathcal{B}_0}
        = 4 Q_{\mathcal{K}_1^0}^{\mathcal{B}_0} Q_{\mathcal{K}_1^1}^{\mathcal{B}_1}\varphi_{P}^{j,\mathcal{B}_1}
        = 0
        = Q_{m}^{\mathcal{B}_0}\varphi_{P}^{j,\mathcal{B}_0}.
      \end{equation*}
      
    \item Let $j\in \mathcal{K}_1^0\cap \mathcal{K}_1^1\setminus\{m\}$.  Since we assumed $s=1$ and by the restriction of Subcase~B2, we have
      $m+1\notin \mathcal{K}_1^0$.  Consequently, $j\notin \{m,m+1\}$, and thus~\eqref{eq:75} yields
      $\varphi_P^{j,\mathcal{B}_0} = \varphi_P^{j,\mathcal{B}_1}$ and furthermore
      \begin{equation*}
        Q_{\mathcal{K}_2^1}^{\mathcal{B}_1}Q_{\mathcal{K}_1^0}^{\mathcal{B}_0}\varphi_{P}^{j,\mathcal{B}_0}
        = Q_{\mathcal{K}_2^1}^{\mathcal{B}_1}\varphi_{P}^{j,\mathcal{B}_1}
        = 0.
      \end{equation*}
      Therefore $T\varphi_{P}^{j,\mathcal{B}_0}=0=Q_{m}^{\mathcal{B}_0}\varphi_{P}^{j,\mathcal{B}_0}$.
      
    \item Let $j=m$.  In the calculation below, we will use the following identities:
      \begin{align*}
        \varphi_{P}^{m,\mathcal{B}_0}
        &= \frac{1}{\sqrt{2}} (\varphi_{P}^{m,\mathcal{B}_1} + \varphi_{P}^{m+1,\mathcal{B}_1}),
        & \varphi_{P}^{m+1,\mathcal{B}_0}
        &= \frac{1}{\sqrt{2}} (\varphi_{P}^{m+1,\mathcal{B}_1} - \varphi_{P}^{m,\mathcal{B}_1}),\\
        \varphi_{P}^{m,\mathcal{B}_1}
        &= \frac{1}{\sqrt{2}} (\varphi_{P}^{m,\mathcal{B}_0} - \varphi_{P}^{m+1,\mathcal{B}_0}),
        & \varphi_{P}^{m+1,\mathcal{B}_1}
        &= \frac{1}{\sqrt{2}} (\varphi_{P}^{m+1,\mathcal{B}_0} + \varphi_{P}^{m,\mathcal{B}_0}).
      \end{align*}
      Recall that by assumption $m\in \mathcal{K}_1^0\cap \mathcal{K}_1^1$, and observe that by the above identities as well as \Cref{dfn:Q} we obtain
      \begin{equation*}
        \begin{split}
          T\varphi_{P}^{m,\mathcal{B}_0}
          &= 4 Q_{\mathcal{K}_1^0}^{\mathcal{B}_0} Q_{\mathcal{K}_1^1}^{\mathcal{B}_1}Q_{\mathcal{K}_1^0}^{\mathcal{B}_0}Q_{\mathcal{K}_2^1}^{\mathcal{B}_1}Q_{\mathcal{K}_1^0}^{\mathcal{B}_0}\varphi_{P}^{m,\mathcal{B}_0}
            = 4 Q_{\mathcal{K}_1^0}^{\mathcal{B}_0} Q_{\mathcal{K}_1^1}^{\mathcal{B}_1}Q_{\mathcal{K}_1^0}^{\mathcal{B}_0}Q_{\mathcal{K}_2^1}^{\mathcal{B}_1}\varphi_{P}^{m,\mathcal{B}_0}\\
          &= \frac{4}{\sqrt{2}} Q_{\mathcal{K}_1^0}^{\mathcal{B}_0} Q_{\mathcal{K}_1^1}^{\mathcal{B}_1}Q_{\mathcal{K}_1^0}^{\mathcal{B}_0}\varphi_{P}^{m+1,\mathcal{B}_1}
            = 2 Q_{\mathcal{K}_1^0}^{\mathcal{B}_0} Q_{\mathcal{K}_1^1}^{\mathcal{B}_1}\varphi_{P}^{m,\mathcal{B}_0}
            = \sqrt{2} Q_{\mathcal{K}_1^0}^{\mathcal{B}_0} \varphi_{P}^{m,\mathcal{B}_1}\\
          &= \varphi_{P}^{m,\mathcal{B}_0}
            = Q_{m}^{\mathcal{B}_0}\varphi_{P}^{m,\mathcal{B}_0}.
        \end{split}
      \end{equation*}
    \end{enumerate}
    We have shown that $T=Q_{m}^{\mathcal{B}_0}$ but $T$ is a continuous operator.  This stands in contradiction with the fact that
    $\mathcal{K}_1^0, \ldots, \mathcal{K}_{d_0}^0$ is the finest bounded partition of $\{1,\ldots,r\}$ for the basis $\mathcal{B}_0$ and that
    $\{1,m\}\subset \mathcal{K}_1^0$.
  \end{proofcase}

  \begin{proofcase}[Case C $(d_1 = d_0 - 1)$]
    If $d_1 = d_0-1$, then by~\eqref{eq:stabbigind} there is only one set in the finest partition that could contain elements from $\mathcal{K}_1^0$
    and $\mathcal{K}_2^0$.  Since we assumed that $1\in \mathcal{K}_1^1$, we therefore have
    \begin{equation*}
      \mathcal{K}_1^1=\mathcal{K}_1^0\cup \mathcal{K}_2^0.
    \end{equation*}
  \end{proofcase}

  \begin{proofcase}[Summary of the cases]
    Reviewing our cases, we find that the only possible case is Case C.  Consequently, the finest bounded partition of $\{1,2,\ldots,r\}$ for $\mathcal{B}_1$ is
    \begin{equation*}
      \mathcal{K}_1^0\cup \mathcal{K}_2^0 ,\mathcal{K}_3^0, \ldots, \mathcal{K}_{d_0}^0
    \end{equation*} 
    Repeating the construction inductively gives us a sequence of bases $(\mathcal{B}_j)_{j=0}^{d_0-1}$ such that $\mathcal{B}_j$ has a corresponding finest bounded
    partition
    \begin{equation*}
      \bigcup_{k=1}^{j+1} \mathcal{K}_k^0 ,\mathcal{K}_{j+2}^0, \ldots, \mathcal{K}_{d_0}^0.
    \end{equation*} 
    In particular $\mathcal{B}_{d_0-1}$ is a basis for which the finest bounded partition of $\{1,\ldots,r\}$ is just $\{1,\ldots,r\}$.  We conclude the proof by putting
    $\mathcal{B} = \mathcal{B}_{d_0-1}$.
  \end{proofcase}
\end{myproof}

%

\section{Isomorphic structure of $Z_V(L^1)$ spaces}
\label{sec:hardy-iso}

Earlier, we defined three properties of Janson-Sobolev spaces in terms of the continuity of projections (see \ref{dfn:space-type}).  In this section,
we give an isomorphic description of the Property~IIA and Property~IIB spaces for the $L^1$ norm.  In the final subsection, we construct a
Janson-Sobolev space which is isomorphic to the non-primary space $L^1\oplus H^1$.

\subsection{Property~IIA and subspaces of scalar valued $L^1$}
\label{sec:proj-fact}

The following theorem shows that Janson-Sobolev spaces satisfying Property~IIA are isomorphic to a subspace of $L^1$.  In the special case $q=4$ and
$b=(1,1,-1,-1)$, \Cref{thm:isomorphic-classification-L1} identifies $Z_V(L^1)$ with the subspace of $L^1$ spanned by the even-level Haar functions.
Whether this subspace is itself isomorphic to $L^1$ is a long-standing open problem posed by G.~Schechtman.  We refer to \cite[Section~3]{MR2224004}
and~\cite[Problem~6.8]{MR4098600} for further discussion.

\begin{thm}\label{thm:isomorphic-classification-L1}
  If $Z_V(L^1) = Z_V((\mathcal{F}_j),L^1,\boldsymbol{\sigma};\mathbb{R}^l)$ is a Janson-Sobolev space with Property~IIA and $b$ is as in~\eqref{eq:defofvectb} of
  \Cref{dfn:space-type}, then $Z_V(L^1)$ is isomorphic to $Z_{\operatorname{span}\{b^t\}}(L^1)$.
\end{thm}

\begin{proof}
  Let $\mathcal{B}=\{\nu_1, \nu_2, \ldots, \nu_r\}$ be an orthonormal basis of $V$. Since $Z_V(L^1)$ has Property~IIA, we have
  \begin{equation}\label{eq:iso1stp}
    Z_V(L^1)\simeq Z_{\operatorname{span}\{\nu_1\}}(L^1)\oplus Z_{\operatorname{span}\{\nu_2\}}(L^1)\oplus \cdots\oplus Z_{\operatorname{span}\{\nu_r\}}(L^1),
  \end{equation}
  and observe that $\nu_j = w_j b^{t}$ with $w_j\neq 0$.  If $f\in Z_{\operatorname{span}\{\nu_j\}}(L^1)$, then by~\eqref{eq:119}
  \begin{equation}\label{eq:207}
    f
    = \sum_{P\in \mathcal{A}} a_P \varphi_P^{1,\{\nu_j\}}
    = \sum_{P\in \mathcal{A}} a_P w_j \varphi_P^{1,\{b^t\}},
  \end{equation}
  hence, $f\in Z_{\operatorname{span}\{\nu_j\}}(L^1)$ if and only if there exists $g\in Z_{\operatorname{span}\{b^t\}}(L^1)$ such that $f = w_j g$.
  Projecting the vector valued function $f$ onto a non-zero coordinate of $w_j$ establishes that $Z_{\operatorname{span}\{\nu_j\}}(L^1)$ and
  $Z_{\operatorname{span}\{b^t\}}(L^1)$ are isomorphic.

  By~\eqref{eq:iso1stp} and~\eqref{eq:121} in \Cref{rem:W-auxiliary}, we obtain
  \begin{equation*}
    Z_V(L^1)
    \simeq Z_{\operatorname{span}\{b^t\}}(L^1)\oplus Z_{\operatorname{span}\{b^t\}}(L^1)\oplus \cdots\oplus Z_{\operatorname{span}\{b^t\}}(L^1)
    \simeq Z_{\operatorname{span}\{b^t\}}(L^1).
  \end{equation*}
\end{proof}

\subsection{Janson-Sobolev-Hardy spaces with $L^1$-norm do not have Property~I}
\label{sec:janson-sobolev-hardy}

In this subsection, we show that Janson-Sobolev-Hardy spaces exhibit a dichotomy with respect to Property~IIA and Property~IIB. In particular, we
prove that every such space must fall into exactly one of these two categories. The argument relies on viewing Janson-Sobolev-Hardy spaces as
subspaces of $H^1((\mathcal{F}_j);\mathbb{R}^l)$ and establishing the boundedness of the canonical projections via a duality argument with
$\mathrm{BMO}$.

\begin{thm}\label{thm:janson-sobolev-hardy-property-I}
  Every Janson-Sobolev-Hardy space $Z_V(L^1) = Z_V((\mathcal{F}_j),L^1,\boldsymbol{\sigma};\mathbb{R}^l)$ either has Property~IIA or Property~IIB.
\end{thm}

From the definition of Janson-Sobolev-Hardy spaces (see \Cref{dfn:W-auxiliary}), we can infer that they are subspaces of
$H^1((\mathcal{F}_j);\mathbb{R}^l)$.  Thus, the proof of \Cref{thm:janson-sobolev-hardy-property-I} follows from the lemma below.

Let $\mathcal{\tilde{B}}$ be a basis of $\mathbb{R}^{l\times q}_0$.  Hence, the corresponding basis
$(\varphi^{k,\tilde{\mathcal{B}}}_P : P\in \mathcal{A},\ 1\leq k\leq l(q-1))$ is a basis of $H^1((\mathcal{F}_j);\mathbb{R}^l)$.
\begin{lem}\label{lem:canonical-projection-H1}
  The canonical projections $Q^{\mathcal{\tilde{B}}}_k\colon H^1((\mathcal{F}_j);\mathbb{R}^l)\rightarrow H^1((\mathcal{F}_j);\mathbb{R}^l)$,
  $1\leq k\leq r$ are all continuous.
\end{lem}

\begin{proof}
  Fix $1\leq k\leq r$, and note that the dual operator of $Q_k^{\tilde{\mathcal{B}}}$ has the same form as $Q_k^{\tilde{\mathcal{B}}}$ itself.  We
  will now prove the boundedness of $Q_k^{\tilde{\mathcal{B}}}$ in the dual space $\mathrm{BMO}$ of $H^1$.  To this end, let $f\in \mathrm{BMO}$.
  First, note that
  \begin{equation}\label{eq:124}
    \cond \bigl((\Delta_{i_0} g)^t\cdot (\Delta_{i_1} g) \big| \mathcal{F}_j\bigr)
    = 0,
    \qquad j \leq i_0 < i_1,
  \end{equation}
  for all $g\in \mathrm{BMO}$.  Hence, we obtain
  \begin{align}
    \|Q_k^{\tilde{\mathcal{B}}} f\|_{\mathrm{BMO}}^2
    &= \sup_j \bigl\| \cond\bigl( \|Q^{\mathcal{\tilde{B}}}_k f - \cond( Q^{\mathcal{\tilde{B}}}_k f | \mathcal{F}_j)\|_{\ell^2_l}^2 \big| \mathcal{F}_j \bigr) \bigr\|_{L^{\infty}}\\
    &= \sup_j \bigl\| \sum_{i=j+1}^{\infty} \cond\bigl( \|\Delta_i Q^{\mathcal{\tilde{B}}}_k f\|_{\ell^2_l}^2 \big| \mathcal{F}_j \bigr) \bigr\|_{L^{\infty}}.
  \end{align}
  Finally, since $\Delta_i Q^{\mathcal{\tilde{B}}}_k f = Q^{\mathcal{\tilde{B}}}_k \Delta_i f$ and
  $\cond( \|Q^{\mathcal{\tilde{B}}}_k \Delta_i f\|_{\ell^2_l} | \mathcal{F}_j ) \leq \cond( \|\Delta_i f\|_{\ell^2_l} | \mathcal{F}_j)$ pointwise
  (since $Q^{\mathcal{\tilde{B}}}_k$ is an orthogonal projection on $L^2(P)$ for each $P\in \mathcal{AF}_j$), we obtain
  \begin{align}
    \|Q_k^{\tilde{\mathcal{B}}} f\|_{\mathrm{BMO}}^2
    &\leq \sup_j \bigl\| \sum_{i=j+1}^{\infty} \cond\bigl( \|\Delta_i f\|_{\ell^2_l}^2 \big| \mathcal{F}_j \bigr) \bigr\|_{L^{\infty}}
      = \|f\|_{\mathrm{BMO}}^2.
  \end{align}
  The latter equality follows again by exploiting~\eqref{eq:124}.  So we have proved that
  $Q^{\mathcal{\tilde{B}}}_k\colon \mathrm{BMO} \rightarrow \mathrm{BMO}$ is bounded, and thus, by duality
  $Q^{\mathcal{\tilde{B}}}_k\colon H^1\rightarrow H^1$ is bounded as well.
\end{proof}

Note that \Cref{lem:canonical-projection-H1} implies that every Janson-Sobolev-Hardy space $Z_V$ is isomorphic to a complemented subspace of $H^1$.
To see this, extend a given basis $\mathcal{B}$ for $V$ to a basis $\tilde{\mathcal{B}}$ of $\mathbb{R}^{l\times q}_0$, and observe that the natural
projection determined by the basis $\mathcal{B}$ is bounded by the above lemma.

\subsection{Property~IIB and the scalar valued Hardy space $H^1$}
\label{sec:property-iib-scalar}

We show that every Janson-Sobolev space with Property~IIB and every Janson-Sobolev-Hardy space is isomorphic to $H^1$.  This requires a more detailed
analysis using results from S.~Janson~\cite{sJ77}.

\begin{thm}\label{thm:isomorphic-classification-H1}
  Let $Z_V(L^1) = Z_V((\mathcal{F}_j),L^1,\boldsymbol{\sigma};\mathbb{R}^l)$ denote either a
  \begin{itemize}
  \item Janson-Sobolev-Hardy space, or a
  \item Janson-Sobolev space with Property~IIB.
  \end{itemize}
  Then $Z_V(L^1)$ is isomorphic to $H^1$.
\end{thm}

As was discussed in the previous subsection, Janson-Sobolev-Hardy spaces are isomorphic to subspaces of $H^1$.  However, this is not immediate for
Janson-Sobolev spaces with Property~IIB and requires the following proposition.

\begin{pro}\label{pro:norank}
  Let $Z_V(L^1) = Z_V((\mathcal{F}_j),L^1,\boldsymbol{\sigma};\mathbb{R}^l)$ denote a Janson-Sobolev space with Property~IIB.  Then $V$ contains no
  rank-one matrices.
\end{pro}

\begin{proof}[Proof of~\Cref{pro:norank}]
  If $r=1$, then $V$ is spanned by just one matrix $v$, and since $Z_V$ has Property~IIB, it does not have Property~IIA.  Thus, $v$ is not a rank-one
  matrix.  Assume that $r>1$ and that $V$ contains a rank-one matrix $u$, i.e., there exist $w\in\mathbb{R}^l$ and $b\in\mathbb{R}^q_0$ such that
  \begin{equation}\label{eq: uasrankone}
    u
    = wb^t.
  \end{equation}
  Since $r>1$ and $Z_V$ has Property~IIB, there exists a matrix $\tilde{v}\in V$ such that
  \begin{equation*}
    \tilde{v}
    \notin \{ xb^t: x\in\mathbb{R}^l \}.
  \end{equation*}
  Note that either $\langle \tilde{v}, u \rangle \neq 0$ or $\langle \tilde{v}+u, u \rangle\neq 0$.  If the former is true, we simply put
  $v = \tilde{v}/\|\tilde{v}\|_V$, and if not, we define $v = (\tilde{v} + u)/\|\tilde{v} + u\|_V$.  We therefore record the following facts:
  \begin{equation}\label{eq:118}
    v\notin\{ xb^t: x\in\mathbb{R}^l \}
    \quad\mbox{and}\quad
    \langle v, u\rangle
    \neq 0.
  \end{equation}
  Take an orthonormal basis $\mathcal{B}:=\{v,\beta_2,\ldots,\beta_r\}$ of $V$.  Then
  \begin{equation}\label{eq: uinbasisv}
    u
    = c_1 v + \sum_{j=2}^{r} c_j \beta_j,
  \end{equation}
  where $c_1,\ldots,c_r\in\mathbb{R}$, and by~\eqref{eq:118} we have $c_1\neq 0$.  As outlined in \Cref{sec:basic-definitions}, the choice of the
  orthonormal basis $\mathcal{B}$ gives rise to the system $(\varphi_P^{k,\mathcal{B}} : P\in \mathcal{A},\ 1\leq k\leq r)$, which by
  \Cref{pro:schauder-Z} forms a Schauder basis for $Z_V(L^1)$.  The associated canonical projections $Q_k^{\mathcal{B}}\colon Z_V(L^1) \to Z_V(L^1) $,
  $1\leq k\leq r$, which are defined in \Cref{dfn:Q}, are bounded since we assumed that $Z_V(L^1)$ has Property~IIB.  Recall that
  $Q_1^{\mathcal{B}}\colon Z_V(L^1) \rightarrow Z_V(L^1)$ given by the formula
  \begin{equation*}
    Q_1^{\mathcal{B}}\Bigl(
    \sum_{P\in \mathcal{A}}\sum_{k=1}^r a_P^k \varphi^{k,\mathcal{B}}_P
    \Bigr)
    = \sum_{P\in \mathcal{A}} a_P^1 \varphi^{1,\mathcal{B}}_P.
  \end{equation*} 
  Observe that by~\eqref{eq:119} and~\eqref{eq: uinbasisv}, we have
  \begin{equation}\label{eq:125}
    Q_1^{\mathcal{B}} \sum_{P\in \mathcal{A}} a_P \varphi_P^{1,\{u\}}
    = c_1\sum_{P\in \mathcal{A}} a_P \varphi^{1,\mathcal{B}}_P.
  \end{equation}  
  Since $u$ is a rank-one matrix the construction from S.~Janson’s paper~\cite[Lemma~6]{sJ77} gives us a function
  $g \in Z_{\operatorname{span}\{b^t\}}(L^1)$ which is not in $H^1((\mathcal{F}_j);\mathbb{R})$.  Recall that $w = (w_i)_{i=1}^l$ and $b$ are defined
  in \eqref{eq: uasrankone}, and extend $\{b\}$ to a basis $\mathcal{D} = \{b,\alpha_2,...,\alpha_{q-1}\}$ of $\mathbb{R}^q_0$.  Let $(e_i)_{i=1}^l$
  denote the standard unit vector basis of $\mathbb{R}^l$.  Now, for each $1\leq i\leq l$, define the operators
  $A_i, A_{i+l}\colon \mathbb{R}^q_0\to \mathbb{R}^q_0$ by the formulas
  \begin{align}
    A_i b
    &= w_i b,
    &A_i \alpha_s
    &= (s + |w_i|) \alpha_s,
      \label{eq:122}\\
    A_{l+i} b
    &= c_1 (e_i^t v)^t,
    &A_{l+i} \alpha_s
    &= b,
      \label{eq:123}
  \end{align}
  for all $2\leq s\leq q-1$, where $v$ is as in~\eqref{eq:118}.  Thus we record that the operators $A_1, \ldots, A_{2l}$ have no common real eigenvector.
  By~\eqref{eq:119}, \Cref{dfn:janson-operator}, \eqref{eq:122}, \eqref{eq:123}, \eqref{eq: uinbasisv}, the family of corresponding Janson operators
  $T_{A_1}, T_{A_2}, \ldots, T_{A_{2l}}$ satisfies
  \begin{equation}\label{eq:126}
    T_{A_i} \varphi_P^{1,\{b^t\}}
    = w_i \varphi_P^{1,\{b^t\}}
    \qquad\text{and}\qquad
    T_{A_{l+i}} \varphi_P^{1,\{b^t\}}
    = c_1 e_i^t \varphi_P^{1,\mathcal{B}},
  \end{equation}
  for all $1\leq i\leq l$.

  Recall that since $g \in Z_{\operatorname{span}\{b^t\}}(L^1)$, we have $g = \sum_{P\in \mathcal{A}} a_P \varphi_P^{1,\{b^t\}}$.  Now put $f = w g$
  pointwise, and note that by~\eqref{eq:119} and~\eqref{eq: uasrankone} we have
  \begin{equation}\label{eq:108}
    f = \sum_{P\in \mathcal{A}} a_P w \varphi_P^{1,\{b^t\}} = \sum_{P\in \mathcal{A}} a_P \varphi_P^{1,\{u\}},
  \end{equation}
  hence $f\in Z_{\operatorname{span}\{u\}}(L^1)\subset Z_V(L^1)$.  Observe that by~\eqref{eq:126} and~\eqref{eq:125}, we obtain
  \begin{equation}\label{eq:127}
    \begin{aligned}
      T_{A_i} g
      &= \sum_{P\in \mathcal{A}} a_P T_{A_i} \varphi_P^{1,\{b^t\}}
        = \sum_{P\in \mathcal{A}} a_P w_i \varphi_P^{1,\{b^t\}}
        = w_i g
        = e_i^t f,\\
      T_{A_{l+i}} g
      &= \sum_{P\in \mathcal{A}} a_P T_{A_{l+i}} \varphi_P^{1,\{b^t\}}
        = e_i^tc_1\sum_{P\in \mathcal{A}} a_P \varphi_P^{1,\mathcal{B}}
        = e_i^t Q_1^{\mathcal{B}} f,
    \end{aligned}
  \end{equation}
  for all $1\leq i\leq l$.

  Since $Z_V(L^1)$ has Property~IIB, $Q_1^{\mathcal{B}}\colon Z_V(L^1) \to Z_V(L^1)$ is bounded; hence if $f\in Z_V(L^1)$, so is $Q_1^{\mathcal{B}}f$.
  Thus, by~\eqref{eq:127}, we obtain $T_{A_i}g \in L^1$ for $i\in \{1,\ldots, 2l\}$.  By S.~Janson's theorem~\cite[Theorem 3]{sJ77}
  (see~\Cref{thm:janson-1}), we obtain $g\in H^1((\mathcal{F}_j);\mathbb{R})$, which contradicts the construction of $g$ (we refer to the paragraph
  below~\eqref{eq:125}).
\end{proof}

For the $L^1$ norm, we will establish that Janson-Sobolev spaces with Property~IIB, Janson-Sobolev-Hardy spaces and $H^1$ are all isomorphic.
\begin{proof}[Proof of \Cref{thm:isomorphic-classification-H1}]
  In this proof, we abbreviate $Z_V((\mathcal{F}_j),L^1,\boldsymbol{\sigma};\mathbb{R}^l)$ by $Z_V(\boldsymbol{\sigma})$.

  First, we claim that for any choice of orthonormal basis $\mathcal{B}$ of $V$, all the projections
  $Q_k^{\mathcal{B}}\colon Z_V(\boldsymbol{\sigma}) \to Z_V(\boldsymbol{\sigma})$, $1\leq k\leq r$ are bounded.  In case of Janson-Sobolev-Hardy
  spaces, inspecting the definition of their norm reveals that it is a subspace of $H^1((\mathcal{F}_j);\mathbb{R}^l)$.  Therefore, the claim follows
  from \Cref{lem:canonical-projection-H1}.  If $Z_V(\boldsymbol{\sigma})$ is a Janson-Sobolev space with Property~IIB, the boundedness of the
  projections follows from \Cref{dfn:space-type}.  Therefore in both cases we have
  \begin{equation*}
    Z_V(\boldsymbol{\sigma}) \simeq \oplus_{k=1}^r Z_{V_k}(\boldsymbol{\sigma}),
  \end{equation*}
  for a fixed orthonormal basis $\mathcal{B}=\{v^1,\ldots, v^r\}$ of $V$ and $V_k = \operatorname{span}\{v^k\}$.  Since
  $H^1\simeq \bigoplus_{k=1}^{r} H^1$ it is enough to prove that $Z_{V_k}(\boldsymbol{\sigma}) \simeq H^1$.

  Let $1\leq k\leq r$ be fixed.  Since the matrix $v^k$ is non-zero, there exists $1\leq p\leq l$ such that $\mu = e_p^t v^k\neq 0$.  As before, $e_p$ denotes the
  $p$-th standard unit vector in $\mathbb{R}^l$.  We extend $\{\mu^t\}$ to an orthonormal basis $\{\mu^t, \nu_2, \ldots,\nu_{q-1}\}$ of
  $\mathbb{R}^{q}_0$, and will now define the linear operators $A_s\colon \mathbb{R}^q_0 \rightarrow \mathbb{R}^q_0$, $1\leq s\leq l$.  The operator $A_0\colon \mathbb{R}^q_0 \rightarrow \mathbb{R}^q_0$ is given by
  \begin{equation}
    A_0(\mu^t)
    = \mu^t
    \quad\mbox{ and }\quad
    A_0(\nu_\tau)
    = \tau \nu_{\tau} \mbox{ for } 2\leq \tau\leq q-1,
  \end{equation}
  and for $1\leq s\leq l$, we define $A_s\colon \mathbb{R}^q_0 \rightarrow \mathbb{R}^q_0$ by the formula
  \begin{equation}
    A_s(\mu^t)
    = (e_s^t v^k)^t
    \quad\mbox{ and }\quad
    A_s(\nu_\tau)
    = \mu^t \mbox{ for } 2\leq \tau\leq q-1.
  \end{equation}
  Given $f\in Z_{V_k}(\boldsymbol{\sigma})$, and putting $g = e_p^t f \in L^1$, observe that by similar reasoning as in~\eqref{eq:127}, we obtain
  \begin{equation}\label{eq:129}
    T_{A_0} g = g
    \qquad\text{and}\qquad
    T_{A_s}g = e_s^t f \in L^1,\quad 1\leq s\leq l.
  \end{equation}
  The linear operators $A_s$, $0\leq s\leq l$, have no common real eigenvector, therefore S.~Janson's theorem~\cite[Theorem~3]{sJ77}
  (see~\Cref{thm:janson-1}) yields $g\in H^1$.  Moreover, inspecting its proof reveals that $\|g\|_{H^1}\lesssim \sum_{s=0}^l \|T_{A_s} g\|_{L^1}$.  Hence,
  by~\eqref{eq:129} we obtain $\|g\|_{H^1}\lesssim \|f\|_{L^1}$.  Since the Janson operators $T_{A_s}\colon H^1\to H^1$, $0\leq s\leq l$ are
  bounded~\cite[Theorem~2]{sJ77} (see~\Cref{thm:janson-0}), the identities~\eqref{eq:129} yield
  \begin{equation*}
    \|f\|_{L^1}
    \lesssim \sum_{s=0}^l \|T_{A_s} g\|_{L^1}
    \leq \sum_{s=0}^l \|T_{A_s} g\|_{H^1}
    \lesssim \|g\|_{H^1}.
  \end{equation*}
  To summarize, we proved
  \begin{equation}\label{eq:130}
    \|g\|_{H^1} \simeq \|f\|_{L^1}.
  \end{equation}
  In the case that $Z_{V_k}(\boldsymbol{\sigma})$ is a Janson-Sobolev-Hardy space, a simplified version of the above reasoning gives us
  $\|g\|_{H^1} \simeq \|f\|_{H^1}$ as well.
  
  We now define the space $X_k\subset H^1$ by
  \begin{equation*}
    X_k
    = \overline{\operatorname{span}\{ \varphi_P^{1,\{\mu^t\}}: P\in\mathcal{A} \}}^{H^1}.
  \end{equation*}
  In either case, we define $S\colon Z_{V_k}(\boldsymbol{\sigma}) \rightarrow X_k$ by $Sf = e_p^t f$, and observe that by~\eqref{eq:129} we have
  \begin{equation*}
    S^{-1}g
    = (T_{A_s} g)_{s=1}^l.
  \end{equation*}
  By~\eqref{eq:130} and the remark below it, $S$ is an isomorphism.
  
  So in both cases we will prove that $X_k\simeq H^1$.  Given an independent sequence of Rademacher functions $\boldsymbol{\rho}$, note that
  $X_k = Z_{\operatorname{span}\{\mu^t\}}((\mathcal{F}_j),L^1,\boldsymbol{\rho};\mathbb{R})$, i.e., we interpret $X_k$ as a Janson-Sobolev-Hardy space.  By
  \Cref{rem:W-auxiliary}, we have that
  \begin{equation}\label{eq:131}
    X_k
    \simeq \bigoplus_{s=1}^{q-1} X_k.
  \end{equation}
  We know that for any linear transformation $A\colon \mathbb{R}^q_0 \rightarrow \mathbb{R}^q_0$, the corresponding Janson operator $T_A$ acts boundedly from
  $H^1$ to $H^1$.  We take an orthonormal basis $\widetilde{\mathcal{B}} = \{\nu_1, \ldots, \nu_{q-1}\}$ of $\mathbb{R}^{1\times q}_0$ and for
  $s\in\{1,2,\ldots,q-1\}$ we take linear transformations $A_s\colon \mathbb{R}^q_0 \rightarrow \mathbb{R}^q_0$ and $B_s\colon \mathbb{R}^q_0 \rightarrow \mathbb{R}^q_0$ such that
  \begin{equation*}
    A_s(\mu^t)
    = \nu_s^t
    \qquad \mbox{ and }\qquad
    B_s(\nu_s^t)
    = \mu^t,
  \end{equation*}
  and we put
  \begin{equation*}
    Y_s
    = \overline{\operatorname{span}\{ \varphi_P^{1,\{\nu_s\}}: P\in\mathcal{A} \}}^{H^1}
    = \overline{\operatorname{span}\{ \varphi_P^{s,\widetilde{\mathcal{B}}}: P\in\mathcal{A} \}}^{H^1}
    = Q_s^{\widetilde{\mathcal{B}}}(H^1).
  \end{equation*}
  By the above we see that $T_{A_s}\colon X_k \rightarrow Y_s$ is an isomorphism and $T_{A_s}^{-1}=T_{B_s}$.  Hence, \eqref{eq:131} yields
  \begin{equation*}
    X_k
    \simeq \bigoplus_{s=1}^{q-1} Y_s
    \simeq H^1((\mathcal{F}_j);\mathbb{R})
  \end{equation*}
  where the last isomorphism follows from the fact $\widetilde{\mathcal{B}}$ is a basis of $\mathbb{R}^{1\times q}_0$.
\end{proof}

\subsection{An example of a non-primary Janson-Sobolev space}
\label{sec:an-example-non}

In this section, we explicitly construct a Janson-Sobolev space that is neither primary nor possesses the primary factorization property.

\begin{thm}\label{thm:non-primary-janson-sobolev}
  Let $(\mathcal{F}_j)_{j=0}^{\infty}$ denote the $3$-adic filtration.  Then there exists a subspace $V$ of $\mathbb{R}_0^{3\times 3}$ such that the
  associated Janson-Sobolev space $Z_V(L^1) = Z_V((\mathcal{F}_j),L^1,(\sigma_j);\mathbb{R}^3)$, with $\sigma_j = 1$ for all $j\in \mathbb{N}_0$, is
  neither primary nor has the primary factorization property.  Moreover, $Z_V(L^1)$ satisfies neither Property~IIA nor Property~IIB, and therefore has
  Property~I.
\end{thm}

\begin{proof}
  To this end, we define
  \begin{equation}\label{eq:H1L1basis}
    v^1
    = \begin{pmatrix} 
      0 & 0 & 0  \\ 
      0 & 0 & 0  \\ 
      1 & -1 & 0  
    \end{pmatrix},
    \qquad
    v^2
    = \begin{pmatrix} 
      0 & 0 & 0  \\ 
      0 & 0 & 0  \\ 
      -1 & -1 & 2  
    \end{pmatrix},
    \qquad
    v^3
    = \begin{pmatrix} 
      1 & 0 & -1  \\ 
      1 & -2 & 1  \\ 
      0 & 0 & 0  
    \end{pmatrix} 
  \end{equation}
  and put
  \begin{equation}\label{eq:213}
    \mathcal{B}
    = \{v^1, v^2, v^3\},
    \qquad
    V
    = \operatorname{span} \mathcal{B}.
  \end{equation}
  Clearly, the above matrices are orthogonal with respect to the scalar product of $V$.  We have a corresponding Schauder basis of $Z_V(L^1)$ denoted
  by $\varphi^{1,\mathcal{B}}_P, \varphi^{2,\mathcal{B}}_P, \varphi^{3,\mathcal{B}}_P$ for $P \in \mathcal{A}$.  Let
  \begin{equation*}
    f
    = \sum_{P \in \mathcal{A}} \sum_{k=1}^{3} a^k_P \varphi_P^{k,\mathcal{B}} \in Z_V(L^1).
  \end{equation*}
  Observe that
  \begin{equation*}
    Q_{\{1,2\}}^{\mathcal{B}} f
    = \sum_{P \in \mathcal{A}} \sum_{k=1}^{2} a^k_P \varphi_P^{k,\mathcal{B}}
    =  e_3^t f  e_3 \in Z_V(L^1),
  \end{equation*}
  and thus, we have the estimate
  \begin{equation}
    \|Q_{\{1,2\}}^{\mathcal{B}} f \|_{Z_V}
    \leq \|f \|_{Z_V}.
  \end{equation}
  Therefore, $Q_{\{1,2\}}^{\mathcal{B}}\colon Z_V(L^1) \rightarrow Z_V(L^1)$ and $Q_{3}^{\mathcal{B}}\colon Z_V(L^1) \rightarrow Z_V(L^1)$ are both
  bounded.  Note that $\operatorname{span}\{ (1, -1, 0), (-1, -1, 2)\} = \mathbb{R}^3_0$.  Hence, the operator
  $T\colon Q_{\{1,2\}}^{\mathcal{B}}(Z_V(L^1)) \rightarrow L^1$ given by
  \begin{equation}\label{eq:defexampleoper}
    T f
    =  e_3^t f 
  \end{equation}
  is an isometric isomorphism.  Now we examine the image of $Q_3^{\mathcal{B}}$.  Observe that
  $Q_{3}^{\mathcal{B}}(Z_V(L^1)) = Z_{\widetilde V}(L^1)$, where $\widetilde V = \operatorname{span} \{ v^3\}$.  Since
  $\operatorname{dim} \widetilde V = 1$, all canonical projections are bounded.  Moreover, $v^3$ is not a rank-one matrix, and thus, the space
  $Z_{\widetilde V}(L^1)$ has Property~IIB (see \Cref{dfn:space-type}).  By \Cref{thm:isomorphic-classification-H1}, $Z_{\widetilde V}(L^1)$ is
  isomorphic to $H^1$.

  So far, we proved that the space $Z_V(L^1)$ is isomorphic to $L^1\oplus H^1$.  Observe that, on the one hand, $L^1\oplus H^1$ is not isomorphic to
  $L^1$ since $H^1$ contains a complemented copy of $\ell^2$ (e.g., the Rademacher system), but $L^1$ cannot contain any infinite-dimensional
  complemented reflexive subspaces (see~\cite{MR2020,MR4094}; see also~\cite[Proposition~5.7.1]{MR3526021}).  On the other hand, since $H^1$ has an
  unconditional basis, $L^1$ cannot embed into it (\cite{MR131154}; see also~\cite[Theorem~6.3.3]{MR3526021}), and thus $L^1\oplus H^1$ cannot be
  isomorphic to $H^1$.  To summarize, we showed that $L^1\oplus H^1$ is neither isomorphic to $L^1$ nor $H^1$, i.e., $L^1\oplus H^1$ is not primary.  By
  Pe{\l}czy\'nski's decomposition technique (see~\cite{MR126145}; see also~\cite{MR3526021}) and since $L^1\oplus H^1$ is isomorphic to its $\ell^1$ direct
  sum, $L^1\oplus H^1$ also cannot possess the primary factorization property.

  Since $v^3$ does not have rank~$1$, $Z_V(L^1)$ cannot possess Property~IIA.  In the above paragraph, we have already shown that $Z_V(L^1)$ is not
  isomorphic to $H^1$, thus by \Cref{thm:isomorphic-classification-H1}, $Z_V(L^1)$ cannot have Property~IIB.  By \Cref{thm:trichotomy}, $Z_V(L^1)$
  must have Property~I.
\end{proof}

%
\section{The variation norm}
\label{sec:branch-vari-norm}

In~\cite{MR2975943}, E.~M.~Semenov and S.~N.~Uksusov characterized the boundedness of Haar multipliers (i.e., diagonal operators with respect to the
Haar basis), via a branch variation norm.  Subsequently, adaptations of this norm for the Haar system were explored: for the bi-parameter setting, we
refer to~\cite{MR4430957,MR4816119}, while in the one-parameter case, related adaptations appeared in~\cite{MR4894823,MR4884827}.

In this section, we extend the dyadic approach to the general setting of $q$-adic filtrations and vector valued diagonal operators.  In this context,
the canonical projections $Q_{\mathcal{K}}$ play a central role and naturally appear in our formulation of the branch variation norm.

\begin{dfn}\label{dfn:variation-norm}
  Let $Z_V$ denote a Janson-Sobolev space or a Janson-Sobolev-Hardy space.  Suppose that $\mathcal{K}_1, \dots, \mathcal{K}_s$ is a partition of
  $\{1, \dots, r\}$ and let $k_t\in \mathcal{K}_t$ for all $1\leq t\leq s$.  For any sequence of scalars
  $D = (d_P^k : P\in \mathcal{A},\ 1\leq k\leq r)$, we define the following \emph{variation norm $\|\cdot\|_{\mathcal{K}_1, \dots, \mathcal{K}_s}$ (with respect to the partition
    $\mathcal{K}_1, \dots, \mathcal{K}_s$)} by
  \begin{equation}\label{eq:pro:variation-norm:1}
    \|D\|_{\mathcal{K}_1, \dots, \mathcal{K}_s}
    := \sum_{t=1}^{s} \|Q_{\mathcal{K}_t}\colon Z_V\to Z_V\|\cdot
    \Bigl(
    |d_{[0,1)}^{k_t}| + \sum_{j=1}^{\infty} \max_{P\in \mathcal{AF}_j} |d_{\pi(P)}^{k_t} - d_P^{k_t}|
    \Bigr).
  \end{equation}
\end{dfn}
We remark that if all the $Q_{\mathcal{K}_t}$, $1\leq t\leq s$ are bounded on $Z_V$, and if $d_P^k = d_P^{k_t}$ for all $P\in \mathcal{A}$,
$k\in \mathcal{K}_t$ and $1\leq t\leq s$, then $\|\cdot\|_{\mathcal{K}_1, \dots, \mathcal{K}_s}$ is a norm on the space of such sequences.

\begin{figure}
  \centering
  \begin{tikzpicture}
    \node[anchor=south west,inner sep=0] (img) {\includegraphics[width=.8\textwidth]{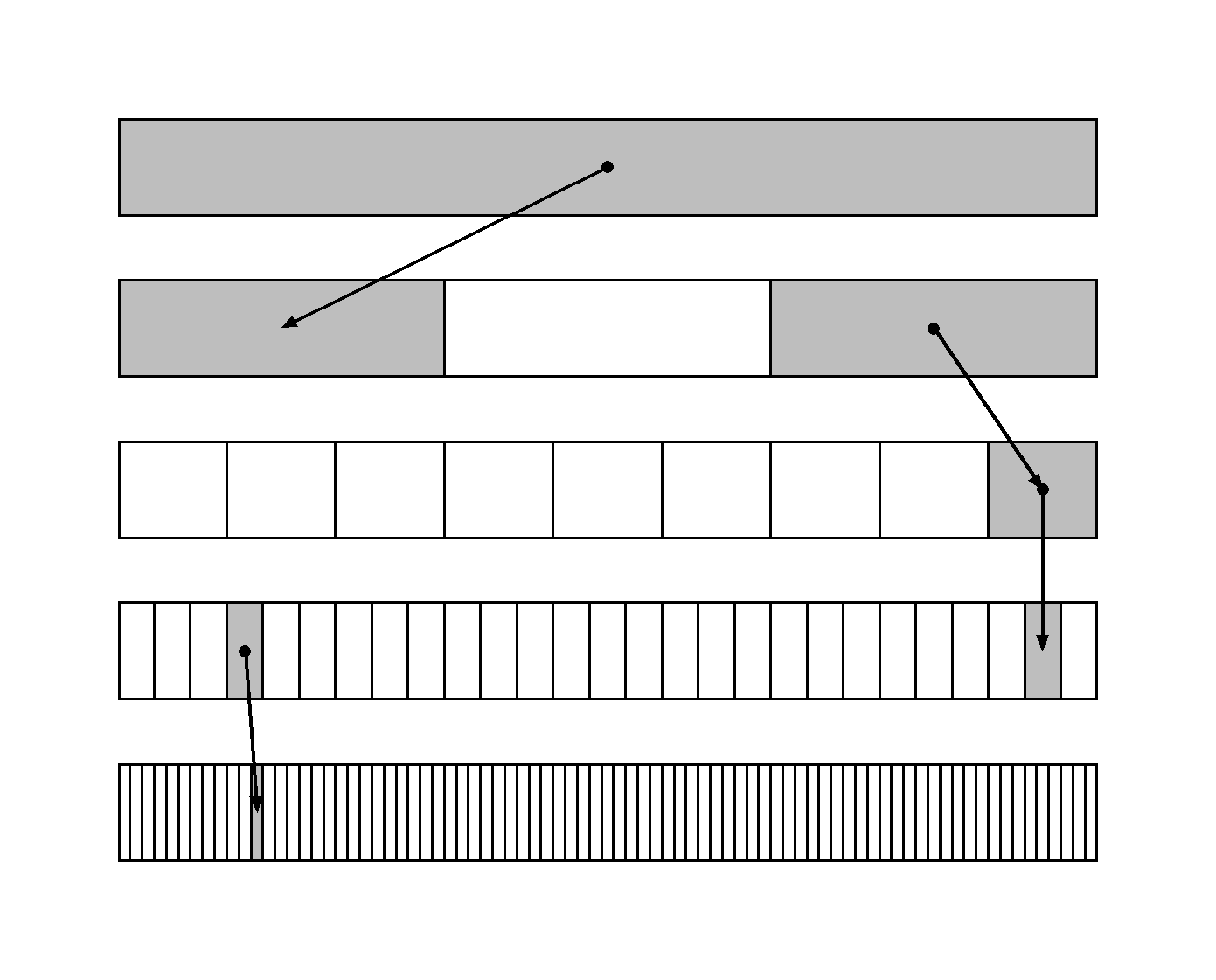}};
    \begin{scope}[x={(img.south east)},y={(img.north west)}]
      \node at (0.02,0.83) {{\small $\mathcal{AF}_0$}};%
      \node at (0.02,0.66) {{\small $\mathcal{AF}_1$}};%
      \node at (0.02,0.495) {{\small $\mathcal{AF}_2$}};%
      \node at (0.02,0.33) {{\small $\mathcal{AF}_3$}};%
      \node at (0.02,0.16) {{\small $\mathcal{AF}_4$}};%
    \end{scope}
  \end{tikzpicture}
  \caption{Let $D=(d_P^k : P\in\mathcal A,\ 1\le k\le r)$ be a diagonal operator on the triadic grid.  For each level $j\in \mathbb{N}$, we choose an atom
    $P\in\mathcal{AF}_j$ maximizing the quantity $|d_{\pi(P)}^{k_t}-d_P^{k_t}|$, as illustrated by the arrows between the gray intervals.  The main part
    to the variation norm of $D$ is obtained by summing these maximal differences.}
\end{figure}

\begin{pro}\label{pro:variation-norm}
  Suppose $\mathcal{K}_1, \dots, \mathcal{K}_s$ is a partition of $\{1, \dots, r\}$ such that all the $Q_{\mathcal{K}_t}$, $1\leq t\leq s$ are bounded, and let
  $k_t\in \mathcal{K}_t$ for all $1\leq t\leq s$.  Let $D = (d_P^k : P\in \mathcal{A},\ 1\leq k\leq r)$ denote any sequence of scalars with
  $d_P^k = d_P^{k_t}$ for all $P\in \mathcal{A}$, $k\in \mathcal{K}_t$ and $1\leq t\leq s$.  Then we have
  \begin{equation}\label{eq:pro:variation-norm:2}
    \|D\colon Z_V\to Z_V\|
    \leq 2\cdot \|D\|_{\mathcal{K}_1, \dots, \mathcal{K}_s}.
  \end{equation}

  If we additionally assume
  \begin{equation}\label{eq:pro:variation-norm:3}
    d_P^{k_t}\neq 0
    \quad\text{and}\quad
    0 < \delta \leq |d_{[0,1)}^{k_t}|,
    \qquad P\in \mathcal{A},\ 1\leq t\leq s,
  \end{equation}
  and that for some $0 < \eta < 1$ we have
  \begin{equation}\label{eq:pro:variation-norm:4}
    |d_{[0,1)}^{k_t} - d_P^{k_t}|
    \leq \eta |d_{[0,1)}^{k_t}|,
    \qquad P\in \mathcal{A},\ 1\leq t\leq s,
  \end{equation}
  then
  \begin{equation}\label{eq:pro:variation-norm:5}
    \|D^{-1}\|_{\mathcal{K}_1, \dots, \mathcal{K}_s}
    \leq \frac{1}{(1-\eta)^2 \delta^2}\cdot \|D\|_{\mathcal{K}_1, \dots, \mathcal{K}_s}.
  \end{equation}
\end{pro}

\begin{proof}
  Let $f = \sum_{P,k} a_P^k \varphi_P^k$ denote a finite sum and observe that by \Cref{dfn:substriction}
  \begin{align*}
    \|D f\|_{Z_V}
    &= \bigl\| \sum_{t=1}^s \sum_{k\in \mathcal{K}_t} \sum_{P\in \mathcal{A}}
      \bigl( d_{[0,1)}^{k_t} - \sum_{\substack{R\in \mathcal{A}:\\P\subset R\subsetneq [0,1)}} (d_{\pi(R)}^{k_t} - d_R^{k_t}) \bigr)
    a_P^k \varphi_P^k
    \bigr\|_{Z_V}\\
    &\leq \sum_{t=1}^s |d_{[0,1)}^{k_t}|
      \bigl\|
      \sum_{k\in \mathcal{K}_t} \sum_{P\in \mathcal{A}} a_P^k \varphi_P^k
      \bigr\|_{Z_V}\\
    &\qquad + \sum_{t=1}^s \sum_{j=1}^{\infty} 
      \bigl\| \sum_{R\in \mathcal{AF}_j} (d_{\pi(R)}^{k_t} - d_R^{k_t})\cdot \Theta_RQ_{\mathcal{K}_t} f \bigr\|_{Z_V}.
  \end{align*}
  Note that since the support of $\Theta_RQ_{\mathcal{K}_t} f$ is contained in $R$, the properties of the norm of $Z_V$ give us
  \begin{align*}
    \bigl\| \sum_{R\in \mathcal{AF}_j} (d_{\pi(R)}^{k_t} - d_R^{k_t})\cdot \Theta_RQ_{\mathcal{K}_t} f \bigr\|_{Z_V}
    &\leq \max_{R\in \mathcal{AF}_j} |d_{\pi(R)}^{k_t} - d_R^{k_t}|\cdot \bigl\| \sum_{R\in \mathcal{AF}_j} \Theta_RQ_{\mathcal{K}_t} f \bigr\|_{Z_V}\\
    &= \max_{R\in \mathcal{AF}_j} |d_{\pi(R)}^{k_t} - d_R^{k_t}|\cdot \| Q_{\mathcal{K}_t} f - \cond(Q_{\mathcal{K}_t} f | \mathcal{F}_j) \|_{Z_V}.
  \end{align*}
  By \Cref{rem:conditional}, the above estimates yield~\eqref{eq:pro:variation-norm:2}.

  To prove the additional claim, observe that by~\eqref{eq:pro:variation-norm:1}
  \begin{equation}\label{eq:32}
    \|D^{-1}\|_{\mathcal{K}_1, \dots, \mathcal{K}_s}
    = \sum_{t=1}^{s} \|Q_{\mathcal{K}_t}\colon Z_V\to Z_V\|\cdot
    \bigl(
    \bigl|\frac{1}{d_{[0,1)}^{k_t}}\bigr| + \sum_{j=1}^{\infty} \max_{P\in \mathcal{AF}_j} \bigl|\frac{1}{d_{\pi(P)}^{k_t}} - \frac{1}{d_P^{k_t}}\bigr|
    \bigr).
  \end{equation}
  Using~\eqref{eq:pro:variation-norm:4}, we note that
  \begin{equation*}
    \bigl|\frac{1}{d_{\pi(P)}^{k_t}} - \frac{1}{d_P^{k_t}}\bigr|
    = \frac{|d_P^{k_t} - d_{\pi(P)}^{k_t}|}{|d_{\pi(P)}^{k_t} d_P^{k_t}|}
    \leq \frac{|d_P^{k_t} - d_{\pi(P)}^{k_t}|}{(1 - \eta)^2|d_{[0,1)}^{k_t}|^2}.
  \end{equation*}
  Inserting the latter estimate into~\eqref{eq:32} and exploiting~\eqref{eq:pro:variation-norm:3} gives us~\eqref{eq:pro:variation-norm:5}.
\end{proof}

%
\section{The Vandermonde technique}
\label{sec:vandermonde-approach}

In this section, we establish a key ingredient for the stabilization procedure presented in \Cref{sec:stabilization}.  The underlying technique may be
viewed as a generalization of the classical analysis of the Vandermonde matrix.

\begin{thm}\label{thm:vantrick1}
  Let $X$ be a topological vector space and assume $Y$ is a topological space.  Let $S_i\colon X \rightarrow X$, $1\leq i\leq d$, be such that
  \begin{enumerate}[(i)]
  \item\label{enu:th:vantrick1:i} each $S_i$, $1\leq i\leq d$ is a continuous linear operator,
  \item\label{enu:th:vantrick1:ii} $\{S_i : 1\leq i\leq d\}$ is a family of mutually commuting operators, i.e.,
    \begin{equation}\label{eq:commssij}
      S_i S_j
      = S_j S_i,
      \qquad 1\leq i,j\leq d.
    \end{equation}
  \end{enumerate}
  Now let $g_k\colon Y \rightarrow X$ denote a function for $1\leq k\leq d$ and define $T_j\colon Y\rightarrow X$ by
  \begin{equation}\label{eq:77}
    T_j
    = \sum_{k=1}^d S^{j}_k g_k,
    \qquad 0\leq j\leq d.
  \end{equation}
  Suppose that
  \begin{enumerate}[(i)]\setcounter{enumi}{2}
  \item\label{enu:th:vantrick1:iii} the functions $T_j$, $0\leq j\leq d$ are all continuous, and
  \item\label{enu:th:vantrick1:iv} there exists $1\leq k_0\leq d$ such that $(S_{k_0}-S_j)^{-1}\colon X \rightarrow X$ is a continuous linear operator for every
    $j\in \{1, \dots, d\}\setminus\{k_0\}$.
  \end{enumerate}
  Then $g_{k_0}\colon Y \rightarrow X$ is a continuous function.
\end{thm}

\begin{proof}
  Suppose that the continuous linear operator $R\colon X \rightarrow X$ commutes with each $S_i$, $1\leq i\leq d$, and define $L_{k_0}(R)\colon X\rightarrow X$ by
  \begin{equation*}
    L_{k_0}(R)
    = \prod_{j:j\neq k_0}(S_{k_0}- S_j)^{-1} \prod_{j:j\neq k_0} (R-S_j).
  \end{equation*}
  Note that
  \begin{equation}\label{eq:LSj}
    L_{k_0}(S_j)
    =
    \begin{cases}
      \Id_X &\text{if $j=k_0$},\\
      0 &\text{if $j\neq k_0$}.   
    \end{cases}
  \end{equation}
  By~\eqref{eq:commssij} and the linearity of $R$ and $S_i$, $1\leq i\leq d$, there exist continuous linear operators $N_j\colon X \rightarrow X$, $0\leq j\leq d$ such that
  \begin{equation}\label{eq:polforL}
    L_{k_0}(R)
    = \sum_{j=0}^{d} N_j R^j.
  \end{equation}
  Observe that by the linearity of the operators $N_j$, $0\leq j\leq d$, \eqref{eq:polforL} and~\eqref{eq:LSj}, we obtain
  \begin{equation*}
    \begin{split}
      \sum_{j=0}^{d} N_j T_j
      &= \sum_{j=0}^d N_j \sum_{k=1}^d S^{j}_k g_k
        = \sum_{j=0}^d\sum_{k=1}^d N_j S^{j}_k g_k
        = \sum_{k=1}^d \left(\sum_{j=0}^d N_j S_k^j\right) g_k\\
      &= \sum_{k=1}^d L_{k_0}(S_k) g_k= g_{k_0}.
    \end{split}
  \end{equation*}
  Clearly $\sum_{j=0}^{d} N_j T_j$ is continuous, and thus $g_{k_0}$ is as well.
\end{proof}

%
\section{The Gamlen-Gaudet construction}
\label{sec:gaml-gaud-constr}

In this section, we establish the conditions under which the Gamlen-Gaudet construction will be applied.  These conditions are subsequently used in
\Cref{sec:fact-oper} to construct bounded embedding and projection operators, which are fundamental to the stabilization procedure developed in
\Cref{sec:stabilization}.

Historically, in the one-parameter setting, the compatibility conditions required for the Gamlen-Gaudet construction have typically been tailored to a
specific Banach space and the particular application under consideration, and consequently vary across the literature (see,
e.g.,~\cite{MR0328575,MR0527010,MR0557491,MR0586594,MR0879418,MR0796906,MR3819715}).  As in~\cite{MR4894823}, the compatibility conditions introduced
below are designed to apply simultaneously to a broad class of Banach spaces.

\begin{dfn}\label{dfn:compatibility}
  Suppose that $0\leq \kappa_P< 1$, $P\in \mathcal{A}$.  A sequence of collections $(\mathcal{B}_P : P\in \mathcal{A})$ is called
  \emph{$(1-\kappa_P)$-faithful} if the following three conditions are satisfied:
  \begin{enumerate}[(C1)]
  \item\label{enu:dfn:compatibility:i} For each $P\in \mathcal{A}$, the non-empty collection $\mathcal{B}_P$ consists of finitely many pairwise disjoint sets in $\mathcal{A}$.
  \item\label{enu:dfn:compatibility:ii} For all $P_0\in \mathcal{A}$, $P_1\in \mathcal{S}(P_0)$, $Q\subset P_1$, $S\in \mathcal{B}_Q$ there exists
    $R_0\in \mathcal{B}_{P_0}$ and $R_1\in \mathcal{S}(R_0)$ with $\ord(R_1) = \ord(P_1)$ such that $S\subset R_1$.
  \item\label{enu:dfn:compatibility:iii} For all $P,Q\in \mathcal{A}$ with $Q\subset P$ and for every $R\in \mathcal{B}_P$, we have
    \begin{equation}\label{eq:33}
      (1 - \kappa_P) \frac{|Q|}{|P|}
      \leq \frac{|R\cap \mathcal{B}_Q^{*}|}{|R|}
      \leq \frac{|Q|}{|P|}.
    \end{equation}

  \end{enumerate}

  Let $(n_i)_{i=0}^{\infty}$ denote an increasing sequence in $\mathbb{N}$.  A $(1-\kappa_P)$-faithful sequence $(\mathcal{B}_P : P\in \mathcal{A})$ which additionally satisfies
  \begin{enumerate}[(C1)]
    \setcounter{enumi}{3}
  \item\label{enu:dfn:compatibility:iv} $\mathcal{B}_P\subset \mathcal{AF}_{n_i}$ for all $P\in \mathcal{AF}_i$ and $i\in \mathbb{N}_0$
  \end{enumerate}
  is called \emph{$(1-\kappa_P)$-faithful with respect to $(n_i)_{i=0}^{\infty}$}.

  Let $U$ denote a set of positive measure.  A $(1-\kappa_P)$-faithful sequence $(\mathcal{B}_P : P\in \mathcal{A})$ (with respect to
  $(n_i)_{i=0}^{\infty}$) which additionally satisfies
  \begin{enumerate}[(C1)]
    \setcounter{enumi}{4}
  \item\label{enu:dfn:compatibility:v} $|U\setminus \mathcal{B}_{[0,1)}^{*}| \leq \kappa_{[0,1)} |\mathcal{B}_{[0,1)}^{*}|$ as well as
    $|R\cap U| \geq (1 - \kappa_P) |R|$ for every $P\in \mathcal{A}$ and $R\in \mathcal{B}_P$
    
  \end{enumerate}
  is called a \emph{$(1-\kappa_P)$-faithful sequence (with respect to $(n_i)_{i=0}^{\infty}$) mounted on $U$}.

  Whenever $\kappa_P = 0$, $P\in \mathcal{A}$, we replace ``$(1-\kappa_P)$-faithful'' simply by ``faithful''.
\end{dfn}

\begin{rem}\label{rem:faithful-mount}
  If a faithful sequence of collections $(\mathcal{B}_P : P\in \mathcal{A})$ is mounted on $U$, then $\mathcal{B}_{[0,1)}^{*} = U$.
\end{rem}

\begin{rem}\label{rem:block-basis}
  First we observe that given $P,Q\in \mathcal{A}$, \textrefp[C]{enu:dfn:compatibility:ii} gives us $\mathcal{B}_Q^{*}\subset \mathcal{B}_P^{*}$ if
  and only if $Q\subset P$.  Secondly, note that multiplying~\eqref{eq:33} by $|R|$, summing over all $R\in \mathcal{B}_P$ and using the first
  observation yields
  \begin{equation}\label{eq:34}
    (1 - \kappa_P) \frac{|Q|}{|P|} |\mathcal{B}_P^{*}|
    \leq |\mathcal{B}_Q^{*}|
    \leq \frac{|Q|}{|P|} |\mathcal{B}_P^{*}|.
  \end{equation}
  Hence, for a faithful sequence of collections, we obtain $|\mathcal{B}_Q^{*}| = \frac{|Q|}{|P|} |\mathcal{B}_P^{*}|$.
\end{rem}

\begin{rem}\label{rem:C2}
  Given $P_0\in \mathcal{A}$, $P_1\in \mathcal{S}(P_0)$ we define the sets $\Gamma_{P_0}^{P_1}$ by
  \begin{equation}\label{eq:135}
    \Gamma_{P_0}^{P_1}
    := \bigcup \bigl\{ S\in \mathcal{S}(R) : R\in \mathcal{B}_{P_0},\ \ord(S) = \ord(P_1) \bigr\},
  \end{equation}
  and note that
  \begin{equation}\label{eq:134}
    \cha_{\mathcal{B}_{P_0}^{*}}
    = \sum_{P_1\in \mathcal{S}(P_0)} \cha_{\Gamma_{P_0}^{P_1}}
    \qquad\text{and}\qquad
    |\mathcal{B}_{P_0}^{*}| = q\cdot |\Gamma_{P_0}^{P_1}|.
  \end{equation}

  Moreover, \textrefp[C]{enu:dfn:compatibility:ii} gives us
  \begin{equation}\label{eq:128}
    \mathcal{B}_Q^{*}
    \subset \Gamma_{P_0}^{P_1}\subset \mathcal{B}_{P_0}^{*}
    \quad\text{if and only if}\quad
    Q\subset P_1,
  \end{equation}
  and thus
  \begin{equation}\label{eq:137}
    \bigcup_{\substack{Q\in \mathcal{AF}_N\\Q\subset P_0}} \mathcal{B}_Q^{*}
    \subset \mathcal{B}_{P_0}^{*}.
  \end{equation}

  In case we have a faithful system, \Cref{rem:block-basis} gives us $|\Gamma_{P_0}^{P_1}| = |\mathcal{B}_{P_1}^{*}|$, and thus~\eqref{eq:128}
  and~\eqref{eq:134} together with \textrefp[C]{enu:dfn:compatibility:i} yield
  \begin{equation}\label{eq:136}
    \mathcal{B}_{P_1}^{*}
    = \Gamma_{P_0}^{P_1},
  \end{equation}
  and consequently, by iterating~\eqref{eq:134} and~\eqref{eq:136}, we obtain
  \begin{equation}\label{eq:139}
    \sum_{\substack{Q\in \mathcal{AF}_N\\Q\subset P_0}} \cha_{\mathcal{B}_Q^{*}}
    = \cha_{\mathcal{B}_{P_0}^{*}}
  \end{equation}
  whenever $P_0$ is $\mathcal{F}_N$-measurable.
\end{rem}

\begin{figure}
  \centering
  \begin{tikzpicture}
    \node[anchor=south west,inner sep=0] (img) {\includegraphics[width=.8\textwidth]{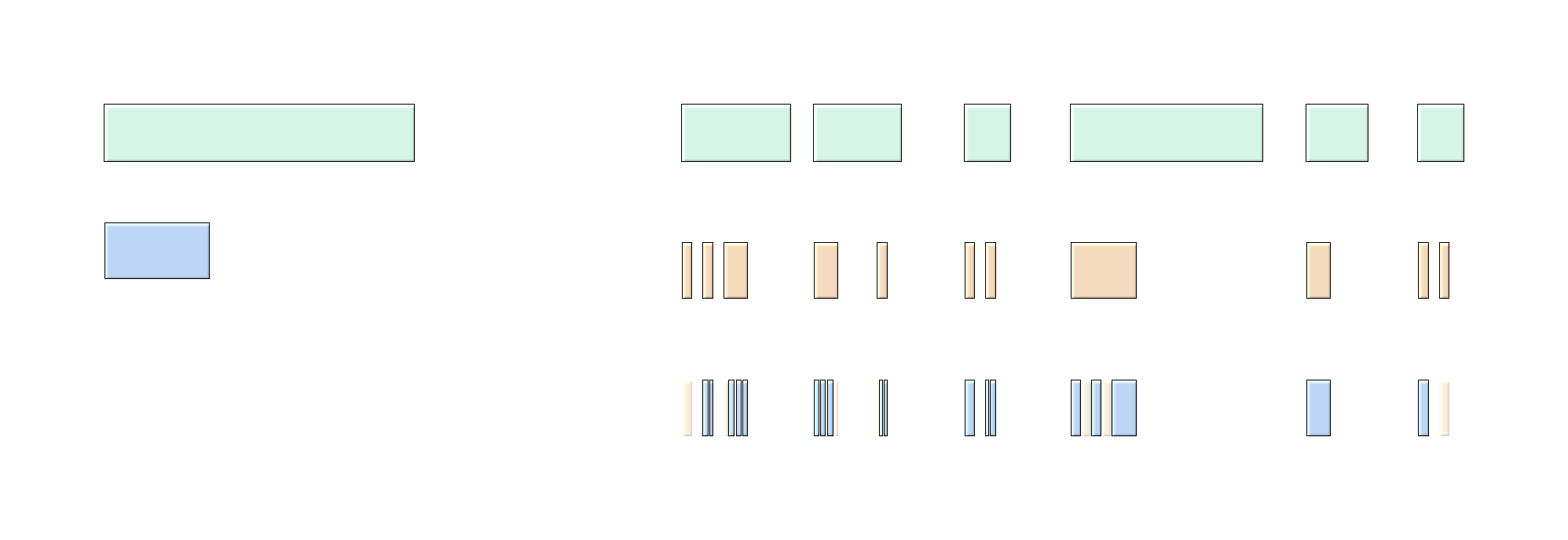}};
    \begin{scope}[x={(img.south east)},y={(img.north west)}]
      \node at (0.02,0.75) {{\small $P_0$}};%
      \node at (0.02,0.53) {{\small $P_1$}};%
      \node at (0.38,0.75) {{\small $\mathcal{B}_{P_0}$}};%
      \node at (0.38,0.5) {{\small $\Gamma_{P_0}^{P_1}$}};%
      \node at (0.38,0.23) {{\small $\mathcal{B}_{P_1}$}};%
    \end{scope}
  \end{tikzpicture}
  \caption{This figure illustrates the relationship between the collections $\mathcal{B}_{P_0}$ and $\mathcal{B}_{P_1}$, $P_0\in \mathcal{A}$,
    $P_1\in \mathcal{S}(P_0)$.  The point set $\Gamma_{P_0}^{P_1}$ is defined as the union of the successors of the intervals in $\mathcal{B}_{P_0}$ that have the same order as
    $P_1$.  The portions of $\Gamma_{P_0}^{P_1}$ that are not covered by $\mathcal{B}_{P_1}$ are highlighted in a very light orange.}
  \label{fig:gamlen-gaudet}
\end{figure}

\begin{rem}\label{rem:Gamma}
  In the next section, we will associate to the collection $\mathcal{B}_P$, $P\in \mathcal{A}$ the functions $\psi_P^k$, $1\leq k\leq r$, which in turn determine our factorization
  operators $A$ and $B$ (see~\eqref{dfn:canonical} and~\eqref{eq:221}).

  The sets $\Gamma_{P_0}^{P_1}$, $P_1\in\mathcal{S}(P_0)$, defined in~\eqref{eq:135}, each have density $|P_1|/|P_0|$ in every
  $R\in\mathcal{B}_{P_0}$.  Moreover, they form a partition of the support of $\psi_{P_0}^k$ such that $\psi_{P_0}^k$ is constant on each of them.  See
  \Cref{fig:gamlen-gaudet} for a depiction.
\end{rem}

\begin{rem}\label{rem:faithful-extension}
  We note that any $(1 - \kappa_P)$-faithful collection $(\mathcal{B}_P : P\in \mathcal{A})$ can be extended to a faithful collection
  $(\hat{\mathcal{B}}_P : P\in \mathcal{A})$ mounted on $\mathcal{B}_{[0,1)}^{*}$ such that $\mathcal{B}_P \subset \hat{\mathcal{B}}_P$ for any
  $P\in\mathcal{A}$.  We call $(\hat{\mathcal{B}}_P : P\in \mathcal{A})$ a \emph{faithful extension} of $(\mathcal{B}_P : P\in \mathcal{A})$.
\end{rem}

\begin{lem}\label{lem:faithful-extension}
  Let $(\mathcal{B}_P : P\in \mathcal{A})$ denote a $(1 - \kappa_P)$-faithful collection and let $(\hat{\mathcal{B}}_P : P\in \mathcal{A})$ be its faithful extension.  Given
  $N\in \mathbb{N}_0$, we define $\mathcal{G}_N$ as the $\sigma$-algebra generated by
  \begin{equation}\label{eq:144}
    \mathcal{C}_N
    := \bigcup_{j=0}^N \bigcup_{P\in\mathcal{AF}_j} \bigcup_{R\in \mathcal{B}_{P}} \mathcal{S}(R).
  \end{equation}
  Then for any $R\in \bigcup_{j=0}^{N} \bigcup_{P\in \mathcal{AF}_j} \hat{\mathcal{B}}_P\setminus \mathcal{B}_P$ there exists an atom $A$ of
  $\mathcal{G}_N$ such that $R\subset A$.
\end{lem}

\begin{proof}
  Let $S\in \mathcal{C}_N$ and suppose that $S\cap R \neq \emptyset$.  Then clearly, there exists a $Q\in \mathcal{AF}_j$, $0\leq j\leq N$ and
  $R'\in \mathcal{B}_Q$ such that $S\in \mathcal{S}(R')$.  In particular, $\mathcal{B}_Q^{*}\cap \mathcal{B}_P^{*} \neq \emptyset$, and thus, by
  \Cref{rem:block-basis} we have that either $Q\subset P$ or $P\subsetneq Q$.  If $Q\subset P$, we use \Cref{rem:block-basis} again and obtain that
  $\hat{\mathcal{B}}_Q^{*}\subset \hat{\mathcal{B}}_P^{*}$.  Since $S\subset \mathcal{B}_Q^{*}$ and
  $R\subset \hat{\mathcal{B}}_{P}^{*}\setminus \mathcal{B}_P^{*}$, this leads to the contradiction
  \begin{equation}\label{eq:147}
    \emptyset
    \neq S\cap R \subset \mathcal{B}_Q^{*}\cap \bigl( \hat{\mathcal{B}}_{P}^{*}\setminus \mathcal{B}_P^{*} \bigr)
    \subset \mathcal{B}_P^{*}\cap \bigl( \hat{\mathcal{B}}_{P}^{*}\setminus \mathcal{B}_P^{*} \bigr)
    = \emptyset.
  \end{equation}
  Thus, we can assume that $P\subsetneq Q$.  By~\textrefp[C]{enu:dfn:compatibility:i} and~\textrefp[C]{enu:dfn:compatibility:ii}, there exists a
  unique $R''\in \hat{\mathcal{B}}_Q$ such that $R\subsetneq R''$.  Moreover, appealing to~\textrefp[C]{enu:dfn:compatibility:i} again, we find that
  $R'' = R'$.  Combining the latter facts and using the nestedness of $\mathcal{A}$ yields $R\subset S$.  To summarize, for every $S\in \mathcal{C}_N$
  either $R\subset S$ or $R\subset [0,1)\setminus S$, and hence
  \begin{equation}\label{eq:148}
    R\subset
    \Bigl(
    \bigcap_{\substack{S\in \mathcal{C}_N\\S\cap R\neq \emptyset}} S
    \cap \bigcap_{\substack{S\in \mathcal{C}_N\\S\cap R= \emptyset}} [0,1)\setminus S
    \Bigr)
    =: A.
  \end{equation}
  The claim follows from the well known fact that $A$ is an atom of $\mathcal{G}_N$.
\end{proof}

\begin{lem}\label{lem:block-basis}
  Let $U$ denote a set of positive measure, suppose that $0\leq \kappa_P < 1$, $P\in \mathcal{A}$, and let $(n_i)_{i=0}^{\infty}$ denote any
  increasing sequence in $\mathbb{N}$.  Then there exists a subsequence $(m_i)_{i=0}^{\infty}$ of $(n_i)_{i=0}^{\infty}$ and a sequence of collections
  $(\mathcal{B}_P : P\in \mathcal{A})$ which is a $(1-\kappa_P)$-faithful sequence with respect to $(m_i)_{i=0}^{\infty}$ mounted on $U$.
\end{lem}

\begin{proof}[Proof of \Cref{lem:block-basis}]
  Choose any numbers $\kappa_i, \widetilde{\kappa}_P > 0$ such that
  \begin{equation}\label{eq:49}
    \kappa_0
    \leq |U|/2,
    \qquad
    \kappa_i
    < 1/q,
    \qquad
    \kappa_i
    \leq \widetilde{\kappa}_P \frac{|Q|}{|P|},
  \end{equation}
  for all $P\in \mathcal{AF}_i$, $Q\in \mathcal{S}(P)$, $i\in \mathbb{N}_0$, and
  \begin{equation}\label{eq:51}
    \sum_{\substack{P'\in \mathcal{A}:\\Q\subsetneq P'\subset P}} \widetilde{\kappa}_{P'}
    \leq \kappa_P,
    \qquad P,Q\in \mathcal{A},\ Q\subset P.
  \end{equation}

  Note that by the martingale convergence theorem $\cond(\cha_U | \mathcal{F}_{n_i}) \to \cha_U$ almost everywhere.  By Egorov's theorem, for every
  $\kappa > 0$ there exists a measurable subset $V$ of $U$ with $|U\setminus V|\leq \kappa$ such that this convergence is uniform on $V$, i.e.,
  \begin{equation*}
    \lim_i \sup_{x\in V} \bigl|\cond(\cha_U | \mathcal{F}_{n_i})(x) - 1 \bigr|
    = 0.
  \end{equation*}
  Hence, switching from $\cha_U$ to $\cha_V$, we obtain
  \begin{equation}\label{eq:14}
    \limsup_i \sup_{x\in V} \bigl|\cond(\cha_V | \mathcal{F}_{n_i})(x) - 1 \bigr|
    \leq \kappa.
  \end{equation}
  Using~\eqref{eq:14}, we will now inductively choose a subsequence $(m_i)_{i=0}^{\infty}$ of $(n_i)_{i=0}^{\infty}$ with $m_0 > 1$ and an increasing sequence
  $(V_i)_{i=0}^{\infty}$ of subsets of $U$ such that
  \begin{align}
    \cond(\cha_{V_0} | \mathcal{F}_{m_0-1})(x)
    &\geq 1-\kappa_0,
    && x\in V_0,\label{eq:11:a}\\
    \cond(\cha_{V_i} | \mathcal{F}_{m_i})(x)
    &\geq 1-\kappa_i,
    && x\in V_i,\ i\in \mathbb{N}_0,\label{eq:11:b}
       \intertext{as well as}
       |U\setminus V_0|
    &\leq \kappa_0,
    &&|U\setminus V_i|
       \leq \kappa_i |R|,\label{eq:11:c}
  \end{align}
  for all $R\in \mathcal{F}_{m_{i-1}}$, $i\in \mathbb{N}$.  We now define the collections $(\mathcal{B}_i)_{i=0}^{\infty}$ by
  \begin{align}
    \mathcal{B}_0
    &= \{ R\in \mathcal{AF}_{m_0} : R\cap V_0\neq \emptyset \},\label{eq:40:a}\\
    \mathcal{B}_i
    &= \{ R\in \mathcal{AF}_{m_i} : R\cap V_i\neq \emptyset,\ R\cap \mathcal{B}_{i-1}^{*}\neq \emptyset \},
      \qquad i\in \mathbb{N}.\label{eq:40:b}
  \end{align}
  Note that, by~\eqref{eq:11:a}, \eqref{eq:11:b} and~\eqref{eq:11:c}, each of the collections $\mathcal{B}_i$, $i\in \mathbb{N}_0$ has the following properties:
  \begin{enumerate}[(a)]
  \item\label{enu:proof:pro:stab-3:a} for every $S\in \mathcal{B}_{i+1}$ there exists a $R\in \mathcal{B}_i$ such that $S\subsetneq R$, and
  \item\label{enu:proof:pro:stab-3:b} $|R\cap V_i| \geq (1-\kappa_i)|R|$ for all $R\in \mathcal{B}_i$.
  \end{enumerate}
  Observe that the collections $(\mathcal{B}_i)_{i=0}^{\infty}$ are non-empty; in fact, by~\eqref{eq:49}, $V_0$ has positive measure, and since
  $V_i \supset V_0$ for all $i \in \mathbb{N}_0$, we obtain from~\eqref{eq:40:a} and~\eqref{eq:40:b} that
  \begin{equation*}
    \mathcal{B}_i
    \supset \{ R\in \mathcal{AF}_{m_i} : R\cap V_0\neq \emptyset \}
    \neq \emptyset.
  \end{equation*}
  Suppose that if $R\in \mathcal{A}$ and $|R\cap U|\geq (1 - \alpha) |R|$ for $\alpha > 0$.  Then for every $S\in \mathcal{S}(R)$ we have
  \begin{equation}\label{eq:9}
    |S\cap U|
    = |S| - |S\setminus U|
    \geq |S| - |R\setminus U|
    \geq |S| - \alpha|R|
    \geq (1 - q \alpha) |S|.
  \end{equation}

  Next, we will use the collections $(\mathcal{B}_i)_{i=0}^{\infty}$ to inductively define $(\mathcal{B}_P : P\in \mathcal{A})$.  For the initial step, we simply put
  \begin{equation}\label{eq:42}
    \mathcal{B}_{[0,1)}
    = \mathcal{B}_0.
  \end{equation}
  Note that by~\eqref{eq:49}, \eqref{eq:11:a} and~\eqref{eq:9}, the collection $\mathcal{B}_{[0,1)}$ is non-empty.
  
  Now assume we have already constructed $\mathcal{B}_P$ for all $P\in \bigcup_{i=0}^{j-1} \mathcal{AF}_i$ and they are non-empty.  We then fix
  $P\in \mathcal{AF}_{j-1}$, $Q\in \mathcal{S}(P)$, define the point set
  \begin{equation}\label{eq:43}
    \Gamma_P^Q
    = \bigcup \bigl\{ S\in \mathcal{S}(R) : R\in \mathcal{B}_P,\ \ord(S) = \ord(Q) \bigr\},
  \end{equation}
  and then complete the inductive step by putting
  \begin{equation}\label{eq:44}
    \mathcal{B}_Q
    = \{R\in \mathcal{B}_j : R\subset \Gamma_P^Q\}.
  \end{equation}
  Observe that, in view of~\eqref{eq:49}, \eqref{eq:11:b}, and~\eqref{eq:9}, each of the collections $\mathcal{B}_Q$, $Q\in \mathcal{AF}_j$ is non-empty.
  Moreover, from~\eqref{eq:42} and iterating~\eqref{eq:43} and~\eqref{eq:44}, we obtain
  \begin{equation}\label{eq:10}
    \mathcal{B}_Q^{*}
    \subset \mathcal{B}_P^{*},
    \qquad P,Q\in \mathcal{A},\ Q\subset P.
  \end{equation}

  We will now proceed to verify~\textrefp[C]{enu:dfn:compatibility:i}--\textrefp[C]{enu:dfn:compatibility:v}.  By~\eqref{eq:40:a}, \eqref{eq:40:b} and
  since each $\mathcal{F}_{m_i}$, $i\in \mathbb{N}_0$ is finite, \textrefp[C]{enu:dfn:compatibility:i} is satisfied.
  Condition~\textrefp[C]{enu:dfn:compatibility:ii} follows by successive application of Property~\eqref{enu:proof:pro:stab-3:a}.  The
  conditions~\textrefp[C]{enu:dfn:compatibility:iv} and~\textrefp[C]{enu:dfn:compatibility:v} are an immediate consequence of~\eqref{eq:40:a}
  and~\eqref{eq:40:b}, Property~\eqref{enu:proof:pro:stab-3:b}, \eqref{eq:42}, \eqref{eq:44}, \eqref{eq:49}, \eqref{eq:51} and the fact that
  $V_i\subset U$.  The remainder of the proof is devoted to establishing condition~\textrefp[C]{enu:dfn:compatibility:iii}.

  We fix $P\in \mathcal{AF}_i$ and $R\in \mathcal{B}_P$.  Suppose, to the contrary, that there exists a $Q_0\in \mathcal{S}(P)$ such that
  \begin{equation}\label{eq:45}
    \frac{|R\cap \mathcal{B}_{Q_0}^{*}|}{|R|} < (1 - \widetilde{\kappa}_P) \frac{|Q_0|}{|P|}.
  \end{equation}
  Then by Property~\eqref{enu:proof:pro:stab-3:b}, \eqref{eq:40:a}, \eqref{eq:40:b} as well as the fact that $V_i\subset V_{i+1}$, we obtain
  \begin{equation}\label{eq:46}
    \begin{aligned}
      (1 - \kappa_i) |R|
      &\leq |R\cap V_i|
        = |R\cap V_i\cap \mathcal{B}_{i+1}^{*}| + |R\cap  V_i\setminus (R\cap \mathcal{B}_{i+1}^{*})|\\
      &\leq |R\cap V_i\cap \mathcal{B}_{i+1}^{*}| + |R\cap  V_i\setminus (R\cap V_{i+1})|
        = |R\cap V_i\cap \mathcal{B}_{i+1}^{*}|.
    \end{aligned}
  \end{equation}

  Examining Property~\eqref{enu:proof:pro:stab-3:a}, \eqref{eq:43} and~\eqref{eq:44}, we find that
  $\mathcal{B}_{i+1} = \bigcup_{Q\in \mathcal{AF}_{i+1}} \mathcal{B}_Q$.  Similarly, we obtain that $\mathcal{B}_Q^{*}\subset \mathcal{B}_P^{*}$ if
  and only if $Q\subset P$, and hence,
  \begin{equation}\label{eq:47}
    R\cap V_i\cap \mathcal{B}_{i+1}^{*}
    = \bigcup_{Q\in \mathcal{AF}_{i+1}} R\cap V_i\cap \mathcal{B}_Q^{*}
    = \bigcup_{Q\in \mathcal{S}(P)} R\cap V_i\cap \mathcal{B}_Q^{*}.
  \end{equation}
  Observe that $\mathcal{B}_Q^{*} \subset \Gamma_P^Q$ for all $Q \in \mathcal{S}(P)$, and that the sets $\Gamma_P^Q$, $Q \in \mathcal{S}(P)$, are
  pairwise disjoint.  Consequently, the sets $\mathcal{B}_Q^{*}$, $Q \in \mathcal{S}(P)$, are pairwise disjoint as well.  Using~\eqref{eq:44}
  and~\eqref{eq:43}, we obtain that for each $Q\in \mathcal{S}(P)$
  \begin{equation}\label{eq:48}
    |R\cap \mathcal{B}_Q^{*}|
    \leq |R\cap \Gamma_P^Q|
    = \sum_{\substack{S\in \mathcal{S}(R)\\\ord(S) = \ord(Q)}} |S|
    = \frac{|Q|}{|P|}\cdot |R|.
  \end{equation}
  Before we proceed with the proof, we note that iterating~\eqref{eq:48} establishes the upper estimate in~\eqref{eq:51}.

  Inserting~\eqref{eq:47} into~\eqref{eq:46} and using~\eqref{eq:45} and~\eqref{eq:48} yields
  \begin{align*}
    (1 - \kappa_i) |R|
    &\leq |R\cap V_i\cap \mathcal{B}_{Q_0}^{*}| + \sum_{Q\in \mathcal{S}(P)\setminus\{Q_0\}} |R\cap V_i\cap \mathcal{B}_Q^{*}|\\
    &< (1 - \widetilde{\kappa}_P) \frac{|Q_0|}{|P|}\cdot |R| + \sum_{Q\in \mathcal{S}(P)\setminus\{Q_0\}} \frac{|Q|}{|P|}\cdot |R|\\
    &= \Bigl( 1 - \widetilde{\kappa}_P \frac{|Q_0|}{|P|} \Bigr)\cdot |R|,
  \end{align*}
  which contradicts~\eqref{eq:49}.  We record that we have established
  \begin{equation}\label{eq:50}
    \frac{|R\cap \mathcal{B}_Q^{*}|}{|R|} \geq (1 - \widetilde{\kappa}_P) \frac{|Q|}{|P|},
    \qquad P\in \mathcal{A},\ Q\in \mathcal{S}(P),\ R\in \mathcal{B}_P.
  \end{equation}
  Finally, iterating~\eqref{eq:50} gives us
  \begin{equation}\label{eq:8}
    \frac{|R\cap \mathcal{B}_Q^{*}|}{|R|} \geq \Bigl( 1 - \sum_{\substack{P'\in \mathcal{A}:\\Q\subsetneq P'\subset P}} \widetilde{\kappa}_{P'} \Bigr) \frac{|Q|}{|P|},
    \qquad P,Q\in \mathcal{A},\ Q\subset P,\ R\in \mathcal{B}_P.
  \end{equation}
  Combining the above with~\eqref{eq:51}, we conclude the proof.
\end{proof}

Reviewing the proof of \Cref{lem:block-basis}, we make the following observations.
\begin{rem}\label{rem:block-basis:2}
  Inspecting~\eqref{eq:44} and~\eqref{eq:43} reveals that for faithful collections with respect to $(m_j)_{j=0}^{\infty}$, we have
  $\mathcal{B}_Q^{*} = \Gamma_P^Q$ and that $\Gamma_P^Q\in \mathcal{F}_{m_{j-1}+1}$ for $P\in \mathcal{AF}_{j-1}$, $Q\in \mathcal{S}(P)$ and
  $j\in \mathbb{N}$.  Consequently, we obtain for $i\leq j$ and $P\in \mathcal{AF}_i$ that
  \begin{equation*}
    \mathcal{B}_P^{*}
    = \bigcup_{\substack{Q\in \mathcal{AF}_j\\Q\subset P}} \mathcal{B}_Q^{*}.
  \end{equation*}
  Similarly, by~\eqref{eq:40:a}, $\mathcal{B}_{[0,1)}^{*}\in \mathcal{F}_{m_0}$.  To summarize, $\mathcal{B}_P^{*}\in \mathcal{F}_{m_{j-1}+1}$ if
  $P\in \mathcal{AF}_j$, $j\in \mathbb{N}$, and (up to a set of measure zero)
  \begin{equation*}
    U
    = \bigcup_{P\in \mathcal{AF}_{0}} \mathcal{B}_P^{*}\in \mathcal{F}_{m_0}.
  \end{equation*}
\end{rem}

\begin{rem}\label{rem:block-basis:3}
  Suppose that $U$ is a finite union of atoms in $\mathcal{A}$.  Then the proof of \Cref{lem:block-basis} produces a faithful sequence of collections
  $(\mathcal{B}_P : P\in \mathcal{A})$.  This can be demonstrated by first noting that in this case, we can pick $V = U$ in~\eqref{eq:14} and
  $\kappa = 0$.  Consequently, this leads to the choices $V_i = U$ and $\kappa_i = 0$, $i\in \mathbb{N}_0$ in~\eqref{eq:11:a}, \eqref{eq:11:b} and~\eqref{eq:11:c}, which
  ultimately leads to~\eqref{eq:8}.  Finally, by inspecting~\eqref{eq:51}, we confirm the faithfulness of $(\mathcal{B}_P : P\in \mathcal{A})$.

  Additionally, we observe that if $U$ is $\mathcal{F}_j$-measurable for some $j\in \mathbb{N}_0$, then by~\eqref{eq:40:a}, $\mathcal{B}_0^{*}$ is also $\mathcal{F}_j$-measurable.
\end{rem}

%

\section{Factorization operators}
\label{sec:fact-oper}

Here we use the Gamlen-Gaudet construction under the conditions established in \Cref{sec:gaml-gaud-constr} to construct the fundamental factorization
operators $A$ and $B$, which are bounded on Janson-Sobolev and Janson-Sobolev-Hardy spaces under these assumptions.  These operators are essential for
the stabilization procedure developed in \Cref{sec:stabilization}.

\begin{dfn}\label{dfn:canonical}
  Given a sequence of finite, pairwise disjoint collections $(\mathcal{B}_P : P\in \mathcal{A})$ with $\mathcal{B}_P\subset \mathcal{A}$, $P\in \mathcal{A}$, we define the block basis functions
  \begin{equation}\label{eq:17}
    \psi_P^k
    = \sum_{R\in \mathcal{B}_P} \varphi_R^k,
    \qquad P\in \mathcal{A},\ 1\leq k\leq r.
  \end{equation}
  Associated to such a block basis are the \emph{canonical factorization operators} $A,B\colon Z_V\to Z_V$ formally defined by
  \begin{equation}\label{eq:221}
    A f
    = \sum_{P\in \mathcal{A}} \sum_{k=1}^{r} \frac{\langle \psi_P^k, f \rangle}{|\mathcal{B}_P^{*}|} \varphi_P^k
    \qquad\text{and}\qquad
    B f
    = \sum_{P\in \mathcal{A}} \sum_{k=1}^{r} \frac{\langle \varphi_P^k, f \rangle}{|P|} \psi_P^k.
  \end{equation}
  To distinguish, faithful sequences $(\mathcal{B}_P : P\in \mathcal{A})$ will be denoted by $(\hat{\mathcal{B}}_P : P\in \mathcal{A})$ and the associated canonical factorization operators
  $A$ and $B$ will be denoted by $\hat{A}$ and $\hat{B}$.
\end{dfn}

First we will study the properties of the canonical factorization operators in the case of faithful sequences.
\begin{thm}\label{thm:canonical}
  Let $U\subset [0,1)$ be a set of positive measure, and suppose that $(\hat{\mathcal{B}}_P : P\in \mathcal{A})$ is a faithful sequence mounted on
  $U$.  Then $\hat{A}, \hat{B}\colon Z_V\to Z_V$ are well defined bounded operators and we have
  \begin{equation}\label{eq:212}
    \hat{A}\hat{B}
    = \Id_{Z_V},
    \qquad
    \|\hat{A}\|
    \leq \gamma_X^u(|\hat{\mathcal{B}}_{[0,1)}^{*}|),
    \qquad
    \|\hat{B}\|
    \leq \gamma_X^d(|\hat{\mathcal{B}}_{[0,1)}^{*}|),
  \end{equation}
  where $\gamma_X^u$ and $\gamma_X^d$ are defined in \Cref{dfn:scaling}.
\end{thm}

\begin{myproof}
  Before we begin with the proof, recall that by \Cref{rem:faithful-mount}, any faithful sequence $(\hat{\mathcal{B}}_P : P\in \mathcal{A})$ mounted on
  $U$ forces $\hat{\mathcal{B}}_{[0,1)}^{*}= U$.
  
  We fix a function $f=\sum_{j=0}^{\infty} \sum_{P\in \mathcal{AF}_j} \sum_{k=1}^r a_P^k \varphi_P^k$ with only finitely many non-zero coefficients $a_P^k$.
  \begin{proofcase}[Estimates for $\hat{B}$]
    By~\eqref{eq:221} and~\eqref{eq:17}, we obtain
    \begin{equation}\label{eq:imageofB}
      \hat{B}f
      = \sum_{j=0}^{\infty}  \sum_{P\in \mathcal{AF}_j} \sum_{k=1}^r a_P^k \hat{\psi}_P^k
      = \sum_{j=0}^{\infty}  \sum_{P\in \mathcal{AF}_j} \sum_{k=1}^r a_P^k \sum_{Q\in \hat{\mathcal{B}}_P} \varphi_Q^k.
    \end{equation}
    Since $(\hat{\mathcal{B}}_P : P\in \mathcal{A})$ is faithful, for every point $x\in U$ there exists a sequence of sets $(P_j(x))_{j=1}^{\infty}$ such that for
    $j\in\mathbb{N}_0$ we have $P_j(x)\in\mathcal{AF}_j$ and there exists a unique $Q_j(x)\in \hat{\mathcal{B}}_{P_j(x)}$ satisfying $x\in Q_j(x)$ and
    $Q_j(x)\in \mathcal{AF}_{t(j,x)}$ for some $t(j,x)\in \mathbb{N}_0$.  Observe that for fixed $x\in [0,1)$, we have the following identity:
    \begin{equation}\label{eq:pointwiseestimatebfrad}
      \hat{B} f(x)
      = \sum_{j=0}^{\infty}  \Delta_j \hat{B}f (x)
      = \sum_{j=0}^{\infty}  \Delta_{t(j,x)} \hat{B}f (x)
      = \sum_{j=0}^{\infty} \sum_{k=1}^r a_{P_j(x)}^k  \varphi_{Q_j(x)}^k(x),
    \end{equation}
    where $\Delta_j$ is defined in~\eqref{eq:24}.  Note that $\hat{B} f(x) = 0$ for all $x\in [0,1)\setminus U$.  By the definition of the
    norm~\eqref{eq:39} of $Z_V$, the above identity gives us
    \begin{equation}\label{eq:25}
      \begin{aligned}
        \|\hat{B}f\|_{Z_V}
        &= \bigl\|
          x\mapsto \cond_{\boldsymbol{\sigma}} \bigl\|
          \sum_{j=0}^{\infty} \sigma_{t(j,x)} \Delta_{t(j,x)} \hat{B}f (x)
          \bigr\|_{\ell^2_l}
          \bigr\|_X\\
        &= \bigl\|
          x\mapsto \cond_{\boldsymbol{\sigma}} \bigl\|
          \sum_{j=0}^{\infty} \sigma_{t(j,x)} \sum_{P\in \mathcal{AF}_j} \sum_{k=1}^r a_P^k \varphi_P^k(x)
          \bigr\|_{\ell^2_l}
          \bigr\|_X\\
        & = \bigl\|
          x\mapsto \cond_{\boldsymbol{\sigma}} \bigl\|
          \sum_{j=0}^{\infty} \sigma_{t(j,x)}  \sum_{k=1}^r a_{P_j(x)}^k \varphi_{Q_j(x)}^k(x)
          \bigr\|_{\ell^2_l}
          \bigr\|_X.
      \end{aligned}
    \end{equation}
    Since for fixed $x\in U$, the random variables
    \begin{equation}
      \label{eq:28}
      \sum_{j=0}^{\infty} \sigma_{t(j,x)} \sum_{k=1}^r a_{P_j(x)}^k \varphi_{Q_j(x)}^k(x)
      \quad\text{and}\quad
      \sum_{j=0}^{\infty} \sigma_j \sum_{k=1}^r a_{P_j(x)}^k \varphi_{Q_j(x)}^k(x)
    \end{equation}
    have the same distribution, we obtain by~\eqref{eq:25}
    \begin{equation}\label{eq:26}
      \|\hat{B}f\|_{Z_V}
      = \Bigl\|
      x\mapsto \cond_{\boldsymbol{\sigma}} \Bigl\|
      \sum_{j=0}^{\infty} \sigma_j  \sum_{k=1}^r a_{P_j(x)}^k \varphi_{Q_j(x)}^k(x)
      \Bigr\|_{\ell^2_l}
      \Bigr\|_X.
    \end{equation}
    Recalling the definition of $P_j(x)$, $Q_j(x)$, using~\eqref{eq:17} and substitution of these expressions into~\eqref{eq:26} yields
    \begin{equation}\label{eq:31}
      \begin{aligned}
        \|\hat{B}f\|_{Z_V}
        &= \Bigl\|
          x\mapsto \cond_{\boldsymbol{\sigma}} \Bigl\|
          \sum_{j=0}^{\infty} \sigma_{j}  \sum_{P\in \mathcal{AF}_j} \sum_{k=1}^r a_P^k \sum_{Q\in \hat{\mathcal{B}}_P} \varphi_Q^k(x)
          \Bigr\|_{\ell^2_l}
          \Bigr\|_X\\
        &= \Bigl\|
          x\mapsto \cond_{\boldsymbol{\sigma}} \Bigl\|
          \sum_{j=0}^{\infty} \sigma_{j}  \sum_{P\in \mathcal{AF}_j}\sum_{k=1}^r a_{P}^k  \hat{\psi}_{P}^k(x)
          \Bigr\|_{\ell^2_l}
          \Bigr\|_X.
      \end{aligned}
    \end{equation}
    To conclude this part of the proof, first note that since $(\hat{\mathcal{B}}_P : P\in \mathcal{A})$ is faithful, the functions
    \begin{equation}\label{eq:30}
      F
      := \cond_{\boldsymbol{\sigma}} \Bigl\|
      \sum_{j=0}^{\infty} \sigma_{j}  \sum_{P\in \mathcal{AF}_j}\sum_{k=1}^r a_{P}^k  \hat{\psi}_{P}^k
      \Bigr\|_{\ell^2_l}
      \quad\text{and}\quad
      G
      := \cond_{\boldsymbol{\sigma}} \Bigl\|
      \sum_{j=0}^{\infty} \sigma_{j}  \sum_{P\in \mathcal{AF}_j}\sum_{k=1}^r a_{P}^k  \varphi_{P}^k
      \Bigr\|_{\ell^2_l}
    \end{equation}
    satisfy
    \begin{equation}\label{eq:59}
      |\{ |F| > \alpha \}|
      = |\hat{\mathcal{B}}_{[0,1)}^{*}|\cdot |\{ |G| > \alpha \}|,
      \qquad \alpha > 0.
    \end{equation}

    Secondly, exploiting the properties of rearrangement invariant spaces (see \Cref{dfn:r.i.-spaces}) and the definition of the norm~\eqref{eq:39} of
    $Z_V$, \eqref{eq:31} and \Cref{dfn:scaling}, we obtain
    \begin{equation}\label{eq:Bisiso}
      \begin{aligned}
        \|\hat{B}f\|_{Z_V}
        &= \|F\|_X
          \leq \gamma_X^d(|\hat{\mathcal{B}}_{[0,1)}^{*}|)\cdot \|G\|_X
          = \gamma_X^d(|\hat{\mathcal{B}}_{[0,1)}^{*}|)\cdot \|f\|_{Z_V},\\
        \|\hat{B}f\|_{Z_V}
        &= \|F\|_X
          \geq \bigl(\gamma_X^u(|\hat{\mathcal{B}}_{[0,1)}^{*}|)\bigr)^{-1}\cdot \|G\|_X
          \geq \bigl(\gamma_X^u(|\hat{\mathcal{B}}_{[0,1)}^{*}|)\bigr)^{-1}\cdot \|f\|_{Z_V}.
      \end{aligned}
    \end{equation}
  \end{proofcase}

  \begin{proofcase}[Estimates for $\hat{A}$]
    We define the auxiliary functions $(\Phi^{R}_{R_0} : R_0\in\mathcal{AF},\ R\in \mathcal{S}(R_0))$ given by
    \begin{equation}\label{eq:56}
      \Phi^{R}_{R_0}
      = q \cha_R- \cha_{R_0}
    \end{equation}
    and we have a corresponding block family of functions
    \begin{equation}\label{eq:57}
      \tilde{\Phi}^{P}_{P_0}
      = \sum_{R_0\in \hat{\mathcal{B}}_{P_0}} \Phi^{R}_{R_0},
    \end{equation}
    for any $P_0\in\mathcal{AF}$, $P\in \mathcal{S}(P_0)$, where $R\in \mathcal{S}(R_0)$ is such that $\ord(R_0) = \ord(P)$.  Appealing to the
    definition of $\Gamma_{P_0}^P$ in \Cref{rem:C2}, and then exploiting the faithfulness of the sequence $(\hat{\mathcal{B}}_P : P\in \mathcal{A})$
    (see~\eqref{eq:136}), we obtain
    \begin{equation}
      \frac{\tilde{\Phi}^{P}_{P_0}}{|P_0|}
      = q\frac{\cha_{\Gamma_{P_0}^P}}{|P_0|} - \frac{\cha_{\hat{\mathcal{B}}_{P_0}^*}}{|P_0|}
      = q\frac{\cha_{\hat{\mathcal{B}}_P^*}}{|P_0|} - \frac{\cha_{\hat{\mathcal{B}}_{P_0}^*}}{|P_0|}
      = \frac{\cha_{\hat{\mathcal{B}}_P^*}}{|P|} - \frac{\cha_{\hat{\mathcal{B}}_{P_0}^*}}{|P_0|},
    \end{equation}
    and consequently
    \begin{equation}\label{eq:charainmarbas}
      \frac{\cha_{\hat{\mathcal{B}}_P^*}}{|P|}
      = \cha_{[0,1)} + \sum_{\substack{Q\in \mathcal{A}\\P\subseteq Q\neq [0,1)}} \frac{\tilde{\Phi}^Q_{\pi(Q)}}{|\pi(Q)|}.
    \end{equation}

    We now choose $N \in\mathbb{N}$ so large that for all $R\in \bigcup_{j\geq N} \bigcup_{P\in \mathcal{AF}_j} \hat{\mathcal{B}}_P$, we have
    $a_R^k = 0$ for every $k\in\{1,\ldots,r\}$.  Hence, we have $\langle\hat{\psi}_P^k , f\rangle=0$ for all $P\in \cup_{j\geq N} \mathcal{AF}_j$ and
    $k\in\{1,\ldots,r\}$.  Let $\mathcal{G}$ denote the $\sigma$-algebra generated by the sets $\hat{\mathcal{B}}_{P}^{*}$, $P\in\mathcal{AF}_N$.

    First we calculate $\cond(\varphi_{R_0}^k|\mathcal{G})$ for $R_0\in\mathcal{AF}$ and $1\leq k\leq r$.  Let $P\in \mathcal{AF}$,
    $P_0 = \pi(P)$ and observe that
    \begin{equation}\label{eq:innerPHI}
      \langle\tilde{\Phi}^{P}_{P_0}, \varphi^k_{R_0}\rangle
      = \left\{\begin{array}{cl}
        v^k_{\mathrm{ord}(P)}|R_0| & R_0\in \hat{\mathcal{B}}_{P_0},\\
        0 & \mbox{otherwise.}
      \end{array}\right.
  \end{equation}
  Indeed, if $R_0\in \hat{\mathcal{B}}_{P_0}$ and $R\in \mathcal{S}({R_0})$ with $\ord(R) = \ord(P)$, then by~\eqref{eq:56}, \eqref{eq:57}
  and~\eqref{eq:119} we obtain
  \begin{equation}\label{eq:58}
    \langle\tilde{\Phi}^{P}_{P_0}, \varphi^k_{R_0}\rangle
    = \langle\Phi^R_{R_0}, \varphi^k_{R_0}\rangle
    = \int_{R_0} (q\cha_R - \cha_{R_0}) \sum_{S\in \mathcal{S}({R_0})} v_{\mathrm{ord}(S)}^k \cha_S.
  \end{equation}
  Recall that $V$ is a matrix space with mean zero rows, the above identity implies~\eqref{eq:innerPHI} if $R_0\in \hat{\mathcal{B}}_{P_0}$.  On the
  other hand, if $R_0\notin \hat{\mathcal{B}}_{P_0}$ and $R_0\cap S \neq \emptyset$ for some $S\in \hat{\mathcal{B}}_{P_0}$, then either
  $R_0\subsetneq S$ or $S\subsetneq R_0$.  In the former case, we exploit that $\varphi_{R_0}^k$ has mean zero, and in the latter case, we refer
  to~\eqref{eq:56} and~\eqref{eq:57} and notice that $\tilde{\Phi}^{P}_{P_0}$ has mean zero on each $R\in \mathcal{S}(R_0)$.

  By~\eqref{eq:charainmarbas} and~\eqref{eq:innerPHI}, we obtain
  \begin{equation}\label{eq:214}
    \langle \cha_{\hat{\mathcal{B}}_{P}^*}, \varphi^{k}_{R_0} \rangle
    = \sum_{\substack{Q\in \mathcal{AF}\\P\subseteq Q\neq [0,1) }} \frac{|P|}{|\pi(Q)|}\langle \tilde{\Phi}^{Q}_{\pi(Q)}, \varphi^{k}_{R_0} \rangle
    = 0
  \end{equation}
  for all $R_0 \in \mathcal{AF}\setminus \Bigl(\bigcup_{j < N} \bigcup_{P\in \mathcal{AF}_j} \hat{\mathcal{B}}_P\Bigr)$ and $1\leq k\leq r$, and thus
  \begin{equation}\label{eq:condoutBP}
    \cond(\varphi^k_{R_0}|\mathcal{G})
    = 0.
  \end{equation}

  Let $P_0\in \bigcup_{j< N} \mathcal{AF}_j$, $P\in\mathcal{S}(P_0)$ and $R\in \hat{\mathcal{B}}_{P_0}$ and suppose that $Q\in \mathcal{AF}_N$ is such
  that $Q\subset P$.  By~\eqref{eq:charainmarbas}, exploiting that $\varphi^{k}_{R}$ has mean zero, and~\eqref{eq:innerPHI} the following identity
  holds
  \begin{equation}\label{eq:testoncharBP}
    \langle \cha_{\hat{\mathcal{B}}_Q^*}, \varphi^{k}_{R} \rangle
    = \sum_{\substack{S\in \mathcal{AF} \\ Q\subseteq S\neq [0,1) }} \frac{|Q|}{|\pi(S)|}\langle \tilde{\Phi}^{S}_{\pi(S)}, \varphi^{k}_{R} \rangle
    = \frac{|Q|}{|P_0|} v_{\mathrm{ord}(P)}^k|R|.
  \end{equation}
  Recall that $\mathcal{G}$ was defined as the $\sigma$-algebra generated by the sets $\hat{\mathcal{B}}_P^*$, $P\in \mathcal{AF}_{N}$.
  Using~\eqref{eq:testoncharBP} and~\eqref{eq:128} in \Cref{rem:C2} yields
  \begin{equation}\label{eq:conintBP}
    \begin{aligned}
      \cond(\varphi_{R}^k|\mathcal{G})
      &= \sum_{\substack{Q\in\mathcal{AF}_N\\Q\subset P_0}} \frac{\langle \cha_{\hat{\mathcal{B}}_{Q}^*}, \varphi^{k}_{R} \rangle}{|Q|} \cha_{\hat{\mathcal{B}}_{Q}^*}  
      = \sum_{P\in \mathcal{S}(P_0)} \sum_{\substack{Q\in\mathcal{AF}_N\\Q\subset P}} \frac{\langle \cha_{\hat{\mathcal{B}}_{Q}^*}, \varphi^{k}_R \rangle}{|Q|} \cha_{\hat{\mathcal{B}}_{Q}^*}\\
      &= \sum_{P\in \mathcal{S}(P_0)} \sum_{\substack{Q\in\mathcal{AF}_N\\Q\subset P}} \frac{|R|}{|P_0|} v_{\mathrm{ord}(P)}^k \cha_{\hat{\mathcal{B}}_{Q}^*}.
    \end{aligned}
  \end{equation}
  Combining the above identity with~\eqref{eq:139} and~\eqref{eq:136} gives us
  \begin{equation}\label{eq:140}
    \cond(\varphi_{R}^k|\mathcal{G})
    = \sum_{P\in \mathcal{S}(P_0)} \frac{|R|}{|P_0|} v_{\mathrm{ord}(P)}^k \cha_{\hat{\mathcal{B}}_{P}^*}
    = \sum_{P\in \mathcal{S}(P_0)} \frac{|R|}{|P_0|} v_{\mathrm{ord}(P)}^k \cha_{\Gamma_{P_0}^P}.
  \end{equation}
  Note that by~\eqref{eq:135} and \textrefp[C]{enu:dfn:compatibility:i} we have
  \begin{equation}\label{eq:141}
    \cha_{\Gamma_{P_0}^P}
    = \sum_{S_0\in \hat{\mathcal{B}}_{P_0}} \sum_{\substack{S\in \mathcal{S}(S_0)\\\ord(S) = \ord(P)}} \cha_S,
  \end{equation}
  and thus
  \begin{align}
    \cond(\varphi_{R}^k|\mathcal{G})
    &= \frac{|R|}{|P_0|} \sum_{P\in \mathcal{S}(P_0)} \sum_{S_0\in \hat{\mathcal{B}}_{P_0}} \sum_{\substack{S\in \mathcal{S}(S_0)\\\ord(S) = \ord(P)}} v_{\mathrm{ord}(S)}^k \cha_S\\
    &= \frac{|R|}{|P_0|} \sum_{S_0\in \hat{\mathcal{B}}_{P_0}} \sum_{P\in \mathcal{S}(P_0)} \sum_{\substack{S\in \mathcal{S}(S_0)\\\ord(S) = \ord(P)}} v_{\mathrm{ord}(S)}^k \cha_S\\
    &= \frac{|R|}{|P_0|} \sum_{S_0\in \hat{\mathcal{B}}_{P_0}} \sum_{S\in \mathcal{S}(S_0)} v_{\mathrm{ord}(S)}^k \cha_S.
  \end{align}
  We reassemble the function $\varphi_{S_0}^k$ (see~\eqref{eq:119}) and then $\hat{\psi}^k_{P_0}$ (see~\eqref{eq:17}) and note that
  \begin{equation}\label{eq:217}
    \cond(\varphi_{R}^k|\mathcal{G})
    = \frac{|R|}{|P_0|}  \sum_{S_0\in \hat{\mathcal{B}}_{P_0}} \varphi_{S_0}^k
    = \frac{|R|}{|P_0|} \hat{\psi}^k_{P_0}.
  \end{equation}
  The above identity~\eqref{eq:217}, together with the orthogonality of the system $(\varphi_P^k : P\in \mathcal{A},\ 1\leq k\leq r)$
  and~\eqref{eq:221} yields
  \begin{equation}\label{eq:142}
    \cond(\varphi_{R}^k|\mathcal{G})
    = \frac{\langle\varphi_R^k, \varphi_R^k\rangle}{|P_0|}\hat{\psi}^k_{P_0}
    = \frac{\langle\varphi_R^k, \hat{\psi}^k_{P_0}\rangle}{|P_0|}\hat{\psi}^k_{P_0},
  \end{equation}
  for all $R\in \hat{\mathcal{B}}_{P_0}$, $1\leq k\leq r$.  Combining~\eqref{eq:142} with~\eqref{eq:condoutBP} gives us
  \begin{equation}\label{eq:143}
    \cond(\varphi_{R}^k|\mathcal{G})
    = \hat{B}\hat{A} \varphi_{R}^k,
    \qquad R\in \hat{\mathcal{B}}_{P_0},\ 1\leq k\leq r.
  \end{equation}
  Since $P_0\in \bigcup_{j< N} \mathcal{AF}_j$ was arbitrary and by our choice of $N$, the above identity yields
  $\cond(f | \mathcal{G}) = \hat{B}\hat{A} f$.  We conclude the proof by appealing to \Cref{rem:conditional} and~\eqref{eq:Bisiso}
  \begin{equation*}
    \|\hat{A}f\|_{Z_V}
    \leq \gamma_X^u(|\hat{\mathcal{B}}_{[0,1)}^{*}|)\cdot \|\hat{B}\hat{A} f\|_{Z_V}
    = \gamma_X^u(|\hat{\mathcal{B}}_{[0,1)}^{*}|)\cdot \|E(f|\mathcal{G})\|_{Z_V}
    \leq \gamma_X^u(|\hat{\mathcal{B}}_{[0,1)}^{*}|)\cdot \|f\|_{Z_V}.
  \end{equation*}
\end{proofcase}
\end{myproof}

\begin{thm}\label{thm:canonicalkp}
  Let $U\subset [0,1)$ be a set of positive measure, and suppose that $(\mathcal{B}_P : P\in \mathcal{A})$ is a $(1 - \kappa_P)$-faithful sequence
  mounted on $U$. Then $A, B\colon Z_V\to Z_V$ are well defined bounded operators and we have
  \begin{equation}\label{eq:215}
    AB
    = \Id_{Z_V},
    \qquad
    \|A\|
    \leq \Bigl(1+ \frac{2\sum_{P\in\mathcal{A}}\kappa_P}{1-\kappa_{[0,1)}} \Bigr)\cdot \gamma_X^u(|\mathcal{B}_{[0,1)}^{*}|),
    \qquad
    \|B\|
    \leq \gamma_X^d(|\mathcal{B}_{[0,1)}^{*}|),
  \end{equation}
  where $\gamma_X^u$ and $\gamma_X^d$ are defined in \Cref{dfn:scaling}.
\end{thm}

\begin{myproof}
  Let $(\hat{\mathcal{B}}_P : P\in \mathcal{A})$ denote the faithful extension of $(\mathcal{B}_P : P\in \mathcal{A})$ (we refer to
  \Cref{rem:faithful-extension}), and recall that the sequences of functions $(\psi_P^k : P\in \mathcal{A},\ 1\leq k\leq r)$ and
  $(\hat{\psi}_P^k : P\in \mathcal{A},\ 1\leq k\leq r)$ are given by
  \begin{equation}
    \psi_P^k
    = \sum_{R\in \mathcal{B}_P} \varphi_R^k
    \qquad\text{and}\qquad
    \hat{\psi}_P^k
    = \sum_{R\in \hat{\mathcal{B}}_P} \varphi_R^k,
    \qquad P\in \mathcal{A},\ 1\leq k\leq r,
  \end{equation}
  and the associated canonical factorization operators $A$, $B$ and $\hat{A}$, $\hat{B}$ are defined according to~\eqref{eq:221}.

  Given $N\in \mathbb{N}_0$, \Cref{lem:faithful-extension} shows that the $\sigma$-algebra $\mathcal{G}_N$ generated by
  \begin{equation}\label{eq:208}
    \mathcal{C}_N
    := \bigcup_{j=0}^N \bigcup_{P\in\mathcal{AF}_j} \bigcup_{R\in \mathcal{B}_{P}} \mathcal{S}(R).
  \end{equation}
  has the property that for any $R\in \bigcup_{j=0}^{N} \bigcup_{P\in \mathcal{AF}_j} \hat{\mathcal{B}}_P\setminus \mathcal{B}_P$ there exists an atom
  $A$ of $\mathcal{G}_N$ such that $R\subset A$, and consequently
  \begin{equation}\label{eq:146}
    \cond(\varphi_R^k | \mathcal{G}_N)
    = 0.
  \end{equation}
  Moreover, note that for all $R'\in \bigcup_{j=0}^N \bigcup_{P\in \mathcal{AF}_j} \mathcal{B}_P$ the functions $\varphi_{R'}^k$, $1\leq k\leq r$ are
  all $\sigma(\mathcal{S}(R'))$-measurable, and hence
  \begin{equation}\label{eq:145}
    \cond(\hat{\psi}_P^k | \mathcal{G}_N)
    = \psi_P^k,
    \qquad P\in \bigcup_{j=0}^N \mathcal{AF}_j,\ 1\leq k\leq r.
  \end{equation}

  Consequently, we obtain
  \begin{equation}
    Bf
    = \cond(\hat{B}f|\mathcal{G}_N),
    \qquad f = \sum_{j=0}^N \sum_{P\in \mathcal{AF}_j} \sum_{k=1}^r a_P^k \varphi_P^k.
  \end{equation}
  Note that by \Cref{rem:faithful-extension,rem:faithful-mount}, we have $\mathcal{B}_{[0,1)}^{*} = \hat{\mathcal{B}}_{[0,1)}^{*}$.  Since conditional
  expectations are contractions on $Z_V$ (see \Cref{rem:conditional}) and using \Cref{thm:canonical}, we obtain
  \begin{equation}
    \|B\|
    \leq \|\hat{B}\|
    \leq \gamma_X^d(|\hat{\mathcal{B}}_{[0,1)}^{*}|)
    = \gamma_X^d(|\mathcal{B}_{[0,1)}^{*}|).
  \end{equation}

  We will now estimate $A$.  First, we observe that by~\eqref{eq:145}, we obtain
  \begin{equation}\label{eq:149}
    \langle \psi_P^k, f \rangle
    = \langle \cond(\hat{\psi}_P^k | \mathcal{G}_N), f \rangle
    = \langle \hat{\psi}_P^k, \cond(f | \mathcal{G}_N) \rangle
  \end{equation}
  for all $f\in Z_V$ and $P\in \bigcup_{j=0}^{N} \mathcal{AF}_j$, and thus
  \begin{equation}\label{eq:120}
    \cond\bigl( \hat{A} \cond(f | \mathcal{G}_N) \big| \mathcal{F}_N \bigr)
    = \sum_{j=0}^{N} \sum_{P\in \mathcal{AF}_j} \sum_{k=1}^{r} \frac{\langle \psi_P^k, f \rangle}{|\hat{\mathcal{B}}_P^{*}|} \varphi_P^k.
  \end{equation}
  Second, putting $d_P = \frac{|\hat{\mathcal{B}}^*_P|}{|\mathcal{B}_P^*|} - 1$ for all $P\in \bigcup_{j=0}^{N} \mathcal{AF}_j$ and $d_P = 0$ otherwise and
  $D = (d_P : P\in \mathcal{A})$, and appealing to~\eqref{eq:145} gives us
  \begin{align*}
    \cond\bigl( Af -  \hat{A} \cond(f | \mathcal{G}_N) \big| \mathcal{F}_N \bigr)
    &= \sum_{j=0}^{N} \sum_{P\in \mathcal{AF}_j} \sum_{k=1}^{r}
      \Bigl(\frac{\langle \psi_P^k, f \rangle}{|\mathcal{B}_P^{*}|} - \frac{\langle \psi_P^k, f \rangle}{|\hat{\mathcal{B}}_P^{*}|}\Bigr) \varphi_P^k\\
    &= \sum_{P\in \mathcal{A}} \sum_{k=1}^{r} d_P \frac{\langle \psi_P^k, f \rangle}{|\hat{\mathcal{B}}_P^{*}|} \varphi_P^k
      = \sum_{P\in \mathcal{A}} \sum_{k=1}^{r} d_P \frac{\langle \cond(\hat{\psi}_P^k | \mathcal{G}_N), f \rangle}{|\hat{\mathcal{B}}_P^{*}|} \varphi_P^k\\
    &= D \hat{A} \cond(f | \mathcal{G}_N).
  \end{align*}
  Third, using \Cref{rem:block-basis} gives us
  \begin{align*}\label{eq:150}
    |d_{\pi(P)} - d_P|
    &= \biggl| \frac{|\hat{\mathcal{B}}_{\pi(P)}^*|}{|\mathcal{B}_{\pi(P)}^*|} - \frac{|\hat{\mathcal{B}}_P^*|}{|\mathcal{B}_P^*|} \biggr|
      = \frac{|\hat{\mathcal{B}}_{P}^*|}{|\mathcal{B}_{P}^*|}\cdot \biggl| \frac{|\pi(P)|}{|\mathcal{B}_{\pi(P)}^*| } \cdot \frac{|\mathcal{B}_{P}^*|}{ |P|} - 1 \biggr|\\
    &\le \frac{|\hat{\mathcal{B}}_{P}^*|}{|\mathcal{B}_{P}^*|}\cdot \kappa_{\pi(P)}=\frac{|\mathcal{B}_{[0,1)}^*|}{|\mathcal{B}_{P}^*|}\cdot \frac{|\hat{\mathcal{B}}_{P}^*|}{|\hat{\mathcal{B}}_{[0,1)}^*|}\cdot \kappa_{\pi(P)}\\
    &= \frac{|P|}{|\mathcal{B}_{P}^*|}   \cdot \frac{|\mathcal{B}_{[0,1)}^*|}{|[0,1)|}\cdot\kappa_{\pi(P)}    \le \frac{\kappa_{\pi(P)}}{1-\kappa_{[0,1)}}.
  \end{align*}
  Using \Cref{pro:variation-norm} on $D = (d_P : P\in \mathcal{A})$ and observing that $d_{[0,1)} = 0$ yields the following estimate
  \begin{align*}
    \| Af - \hat{A} \cond(f|\mathcal{G}_N) \|_{Z_V}	
    &\le \|D\|\cdot \|\hat{A}\|\cdot \|\cond(f|\mathcal{G}_N)\|_{Z_V} \leq \|D\| \cdot \|\hat{A}\|\cdot \|f\|_{Z_V}\\
    &\leq \Bigl(\sum_{P\in\mathcal{A}}\frac{2\kappa_{P}}{1-\kappa_{[0,1)}}\Bigr)\|\hat{A}\|\|f\|. 
  \end{align*}
  By \Cref{thm:canonical} and using that $\mathcal{B}_{[0,1)}^{*} = \hat{\mathcal{B}}_{[0,1)}^{*}$ again, we get
  \begin{align*}
    \|Af\|_{Z_V}
    &\leq \Bigl(1+ \sum_{P\in\mathcal{A}} \frac{2\kappa_{P}}{1-\kappa_{[0,1)}}\Bigr) \|\hat{A}\|\\
    &\le \Bigl(1+ \sum_{P\in\mathcal{A}} \frac{2\kappa_{P}}{1-\kappa_{[0,1)}}\Bigr)\cdot \gamma_X^u(|\mathcal{B}_{[0,1)}^{*}|).
      \qedhere
  \end{align*}
\end{myproof}

\begin{lem}\label{lem:functional-stability}
  Let $Z_V = Z_V((\mathcal{F}_j),X,\boldsymbol{\sigma};\mathbb{R}^l)$ denote a Janson-Sobolev space or a Janson-Sobolev-Hardy space, and let
  $\mathcal{B}$ denote an orthonormal basis of $V$.  Suppose that $D = (d_P^k : P\in \mathcal{A},\ 1\leq k\leq r)\colon Z_V\to Z_V$ is a bounded diagonal operator with respect to
  $(\varphi_P^{k,\mathcal{B}} : P\in \mathcal{A},\ 1\leq k\leq r)$, and let $\mathcal{U}$ denote any non-principal ultrafilter on $\mathbb{N}_0$.  Then there exists an increasing sequence
  $(\tilde{n}_i)_{i=0}^{\infty}$ in $\mathbb{N}_0$ such that for any subsequence $(n_i)_{i=0}^{\infty}$ of $(\tilde{n}_i)_{i=0}^{\infty}$, for all
  $0\leq \kappa_P < 1$, $P\in \mathcal{A}$ with $0 \leq \kappa_{[0,1)} \leq \kappa< \tfrac{1}{2}$ and any $(1 - \kappa_P)$-faithful system
  $(\mathcal{B}_P : P\in \mathcal{A})$ with respect to $(n_i)_{i=0}^{\infty}$ mounted on a measurable set $U$ with $|U| \geq 1 - \kappa$, we have
  \begin{equation}\label{eq:67}
    |\lambda_{\mathcal{U}}^k(D) - \lambda_{\mathcal{U}}^k(ADB)|
    \leq \frac{6\kappa}{1-2\kappa} \|D\|,
    \qquad 1\leq k\leq r,
  \end{equation}
  where $A$, $B$ are the canonical factorization operators with respect to $(\mathcal{B}_P : P\in \mathcal{A})$ given by~\eqref{eq:221}.
\end{lem}

\begin{proof}
  By \Cref{lem:functionals-classical-limit}, we can find a sequence $(\tilde{n}_i)_{i=0}^{\infty}$ such that
  \begin{equation}\label{eq:170}
    \lambda_{\mathcal{U}}^k(D)
    = \lim_{i\to \infty} \sum_{P\in \mathcal{AF}_{\tilde{n}_i}} |P| d_P^k,
    \qquad 1\leq k\leq r.
  \end{equation}
  Let $(n_i)_{i=0}^{\infty}$, $(\kappa_P)$, $(\mathcal{B}_P : P\in \mathcal{A})$ and $U$ be as specified by our hypothesis.

  Now we put $\widetilde{D} = (\widetilde{d}_P^k : P\in \mathcal{A},\ 1\leq k\leq r) = ADB$, where $A$, $B$ are given by~\eqref{eq:221}, and note that
  \begin{equation}\label{eq:171}
    \widetilde{d}_P^k
    = \frac{\langle \psi_P^k, D \psi_P^k \rangle}{|\mathcal{B}_P^{*}|}
    = \sum_{R\in \mathcal{B}_P} \frac{|R|}{|\mathcal{B}_P^{*}|} d_R^k.
  \end{equation}
  Recall that the ultrafilter limit agrees with the classical limit if it exists, thus~\eqref{eq:170} and~\eqref{eq:171} yield
  \begin{equation}
    |\lambda_{\mathcal{U}}^k(D) - \lambda_{\mathcal{U}}^k(\widetilde{D})|
    = \lim_{i\to \mathcal{U}} \Bigl|
    \sum_{R\in \mathcal{AF}_{n_i}} |R| d_R^k
    - \sum_{P\in \mathcal{AF}_i} |P| \sum_{R\in \mathcal{B}_P} \frac{|R|}{|\mathcal{B}_P^{*}|} d_R^k
    \Bigr|.
  \end{equation}
  Now we put $\mathcal{B}_i = \bigcup_{P\in \mathcal{AF}_i} \mathcal{B}_P$, recall that $|d_R^k|\leq \|D\|$ and observe
  \begin{equation}\label{eq:178}
    \begin{aligned}
      |\lambda_{\mathcal{U}}^k(D) - \lambda_{\mathcal{U}}^k(\widetilde{D})|
      =& \lim_{i\to \mathcal{U}} \Bigl|
         \sum_{R\in \mathcal{AF}_{n_i}\setminus \mathcal{B}_i} |R| d_R^k
         + \sum_{P\in \mathcal{AF}_i} \sum_{R\in \mathcal{B}_P}
         \Bigl(|R| d_R^k - |P| \frac{|R|}{|\mathcal{B}_P^{*}|} d_R^k\Bigr)
         \Bigr|\\
      \leq& \lim_{i\to \mathcal{U}} \sum_{R\in \mathcal{AF}_{n_i}\setminus \mathcal{B}_i} |R|\cdot |d_R^k|
            + \sum_{P\in \mathcal{AF}_i} \sum_{R\in \mathcal{B}_P}
            |R|\cdot |d_R^k|\cdot \bigl| 1 -  \frac{|P|}{|\mathcal{B}_P^{*}|} \bigr|\\
      \leq& \|D\|\cdot \lim_{i\to \mathcal{U}} \Bigl(
            1 - |\mathcal{B}_i^{*}| + \sup_{P\in \mathcal{A}} \bigl| 1 -  \frac{|P|}{|\mathcal{B}_P^{*}|} \bigr|
            \Bigr).
    \end{aligned}
  \end{equation}
  Putting $P=[0,1)$ and $Q = P$ in~\eqref{eq:34} gives us
  \begin{equation}\label{eq:174}
    \bigl| 1 -  \frac{|P|}{|\mathcal{B}_P^{*}|} \bigr|
    \leq (1 - \kappa_{[0,1)})^{-1} |\mathcal{B}_{[0,1)}^{*}|^{-1} - 1,
  \end{equation}
  as well as
  \begin{equation}\label{eq:175}
    |\mathcal{B}_i^{*}|
    = \sum_{P\in \mathcal{AF}_i} |\mathcal{B}_P^{*}|
    \geq \sum_{P\in \mathcal{AF}_i} |P|\cdot (1 - \kappa_{[0,1)})\cdot |\mathcal{B}_{[0,1)}^{*}|
    = (1 - \kappa_{[0,1)})\cdot |\mathcal{B}_{[0,1)}^{*}|.
  \end{equation}
  Exploiting \textrefp[C]{enu:dfn:compatibility:v} yields
  \begin{equation}\label{eq:176}
    1 - \kappa
    \leq |U|
    = |U\cap \mathcal{B}_{[0,1)}^{*}| + |U\setminus \mathcal{B}_{[0,1)}^{*}|
    \leq |\mathcal{B}_{[0,1)}^{*}| + \kappa_{[0,1)}\cdot|\mathcal{B}_{[0,1)}^{*}|,
  \end{equation}
  and we record the resulting estimate
  \begin{equation}\label{eq:177}
    |\mathcal{B}_{[0,1)}^{*}|
    \geq \frac{1 - \kappa}{1 + \kappa_{[0,1)}}.
  \end{equation}
  Combining~\eqref{eq:175} with~\eqref{eq:177} and recalling that $\kappa_{[0,1)}\leq \kappa$ gives us
  \begin{equation}\label{eq:216}
    |\mathcal{B}_i^{*}|
    \geq \frac{(1 - \kappa_{[0,1)})(1 - \kappa)}{1 + \kappa_{[0,1)}}
    \geq \frac{(1 - \kappa)^2}{1 + \kappa}
    \geq \frac{1 - 2\kappa}{1 + \kappa}
    = 1 - \frac{3\kappa}{1 + \kappa},
  \end{equation}
  and inserting~\eqref{eq:177} into~\eqref{eq:174} yields
  \begin{equation}\label{eq:179}
    \begin{aligned}
      \bigl| 1 -  \frac{|P|}{|\mathcal{B}_P^{*}|} \bigr|
      &\leq \frac{1 + \kappa_{[0,1)}}{(1 - \kappa_{[0,1)})(1 - \kappa)} - 1
        \leq \frac{1 + \kappa_{[0,1)}}{1 - \kappa_{[0,1)} - \kappa} - 1\\
      &\leq \frac{3\kappa}{1 - \kappa_{[0,1)} - \kappa}
        \leq \frac{3\kappa}{1 - 2\kappa}.
    \end{aligned}
  \end{equation}

  To summarize, by combining~\eqref{eq:178}, \eqref{eq:216} and~\eqref{eq:179}, we obtain
  \begin{equation}
    |\lambda_{\mathcal{U}}^k(D) - \lambda_{\mathcal{U}}^k(\widetilde{D})|
    \leq \|D\|\cdot \Bigl(\frac{3\kappa}{1 + \kappa} + \frac{3\kappa}{1 - 2\kappa}\Bigr)
    \leq \|D\|\cdot \frac{6\kappa}{1 -2\kappa}
  \end{equation}
  as claimed.

\end{proof}

%

\section{Stabilization}
\label{sec:stabilization}

The stochastic stabilization method introduced in~\cite{MR4430957} and subsequently employed in~\cite{MR4816119,MR4894823,MR4884827} is based on
randomized Haar bases.  It selects a randomly chosen group action compatible with a standard Gamlen-Gaudet construction step in order to stabilize a
given diagonal operator.  More precisely, Haar functions are modulated by random signs at predetermined positions.  It appears that a similar approach
would be feasible in our setting, provided that sufficient degrees of freedom are available in a compatible group action on $V$, such as translation
invariance in~\cite[Definition~7.1]{MR4865002}.  However, this would impose substantial restrictions on the subspace $V$.  We therefore develop, in
this section, a novel stabilization procedure based on martingale methods, which applies to all Janson-Sobolev and Janson-Sobolev-Hardy spaces.

\subsection{Stabilization via martingale diagonal operators}
\label{sec:stab-via-mart}

Starting from a generic bounded diagonal operator, we use factorization to successively improve the convergence properties of the associated adapted
processes.  The first step produces a closed martingale, which in particular converges pointwise almost everywhere.  In a subsequent step we upgrade
this to uniform convergence everywhere with arbitrarily fast rate of convergence.

\begin{pro}\label{pro:stab-2}
  Let $D = (d_P^k : P\in \mathcal{A},\ 1\leq k\leq r)$ denote a bounded diagonal operator, and let $\mathcal{U}$ denote any non-principal ultrafilter on
  $\mathbb{N}_0$.  Then there exists another bounded diagonal operator $\widehat{D} = (\widehat{d}_P^k : P\in \mathcal{A},\ 1\leq k\leq r)$ which almost projectionally factors
  through $D$ with constant~$1$, and has the additional property that, for every $1\leq k\leq r$, the associated adapted process
  $(A_n^{\widehat{D},k})_{n=0}^{\infty}$ is a closed martingale. In particular, it converges almost everywhere and in $L^1$.  Moreover, we have
  \begin{equation}\label{eq:180}
    \lambda_{\mathcal{U}}^k(\widehat{D})
    = \lambda_{\mathcal{U}}^k(D),
    \qquad 1\leq k\leq r.
  \end{equation}
\end{pro}

\begin{proof}
  Let $\eta > 0$ and pick a sequence of positive numbers $(\eta_P : P\in \mathcal{A})$ such that
  $\sum_{P\in \mathcal{A}} \eta_P \leq \eta/(2\lambda_{Z_V})$, where $\lambda_{Z_V}$ is the basis constant of $(\varphi_P^k : P\in \mathcal{A},\ 1\leq k\leq r)$.

  By \Cref{lem:functional-stability}, there exists an increasing sequence $(\tilde{n}_i)_{i=0}^{\infty}$ in $\mathbb{N}_0$ such that for any subsequence
  $(n_i)_{i=0}^{\infty}$ of $(\tilde{n}_i)_{i=0}^{\infty}$ and any faithful system $(\mathcal{B}_P : P\in \mathcal{A})$ with respect to
  $(n_i)_{i=0}^{\infty}$ mounted on $[0,1)$, we have $\lambda_{\mathcal{U}}^k(D) = \lambda_{\mathcal{U}}^k(ADB)$, $1\leq k\leq r$, where $A$, $B$ are
  the canonical factorization operators with respect to $(\mathcal{B}_P : P\in \mathcal{A})$ given by~\eqref{eq:221}.  Moreover, note that by
  \Cref{thm:canonical} and \Cref{dfn:scaling}, we have $\|A\| = \|B\| = 1$.

  As discussed in~\eqref{eq:1} to~\eqref{eq:4}, there exists a subsequence $(n_j)_{j=0}^{\infty}$ of $(\tilde{n}_j)_{j=0}^{\infty}$ such that each of
  the processes $(A_{n_j}^{D,k})_{j=0}^{\infty}$, $1\leq k\leq r$ converges in the weak$^{*}$ topology of $L^{\infty}$.  We denote these limits by
  $M^{D,k}$, $1\leq k\leq r$, and define the diagonal operator $\widetilde{D} = (\widetilde{d}_P^k : P\in \mathcal{A},\ 1\leq k\leq r)$ by
  \begin{equation}\label{eq:35}
    \widetilde{d}_P^k
    := M_n^{D,k}(x)
    = \cond(M^{D,k}|\mathcal{F}_n)(x),
  \end{equation}
  for all $1\leq k\leq r$ and whenever $x\in P\in \mathcal{AF}_n$ for $n\in \mathbb{N}_0$.

  Next, we define $\eta_j = \min_{P\in \mathcal{AF}_j} \eta_P > 0$, $j\in \mathbb{N}_0$ and let $1\leq k\leq r$.  Since both
  $(A_{n_j}^{D,k})_{j=0}^{\infty}$ and $(M_{n_j}^{D,k})_{j=0}^{\infty}$ converge to $M^{D,k}$ as $j\to \infty$ (see also~\eqref{eq:3}
  and~\eqref{eq:4}), we obtain
  \begin{equation*}
    \lim_j \Bigl| \frac{1}{|\Gamma|} \int_{\Gamma} A_{n_j}^{D,k} - \frac{1}{|\Gamma|} \int_{\Gamma} M_{n_j}^{D,k} \Bigr|
    = 0,
    \qquad \Gamma\in \bigcup_{j=0}^{\infty} \mathcal{F}_j.
  \end{equation*}
  Thus, since the $\sigma$-algebras $\mathcal{F}_j$, $j\in \mathbb{N}_0$ are finite, we can extract a subsequence $(m_j)_{j=0}^{\infty}$ of
  $(n_j)_{j=0}^{\infty}$ with the following properties:
  \begin{enumerate}[(a)]
  \item\label{item:1} $\bigl| \int_{\Gamma} A_{m_0}^{D,k} - \int_{\Gamma} M_{m_0}^{D,k} \bigr| \leq \eta_0 |\Gamma|$ for all
    $\Gamma \in \mathcal{F}_0$.
  \item\label{item:2} $\bigl| \int_{\Gamma} A_{m_j}^{D,k} - \int_{\Gamma} M_{m_j}^{D,k} \bigr| \leq \eta_j |\Gamma|$ for all
    $\Gamma \in \mathcal{F}_{m_{j-1}+1}$ and $j\in \mathbb{N}$.
  \end{enumerate}
  
  Utilizing \Cref{lem:block-basis} for $U=[0,1)$ and $(m_j)_{j=0}^{\infty}$, we obtain a sequence of collections $(\mathcal{B}_P : P\in \mathcal{A})$
  which is faithful with respect to a subsequence of $(m_j)_{j=0}^{\infty}$.  For ease of notation, we denote this subsequence also by
  $(m_j)_{j=0}^{\infty}$.  According to~\eqref{eq:17}, $(\mathcal{B}_P : P\in \mathcal{A})$ defines the block basis
  $(\psi_P^k : P\in \mathcal{A},\ 1\leq k\leq r)$ of $(\varphi_P^k : P\in \mathcal{A},\ 1\leq k\leq r)$, which in turn determines the associated
  factorization operators $A,B\colon Z_V\to Z_V$ (see~\eqref{eq:221}).
  
  Since $(\mathcal{B}_P : P\in \mathcal{A})$ is faithful, \Cref{thm:canonical}, together with $\gamma_X^u(1) = \gamma_X^d(1) = 1$ (see
  \Cref{dfn:scaling}), yields
  \begin{equation}\label{eq:7}
    AB = \Id_{Z_V}
    \quad\text{and}\quad
    \|A\| = \|B\| = 1.
  \end{equation}

  Let $P\in \mathcal{AF}_j$ for some $j\in \mathbb{N}_0$ and $1\leq k\leq r$ be fixed and observe that combining~\eqref{eq:35} with
  \Cref{dfn:canonical} yields
  \begin{equation}\label{eq:52}
    A(D - \widetilde{D})B \varphi_P^k
    = \frac{1}{|\mathcal{B}_P^{*}|}\langle \psi_P^k, (D - \widetilde{D}) \psi_P^k \rangle \varphi_P^k
    = \Bigl( \sum_{Q\in \mathcal{B}_P} \frac{|Q|}{|\mathcal{B}_P^{*}|} (d_Q^k - \widetilde{d}_Q^k) \Bigr) \varphi_P^k.
  \end{equation}
  In particular, we obtain that $A(D - \widetilde{D})B$ is a diagonal operator with respect to $(\varphi_P^k : P\in \mathcal{A},\ 1\leq k\leq r)$.
  Since the sequence of collections $(\mathcal{B}_P : P\in \mathcal{A})$ is faithful with respect to $(m_j)_{j=0}^{\infty}$,
  Condition~\textrefp[C]{enu:dfn:compatibility:iv} gives us $\mathcal{B}_P\subset \mathcal{AF}_{m_j}$ and \Cref{rem:block-basis:2} ensures that
  $\mathcal{B}_P^{*}\in \mathcal{F}_{m_{j-1}+1}$ whenever $P\in \mathcal{AF}_j$, $j\in \mathbb{N}$.  Moreover, by \Cref{rem:block-basis:3} we have
  that $\mathcal{B}_P^{*}\in \mathcal{F}_0$, since $U=[0,1)$ is $\mathcal{F}_0$-measurable.  Recall that we fixed $P\in \mathcal{AF}_j$. Hence,
  by~\eqref{eq:1}, \eqref{eq:35} and using~\eqref{item:1} and~\eqref{item:2} with $\Gamma = \mathcal{B}_P^{*}$, we obtain
  \begin{equation*}
    \Bigl|\sum_{Q\in \mathcal{B}_P} \frac{|Q|}{|\mathcal{B}_P^{*}|} (d_Q^k - \widetilde{d}_Q^k)\Bigr|
    = \frac{1}{|\mathcal{B}_P^{*}|} \Bigl| \int_{\mathcal{B}_P^{*}} A_{m_j}^{D,k} - M_{m_j}^{D,k} \Bigr|
    \leq \eta_j
    \leq \eta_P.
  \end{equation*}
  Recall that we defined $\eta_j = \min_{P\in \mathcal{AF}_j} \eta_P$.  Hence, applying \Cref{lem:schauder} to the diagonal operator
  $A(D - \widetilde{D})B$ (see~\eqref{eq:52}) yields
  \begin{equation}\label{eq:19}
    \|A(D - \widetilde{D})B\|
    \leq 2\lambda_{Z_V} \sum_{P\in \mathcal{A}} \sum_{k=1}^r \eta_P
    \leq \eta.
  \end{equation}
  We define $\widehat{D} = A\widetilde{D}B$ and record the estimate
  \begin{equation}\label{eq:20}
    \|ADB - \widehat{D}\|
    \leq \eta,
  \end{equation}
  which is simply~\eqref{eq:19}.  Since $\|A\| = \|B\| = 1$ and $\eta>0$ can be made arbitrarily small, \eqref{eq:20} shows that $\widehat{D}$ almost
  projectionally factors through $D$ with constant $1$, as claimed.

  Our final task is to demonstrate that $(A_j^{\widehat{D},k})_{j=0}^{\infty}$ forms a closed martingale for each $1\leq k\leq r$.  First, note that
  for fixed $1\leq k\leq r$ we obtain from \Cref{dfn:canonical} and~\eqref{eq:35} that
  \begin{equation}\label{eq:12}
    A_j^{\widehat{D},k}(x)
    = \widehat{d}_P^{k}
    = \sum_{R\in \mathcal{B}_P} \frac{|R|}{|P|} \widetilde{d}_R^k
    = \frac{1}{|\mathcal{B}_P^{*}|} \int_{\mathcal{B}_P^{*}} M_{m_j}^{D,k}
    = \frac{1}{|\mathcal{B}_P^{*}|} \int_{\mathcal{B}_P^{*}} M^{D,k},
  \end{equation}
  for all $x\in P\in \mathcal{AF}_j$, $j\in \mathbb{N}_0$.  The last equality follows from the fact that
  $M_{m_j}^{D,k} = \cond(M^{D,k}|\mathcal{F}_{m_j})$ and $\mathcal{B}_P^{*}$ is $\mathcal{F}_{m_j}$-measurable whenever $P\in \mathcal{AF}_j$,
  $j\in \mathbb{N}_0$.  Secondly, note that the faithfulness of $(\mathcal{B}_P : P\in \mathcal{A})$ (i.e., $\kappa_P = 0$, $P\in \mathcal{A}$),
  together with~\eqref{eq:34} yields $\frac{|\mathcal{B}_Q^{*}|}{|\mathcal{B}_P^{*}|} = \frac{|Q|}{|P|}$, for all $P,Q\in \mathcal{A}$ with
  $Q\subset P$.  Hence, by the above identity~\eqref{eq:12} and \Cref{rem:block-basis:2} we observe that
  \begin{align*}
    \cond (A_{j+1}^{\widehat{D},k} | \mathcal{F}_j)(x)
    &= \frac{1}{|P|} \sum_{\substack{Q\in \mathcal{F}_{j+1}\\Q\subset P}} \frac{|Q|}{|\mathcal{B}_Q^{*}|} \int_{\mathcal{B}_Q^{*}} M^{D,k}
    = \frac{1}{|\mathcal{B}_P^{*}|} \sum_{\substack{Q\in \mathcal{F}_{j+1}\\Q\subset P}} \int_{\mathcal{B}_Q^{*}} M^{D,k}\\
    &= \frac{1}{|\mathcal{B}_P^{*}|} \int_{\mathcal{B}_P^{*}} M^{D,k}
      = A_j^{\widehat{D},k}(x),
  \end{align*}
  for all $x\in P\in \mathcal{AF}_j$, $j\in \mathbb{N}_0$, which demonstrates that $(A_j^{\widehat{D},k})_{j=0}^{\infty}$ is a martingale.  Since
  $\|M^{D,k}\|_{L^{\infty}}\leq \|D\| < \infty$ (see~\eqref{eq:74}, \eqref{eq:1} and~\eqref{eq:2}), it follows from~\eqref{eq:12} that
  $\sup_j \|A_j^{\widehat{D},k}\|_{L^{\infty}}\leq \|M^{D,k}\|_{L^{\infty}}$, and hence, $(A_j^{\widehat{D},k})_{j=0}^{\infty}$ is a closed martingale
  by the martingale convergence theorem.
\end{proof}

We will now upgrade pointwise convergence to arbitrarily fast uniform convergence.
\begin{pro}\label{pro:stab-3}
  Suppose that $D = (d_P^k : P\in \mathcal{A},\ 1\leq k\leq r)$ is a bounded diagonal operator such that $(A_n^{D,k})_{n=0}^{\infty}$ converges almost
  everywhere for all $1\leq k\leq r$.  Let $\mathcal{U}$ denote any non-principal ultrafilter on $\mathbb{N}_0$, let $0 < \kappa < \tfrac{1}{2}$ and
  denote by $(\eta_n)_{n=0}^{\infty}$ a sequence of positive numbers with $\eta_n\to 0$.

  Then there exists a bounded diagonal operator $\widehat{D} = (\widehat{d}_P^k : P\in \mathcal{A},\ 1\leq k\leq r)$ which projectionally factors
  through $D$ with constant~$1^+$ and so that
  \begin{equation}\label{eq:pro:stab-3:1}
    \bigl| A_i^{\widehat{D},k}(x) - A_j^{\widehat{D},k}(x) \bigr|
    \leq \eta_i + \eta_j,
    \qquad x\in [0,1),\ 1\leq k\leq r,\ i,j\in \mathbb{N}_0.
  \end{equation}
  In particular, $(A_n^{\widehat{D},k})_{n=0}^{\infty}$ converges uniformly on $[0,1)$ for all $1\leq k\leq r$.  Moreover, we have
  \begin{equation}\label{eq:181}
    |\lambda_{\mathcal{U}}^k(D) - \lambda_{\mathcal{U}}^k(\widehat{D})|
    \leq \frac{6\kappa}{1-2\kappa} \|D\|,
    \qquad 1\leq k\leq r.
  \end{equation}
\end{pro}

\begin{myproof}
  Let $\delta > 0$ and $0\leq \kappa_P < 1$, $P\in \mathcal{A}$ such that
  \begin{equation}\label{eq:152}
    \Bigl(1+ \frac{2\sum_{P\in\mathcal{A}}\kappa_P}{1-\kappa_{[0,1)}} \Bigr)\cdot \gamma_X^u(s) \gamma_X^d(s)
    \leq 1 + \delta,
    \qquad\text{for all $s\geq 1-\kappa-\kappa_{[0,1)}$.}
  \end{equation}
  Such a choice is possible since $\lim_{s\to 1} \gamma_X^u(s) \gamma_X^d(s) = 1$ (see \Cref{rem:scaling}).  By Egorov's theorem, we can find a
  measurable set $U\subset [0,1)$ with $|U| > 1 - \kappa$ such that $(A_n^{D,k})_{n=0}^{\infty}$ converges uniformly on $U$ for all $1\leq k\leq r$.

  According to \Cref{lem:functional-stability}, we select an increasing sequence $(\tilde{n}_i)_{i=0}^{\infty}$ in $\mathbb{N}_0$ such that for any
  subsequence $(n_i)_{i=0}^{\infty}$ of $(\tilde{n}_i)_{i=0}^{\infty}$ and any $(1 - \kappa_P)$-faithful system $(\mathcal{B}_P : P\in \mathcal{A})$
  with respect to $(n_i)_{i=0}^{\infty}$ mounted on $U$, we have
  \begin{equation}\label{eq:182}
    |\lambda_{\mathcal{U}}^k(D) - \lambda_{\mathcal{U}}^k(ADB)|
    \leq \frac{6\kappa}{1-2\kappa} \|D\|,
    \qquad 1\leq k\leq r,
  \end{equation}
  where $A$, $B$ are the canonical factorization operators with respect to $(\mathcal{B}_P : P\in \mathcal{A})$ given by~\eqref{eq:221}.  In the
  following, every chosen sequence will be a subsequence of $(\tilde{n}_i)_{i=0}^{\infty}$.

  Using \Cref{lem:block-basis} with our $U$ and $(\kappa_P : P\in \mathcal{A})$, we can find a sequence of collections
  $(\mathcal{B}_P : P\in \mathcal{A})$ which is $(1 - \kappa_P)$-faithful.  Since $|U| > 1 - \kappa$ and by
  \Cref{dfn:compatibility}~\textrefp[C]{enu:dfn:compatibility:v}, we obtain
  \begin{equation}\label{eq:13}
    |\mathcal{B}_{[0,1)}^{*}|
    \geq 1 - \kappa - \kappa_{[0,1)}.
  \end{equation}

  Next, we define the associated block basis $(\psi_P^k : P\in \mathcal{A},\ 1\leq k\leq r)$ and the canonical factorization operators
  $A,B\colon Z_V\to Z_V$ according to \Cref{dfn:canonical}.  Using \Cref{thm:canonicalkp}, we obtain
  \begin{equation}\label{eq:151}
    AB
    = \Id_{Z_V},
    \quad
    \|A\|
    \leq \Bigl(1+ \frac{2\sum_{P\in\mathcal{A}}\kappa_P}{1-\kappa_{[0,1)}} \Bigr)\cdot \gamma_X^u(|\mathcal{B}_{[0,1)}^{*}|),
    \quad
    \|B\|
    \leq \gamma_X^d(|\mathcal{B}_{[0,1)}^{*}|).
  \end{equation}

  Now, define the diagonal operator $\widetilde{D} = (\widetilde{d}_P^k : P\in \mathcal{A},\ 1\leq k\leq r)$ as the linear operator
  $\widetilde{D}\colon Z_V\to Z_V$ given by $\widetilde{D} = ADB$, i.e.,
  \begin{equation}\label{eq:29}
    \widetilde{D} \varphi_P^k
    = \widetilde{d}_P^k \varphi_P^k
    = \frac{\langle \psi_P^k, D\psi_P^k \rangle}{|\mathcal{B}_P^{*}|} \varphi_P^k,
    \qquad P\in \mathcal{A},\ 1\leq k\leq r.
  \end{equation}
  Thus, by~\eqref{eq:151}, \eqref{eq:13} and~\eqref{eq:152}, $\widetilde{D}$ projectionally factors through $D$ with constant $1+\delta$, and
  by~\eqref{eq:182}, we also have
  \begin{equation}\label{eq:183}
    |\lambda_{\mathcal{U}}^k(D) - \lambda_{\mathcal{U}}^k(\widetilde{D})|
    \leq \frac{6\kappa}{1-2\kappa} \|D\|,
    \qquad 1\leq k\leq r,
  \end{equation}

  Next, we show that $(A_n^{\widetilde{D},k})_{n=0}^{\infty}$ converges uniformly on $[0,1)$ for all $1\leq k\leq r$.  First recall that we chose $U$
  such that $(A_n^{D,k})_{n=0}^{\infty}$ converges uniformly on $U$ for all $1\leq k\leq r$.  Thus, given $\varepsilon > 0$ we can find
  $m_0\in \mathbb{N}$ so that for all $m,n\in \mathbb{N}$ with $n > m \geq m_0$ we have $|A_m^{D,k}(x) - A_n^{D,k}(x)| \leq \varepsilon$ whenever
  $x\in U$, i.e., by~\eqref{eq:1}:
  \begin{equation}\label{eq:5}
    | d_R^k - d_S^k | \leq \varepsilon,
    \qquad R\in \mathcal{AF}_m,\ S\in \mathcal{AF}_n,\ U\cap S\cap R\neq \emptyset.
  \end{equation}
  By \Cref{dfn:compatibility}~\textrefp[C]{enu:dfn:compatibility:v}, the above inequality~\eqref{eq:5}, and since $\kappa_P\to 0$ as $|P| \to 0$
  together with~\textrefp[C]{enu:dfn:compatibility:iv}, we can find an index $i_0$ so that for all $j > i \geq i_0$ we have
  \begin{equation}\label{eq:37}
    | d_R^k - d_S^k | \leq \varepsilon
    \qquad\text{as well as}\qquad
    \frac{\kappa_P}{1-\kappa_P} \|D\|
    \leq \varepsilon,
  \end{equation}
  whenever $P\in \mathcal{AF}_i$, $Q\in \mathcal{AF}_j$, $R\in \mathcal{B}_P$, $S\in \mathcal{B}_Q$ and $S\subset R$.
  
  Now let $x\in [0,1)$ be fixed.  Then there exist unique $P\in \mathcal{AF}_i$ and $Q\in \mathcal{AF}_j$ such that $x\in Q\subset P$.  Again, by the
  definition of $A_n^{\widetilde{D},k}$ and $\widetilde{D}$ (see~\eqref{eq:1} and~\eqref{eq:29}) as well as exploiting the orthogonality of the system
  $(\varphi_R^k)$, we obtain
  \begin{align*}
    |A_i^{\widetilde{D},k}(x) - A_j^{\widetilde{D},k}(x)|
    &= |d_P^{\widetilde{D},k} - d_Q^{\widetilde{D},k}|
      = \Bigl| \frac{\langle \psi_P^k, D\psi_P^k \rangle}{|\mathcal{B}_P^{*}|} - \frac{\langle \psi_Q^k, D\psi_Q^k \rangle}{|\mathcal{B}_Q^{*}|} \Bigr|\\
    &= \Bigl| \sum_{R\in \mathcal{B}_P} \frac{|R|}{|\mathcal{B}_P^{*}|} d_R^k - \sum_{S\in \mathcal{B}_Q} \frac{|S|}{|\mathcal{B}_Q^{*}|} d_S^k \Bigr|.
  \end{align*}
  We use Property~\textrefp[C]{enu:dfn:compatibility:ii}, to decompose the second summation as follows:
  \begin{equation}\label{eq:27}
    |A_i^{\widetilde{D},k}(x) - A_j^{\widetilde{D},k}(x)|
    = \Bigl|
    \sum_{R\in \mathcal{B}_P} \frac{|R|}{|\mathcal{B}_P^{*}|} \Bigl(
    d_R^k - \sum_{\substack{S\in \mathcal{B}_Q\\S\subset R}} \frac{|\mathcal{B}_P^{*}|}{|R|} \frac{|S|}{|\mathcal{B}_Q^{*}|} d_S^k
    \Bigr)
    \Bigr|.
  \end{equation}
  Replacing $d_S^k$ in the inner sum by $d_R^k$ yields
  \begin{equation}\label{eq:18}
    \begin{aligned}
      |A_i^{\widetilde{D},k}(x) - A_j^{\widetilde{D},k}(x)|
      &\leq \Bigl|
        \sum_{R\in \mathcal{B}_P} \frac{|R|}{|\mathcal{B}_P^{*}|} \Bigl(
        d_R^k - \sum_{\substack{S\in \mathcal{B}_Q\\S\subset R}} \frac{|\mathcal{B}_P^{*}|}{|R|} \frac{|S|}{|\mathcal{B}_Q^{*}|} d_R^k
      \Bigr)
      \Bigr|\\
      &\qquad + \Bigl|
        \sum_{R\in \mathcal{B}_P} \Bigl(
        \sum_{\substack{S\in \mathcal{B}_Q\\S\subset R}} \frac{|S|}{|\mathcal{B}_Q^{*}|} (d_S^k - d_R^k)
      \Bigr)
      \Bigr|
      =: \alpha + \beta.
    \end{aligned}
  \end{equation}

  We will now separately estimate $\alpha$ and $\beta$.  Exploiting~\eqref{eq:37}, property \textrefp[C]{enu:dfn:compatibility:i} of our construction and the
  fact that $\mathcal{B}_Q^{*}\subset \mathcal{B}_P^{*}$ (see \Cref{rem:block-basis}) yields
  \begin{equation}\label{eq:54}
    \beta
    \leq \sum_{R\in \mathcal{B}_P} \sum_{\substack{S\in \mathcal{B}_Q\\S\subset R}} \frac{|S|}{|\mathcal{B}_Q^{*}|} |d_S^k - d_R^k|
    \leq \varepsilon\cdot \sum_{R\in \mathcal{B}_P} \frac{|\mathcal{B}_Q^{*}\cap R|}{|\mathcal{B}_Q^{*}|}
    = \varepsilon.
  \end{equation}
  For the other term, we proceed as follows.  Since $|d_R^{k}|\leq \|D\|$, we obtain
  \begin{equation}\label{eq:38}
    \begin{aligned}
      \alpha
      &\leq \|D\|\cdot \sum_{R\in \mathcal{B}_P} \frac{|R|}{|\mathcal{B}_P^{*}|}\cdot \Bigl|
        1 - \sum_{\substack{S\in \mathcal{B}_Q\\S\subset R}} \frac{|\mathcal{B}_P^{*}|}{|R|} \frac{|S|}{|\mathcal{B}_Q^{*}|} 
      \Bigr|\\
      &= \|D\|\cdot \sum_{R\in \mathcal{B}_P} \frac{|R|}{|\mathcal{B}_P^{*}|}\cdot \Bigl|
        1 - \frac{|\mathcal{B}_P^{*}|}{|\mathcal{B}_Q^{*}|} \frac{|\mathcal{B}_Q^{*}\cap R|}{|R|}
        \Bigr|.
    \end{aligned}
  \end{equation}
  Recall that $Q\subset P$ and $R\in \mathcal{B}_P$. Thus, by~\textrefp[C]{enu:dfn:compatibility:iii} and~\eqref{eq:34} we obtain
  \begin{equation*}
    (1 - \kappa_P)
    \leq \frac{|\mathcal{B}_P^{*}|}{|\mathcal{B}_Q^{*}|} \frac{|\mathcal{B}_Q^{*}\cap R|}{|R|}
    \leq (1 - \kappa_P)^{-1},
  \end{equation*}
  Consequently,
  \begin{equation*}
    \Bigl|
    1 - \frac{|\mathcal{B}_P^{*}|}{|\mathcal{B}_Q^{*}|} \frac{|\mathcal{B}_Q^{*}\cap R|}{|R|}
    \Bigr|
    \leq \frac{\kappa_P}{1 - \kappa_P}.
  \end{equation*}
  Inserting the latter estimate into~\eqref{eq:38}, applying~\textrefp[C]{enu:dfn:compatibility:i} and subsequently using~\eqref{eq:37} yields
  \begin{equation*}
    \alpha
    \leq \sum_{R\in \mathcal{B}_P} \frac{|R|}{|\mathcal{B}_P^{*}|} \frac{\kappa_P}{1 - \kappa_P} \|D\|
    \leq \frac{\kappa_P}{1 - \kappa_P} \|D\|
    \leq \varepsilon.
  \end{equation*}

  Combining~\eqref{eq:18}, \eqref{eq:54} and the above estimate gives us
  \begin{equation}\label{eq:36}
    |A_i^{\widetilde{D},k}(x) - A_j^{\widetilde{D},k}(x)|
    \leq 2\varepsilon.
  \end{equation}
  To summarize, we proved that $(A_n^{\widetilde{D},k})_{n=0}^{\infty}$ is a Cauchy sequence with respect to uniform norm in $[0,1)$ for all
  $1\leq k\leq r$.

  By passing to a suitable subsequence $(A_{n_i}^{\widetilde{D},k})_{i=0}^{\infty}$ of $(A_n^{\widetilde{D},k})_{n=0}^{\infty}$ we can assume
  that~\eqref{eq:pro:stab-3:1} holds for $(A_{n_i}^{\widetilde{D},k})_{i=0}^{\infty}$.  Applying \Cref{lem:block-basis} once more, this time to
  $U = [0,1)$ and $(n_i)_{i=0}^{\infty}$, we obtain a faithful sequence of collections $(\mathcal{B}_P' : P\in \mathcal{A})$ (see
  \Cref{rem:block-basis:3}), which determines a new pair of canonical factorization operators $A',B'\colon Z_V\to Z_V$ according to
  \Cref{dfn:canonical}.  If we put $\widehat{D} = A'\widetilde{D}B'$, then by the above calculation, \eqref{eq:pro:stab-3:1} holds for
  $(A_n^{\widehat{D},k})_{n=0}^{\infty}$.  Clearly $\gamma_X^u(1) = \gamma_X^d(1) = 1$, and since $\widehat{D} = A'ADBB'$ projectionally factors
  through $D$, \Cref{thm:canonical} and \Cref{thm:canonicalkp} give us that the factorization constant is $1+\delta$.  Since $\delta$ can be made
  arbitrarily small, we have proved that $\widehat{D}$ projectionally factors through $D$ with constant $1^+$.  Finally, note that by
  \Cref{lem:functional-stability}, we also have $\lambda_{\mathcal{U}}^k(\widehat{D}) = \lambda_{\mathcal{U}}^k(\widetilde{D})$, $1\leq k\leq r$,
  which combined with~\eqref{eq:183} gives us~\eqref{eq:181}.
\end{myproof}

\subsection{Pointwise multipliers align with the finest bounded partition}
\label{sec:pointw-mult-conf}

At this stage of the stabilization procedure, the associated adapted processes converge almost everywhere to a pointwise multiplier.
\Cref{thm:agreement} below asserts that these pointwise multipliers must be aligned with the finest bounded partition associated with the underlying
Janson-Sobolev or Janson-Sobolev-Hardy space.  This observation will be used in the final step of the stabilization procedure in
\Cref{sec:reduct-canon-proj}, where we will reduce any multiplier to a linear combination of canonical projections corresponding to the finest bounded
partition.

\begin{lem}\label{lem:agreement}
  Let $D = (d_P^k : P\in \mathcal{A},\ 1\leq k\leq r)$ denote a bounded diagonal operator such that $(A_n^{D,k})_{n=0}^{\infty}$ converges almost
  everywhere to $M^{D,k}$ for all $1\leq k\leq r$.  Suppose there exist $1\leq k_0\neq k_1 \leq r$ such that the set
  \begin{equation}\label{eq:15}
    S:= \{ M^{D,k_0} \neq M^{D,k_1} \}
    \quad\text{has positive measure}.
  \end{equation}
  Then there exists a partition of $\{1, \dots, r\}$ into $\mathcal{K}\ni k_0$ and $\mathcal{L}\ni k_1$ such that
  \begin{equation*}
    Q_{\mathcal{K}} = \sum_{k\in \mathcal{K}} Q_k
    \qquad\text{and}\qquad
    Q_{\mathcal{L}} = \sum_{k\in \mathcal{L}} Q_k  
  \end{equation*}
  are both bounded operators on $Z_V$.
\end{lem}

Before we prove \Cref{lem:agreement}, we state and prove the following consequence.
\begin{thm}\label{thm:agreement}
  Let $D = (d_P^k : P\in \mathcal{A},\ 1\leq k\leq r)$ denote a bounded diagonal operator such that $(A_n^{D,k})_{n=0}^{\infty}$ converges to
  $M^{D,k}$ on $[0,1)$ for all $1\leq k\leq r$.  Suppose that $\mathcal{K}_1,\dots, \mathcal{K}_s$ is the finest bounded partition of
  $\{1,\dots, r\}$.  Then $M^{D,k} = M^{D,k'}$ almost everywhere on $[0,1)$ for all $k,k'\in \mathcal{K}_t$ and $1\leq t\leq s$.
\end{thm}

\begin{proof}[Proof of \Cref{thm:agreement}]
  Let $\mathcal{K}_1,\dots, \mathcal{K}_s$ denote the finest bounded partition of $\{1,\dots, r\}$ according to \Cref{dfn:finest-bounded-partition}.
  Suppose to the contrary that there exist $1\leq t_0\leq s$ and $k_0\neq k_1\in \mathcal{K}_{t_0}$ such that the set $\{ M^{D,k_0} \neq M^{D,k_1} \}$
  has positive measure.  Utilizing \Cref{lem:agreement} shows that there exists a partition $\mathcal{K}\ni k_0$ and $\mathcal{L}\ni k_1$ of
  $\{1,\dots, r\}$ such that the projections
  \begin{equation*}
    Q_{\mathcal{K}} = \sum_{k\in \mathcal{K}} Q_k
    \qquad\text{and}\qquad
    Q_{\mathcal{L}} = \sum_{k\in \mathcal{L}} Q_k
  \end{equation*}
  are bounded operators on $Z_V$.  Consequently, the operator $Q_{\mathcal{K}} Q_{\mathcal{K}_{t_0}} = Q_{\mathcal{K}\cap \mathcal{K}_{t_0}}$ is also
  bounded.  On the other hand, since $\emptyset\neq \mathcal{K}\cap \mathcal{K}_{t_0}\subsetneq \mathcal{K}_{t_0}$ and
  $\mathcal{K}_1,\dots, \mathcal{K}_s$ is the finest bounded partition, we arrive at the contradiction that the operator
  $Q_{\mathcal{K}\cap \mathcal{K}_{t_0}}$ is unbounded.
\end{proof}

\begin{proof}[Proof of \Cref{lem:agreement}]
  First, we choose a subset $U$ of $S$ with positive measure, such that there exists a partition $\mathcal{K}_1,\dots, \mathcal{K}_s$ of
  $\{1, \dots, r\}$ as well as constants $c_1, \dots, c_s$ and $\delta > 0$ with the following properties:
  \begin{enumerate}[(a)]
  \item\label{enu:lem:agreement:a} $M^{D,k} = M^{D,k'}$ almost everywhere on $U$ for all $1\leq t\leq s$ and $k,k'\in \mathcal{K}_{t}$,
  \item\label{enu:lem:agreement:b} $|M^{D,k}(x) - c_t| \leq \frac{\delta}{4}$ for almost all $x\in U$, $1\leq t\leq s$ and $k\in \mathcal{K}_{t}$,
  \item\label{enu:lem:agreement:c} $|c_t - c_{t'}| \geq \delta$ for all $1\leq t\neq t'\leq s$, and
  \item\label{enu:lem:agreement:d} $(A_n^{D,k})_{n=0}^{\infty}$ converges uniformly to $M^{D,k}$ on $U$, for all $1\leq k\leq r$.
  \end{enumerate}
  Since $U$ is a subset of $S$, there exist $1\leq t_0\neq t_1\leq s$ such that $k_0\in \mathcal{K}_{t_0}$ and $k_1\in \mathcal{K}_{t_1}$.

  Recall that $\lambda_{Z_V}$ denotes the basis constant of $(\varphi_P^k : P\in \mathcal{A},\ 1\leq k\leq r)$ and pick any sequence
  $(\eta_i)_{i=0}^{\infty}$ of positive numbers and $\eta > 0$ such that
  \begin{equation}\label{eq:85}
    4\lambda_{Z_V} r \sum_{i=0}^{\infty} (\# \mathcal{AF}_i) \eta_i
    \leq \eta
    < \min(1,\tfrac{\delta}{16}).
  \end{equation}
  Next, pick any sequence of positive numbers $(\kappa_P : P\in \mathcal{A})$ with $0\leq \kappa_P < 1$ and
  $\sum_{P\in \mathcal{A}} \kappa_P < \infty$.  As in the proof of \Cref{pro:stab-3}, we apply \Cref{lem:block-basis} to our $U$, $n_i= i$,
  $i\in \mathbb{N}_0$ and obtain a $(1 - \kappa_P : P\in \mathcal{A})$-faithful sequence of collections $(\mathcal{B}_P : P\in \mathcal{A})$ with
  respect to $(m_i)_{i=0}^{\infty}$ mounted on $U$.  Let $(\psi_P^k : P\in \mathcal{A},\ 1\leq k\leq r)$ and $A,B\colon Z_V\to Z_V$ denote the
  associated block basis and canonical factorization operators according to \Cref{dfn:canonical}, and put $\widetilde{D} = ADB$.  By
  \Cref{thm:canonicalkp}, the operators $A$ and $B$ are bounded and therefore $\widetilde{D}$ is too.  In the proof of \Cref{pro:stab-3}, it was shown
  that
  \begin{equation}\label{eq:53}
    \bigl| A_i^{\widetilde{D},k}(x) - A_j^{\widetilde{D},k}(x) \bigr|
    \leq \eta_i + \eta_j,
    \qquad x\in [0,1),\ 1\leq k\leq r.
  \end{equation}
  We claim that
  \begin{enumerate}[(a)]
    \setcounter{enumi}{4}
  \item\label{enu:lem:agreement:e} $M^{\widetilde{D},k} = M^{\widetilde{D},k'}$ everywhere on $[0,1)$ for all $1\leq t\leq s$ and
    $k,k'\in \mathcal{K}_{t}$,
  \item\label{enu:lem:agreement:f} $|M^{\widetilde{D},k}(x) - c_t| \leq \frac{\delta}{4}$ for all $x\in [0,1)$, $1\leq t\leq s$ and
    $k\in \mathcal{K}_{t}$,
  \item\label{enu:lem:agreement:g} $|c_t - c_{t'}| \geq \delta$ for all $1\leq t\neq t'\leq s$, and
  \item\label{enu:lem:agreement:h} $(A_i^{\widetilde{D},k})_{i=0}^{\infty}$ converges uniformly on $[0,1)$ to $M^{\widetilde{D},k}$ for all
    $1\leq k\leq r$.
  \end{enumerate}

  Property~\eqref{enu:lem:agreement:c} is identical to Property~\eqref{enu:lem:agreement:g}, and thus requires no verification.  Furthermore, by
  equation \eqref{eq:53}, the sequence $(A_i^{\widetilde{D}, k})_{i=0}^{\infty}$ converges uniformly on the interval $[0, 1)$, thereby verifying
  condition \eqref{enu:lem:agreement:h}.  We will now show that~\eqref{enu:lem:agreement:e} holds true.  Note that
  $\widetilde{D} = (\widetilde{d}_P^k : P\in \mathcal{A},\ 1\leq k\leq r)$ is given by
  \begin{equation}\label{eq:16}
    \widetilde{d}_P^k
    = \frac{\langle \psi_P^k, D \psi_P^k \rangle}{|\mathcal{B}_P^{*}|}
    = \sum_{R\in \mathcal{B}_{P}} \frac{|R|}{|\mathcal{B}_P^{*}|} d_R^k,
    \qquad P\in \mathcal{A},\ 1\leq k\leq r.
  \end{equation}
  Let $x\in P\in \mathcal{AF}_i$ for some $i\in \mathbb{N}_0$, and note that $\mathcal{B}_P\subset \mathcal{AF}_{m_i}$ by
  \textrefp[C]{enu:dfn:compatibility:iv}.  Hence, by the definition of $A_i^{\widetilde{D},k}$ (see~\eqref{eq:1}) and the above
  identity~\eqref{eq:16}, we obtain
  \begin{equation}\label{eq:55}
    A_i^{\widetilde{D},k}(x)
    = \sum_{R\in \mathcal{B}_{P}} \frac{|R|}{|\mathcal{B}_P^{*}|} \frac{1}{|R|} \int_{R} A_{m_i}^{D,k}
    = \frac{1}{|\mathcal{B}_P^{*}|} \int_{\mathcal{B}_P^{*}} A_{m_i}^{D,k},
  \end{equation}
  for all $x\in P\in \mathcal{AF}_i$, $i\in \mathbb{N}_0$ and $1\leq k\leq r$.

  Now let $1\leq t\leq s$, pick any $k,k'\in \mathcal{K}_t$ and let $x\in [0,1)$ be fixed.  Moreover, we fix $i\in \mathbb{N}_0$,
  $x\in P\in \mathcal{AF}_{m_i}$ and note that by~\eqref{eq:55}
  \begin{equation}\label{eq:79}
    \begin{aligned}
      |A_i^{\widetilde{D},k}(x) - A_i^{\widetilde{D},k'}(x)|
      &\leq \frac{1}{|\mathcal{B}_P^{*}|} \int_{\mathcal{B}_P^{*}} |A_{m_i}^{D,k} - A_{m_i}^{D,k'}|\\
      &\leq \frac{1}{|\mathcal{B}_P^{*}|} \int_{\mathcal{B}_P^{*}\cap U} |A_{m_i}^{D,k} - A_{m_i}^{D,k'}|
        + \frac{1}{|\mathcal{B}_P^{*}|} \int_{\mathcal{B}_P^{*}\setminus U} |A_{m_i}^{D,k} - A_{m_i}^{D,k'}|\\
      &=: \alpha + \beta.
    \end{aligned}
  \end{equation}
  First, we will estimate $\alpha$ by
  \begin{equation}\label{eq:60}
    \alpha
    \leq \|A_{m_i}^{D,k} - A_{m_i}^{D,k'}\|_{L^{\infty}(U)}.
  \end{equation}
  For $\beta$, recall that $\|A_{m_i}^{D,k}\|_{L^{\infty}}\leq \|D\|$, and hence
  \begin{equation}\label{eq:41}
    \beta
    \leq 2\frac{|\mathcal{B}_P^{*}\setminus U|}{|\mathcal{B}_P^{*}|} \|D\|.
  \end{equation}
  By~\textrefp[C]{enu:dfn:compatibility:v}, we obtain
  \begin{equation}\label{eq:82}
    |\mathcal{B}_P^{*}\setminus U|
    = \sum_{R\in \mathcal{B}_{P}} |R\setminus U|
    \leq \kappa_P \sum_{R\in \mathcal{B}_{P}} |R|
    = \kappa_P |\mathcal{B}_P^{*}|,
  \end{equation}
  which together with~\eqref{eq:41} gives us
  \begin{equation}\label{eq:78}
    \beta
    \leq 2\kappa_P \|D\|.
  \end{equation}
  Combining~\eqref{eq:79}, \eqref{eq:60} and~\eqref{eq:78} yields
  \begin{equation*}
    |A_i^{\widetilde{D},k}(x) - A_i^{\widetilde{D},k'}(x)|
    \leq \|A_{m_i}^{D,k} - A_{m_i}^{D,k'}\|_{L^{\infty}(U)} + 2\kappa_P \|D\|,
  \end{equation*}
  for all $i\in \mathbb{N}_0$ and $x\in P\in \mathcal{AF}_{m_i}$.  By Property~\eqref{enu:lem:agreement:d}, the first term on the right hand side of the preceding
  inequality converges to $0$ as $i\to \infty$.  For the second term, recall that $P\in \mathcal{AF}_{m_i}$ and that $\kappa_P\to 0$ for $|P| \to 0$. Hence the second
  term also converges to $0$ as $i\to \infty$.  Together, we proved that
  \begin{equation*}
    \lim_i |A_i^{\widetilde{D},k} - A_i^{\widetilde{D},k'}| = 0
    \quad \text{uniformly on $[0,1)$}.
  \end{equation*}
  Appealing to Property~\eqref{enu:lem:agreement:h}, we conclude the proof of Property~\eqref{enu:lem:agreement:e}.

  To verify Property~\eqref{enu:lem:agreement:f}, we essentially replicate the argument used in the proof of Property~\eqref{enu:lem:agreement:e}.
  Let $1\leq t\leq s$, $k\in \mathcal{K}_t$, $x\in P\in \mathcal{AF}_i$ and note that by~\eqref{eq:55}
  \begin{equation}\label{eq:61}
    \begin{aligned}
      |A_i^{\widetilde{D},k}(x) - c_t|
      &\leq \frac{1}{|\mathcal{B}_P^{*}|} \int_{\mathcal{B}_P^{*}} |A_{m_i}^{D,k} - c_t|\\
      &\leq \frac{1}{|\mathcal{B}_P^{*}|} \int_{\mathcal{B}_P^{*}\cap U} |A_{m_i}^{D,k} - c_t|
        + \frac{1}{|\mathcal{B}_P^{*}|} \int_{\mathcal{B}_P^{*}\setminus U} |A_{m_i}^{D,k} - c_t|\\
      &= \alpha' + \beta'
    \end{aligned}
  \end{equation}
  The following estimates for $\alpha'$ and $\beta'$ are obtained analogously to those for $\alpha$ and $\beta$.  Exploiting
  Property~\eqref{enu:lem:agreement:b}, we obtain
  \begin{equation}\label{eq:62}
    \begin{aligned}
      \alpha'
      &\leq \frac{1}{|\mathcal{B}_P^{*}|} \int_{\mathcal{B}_P^{*}\cap U} |A_{m_i}^{D,k} - M^{D,k}|
        + \frac{1}{|\mathcal{B}_P^{*}|} \int_{\mathcal{B}_P^{*}\cap U} |M^{D,k} - c_t|\\
      &\leq \|A_{m_i}^{D,k} - M^{D,k}\|_{L^{\infty}(U)}
        + \tfrac{\delta}{4}.
    \end{aligned}
  \end{equation}
  Since $\|A_{m_i}^{D,k}\|_{L^{\infty}}\leq \|D\|$, \eqref{eq:82} gives us
  \begin{equation}\label{eq:63}
    \beta'
    \leq \kappa_P(\|D\| + |c_t|).
  \end{equation}
  Combining~\eqref{eq:62} with~\eqref{eq:63} we obtain
  \begin{equation*}
    |A_i^{\widetilde{D},k}(x) - c_t|
    \leq \|A_{m_i}^{D,k} - M^{D,k}\|_{L^{\infty}(U)} + \tfrac{\delta}{4} + \kappa_P(\|D\| + |c_t|).
  \end{equation*}
  Finally, note that since $\lim_i \|A_{m_i}^{D,k} - M^{D,k}\|_{L^{\infty}(U)} = 0$ by Property~\eqref{enu:lem:agreement:d}, and since $\kappa_P\to 0$
  whenever $|P| \to 0$, we have established Property~\eqref{enu:lem:agreement:f}.

  Above, we have just established that the bounded diagonal operator $\widetilde{D} = (\widetilde{d}_P^k : P\in \mathcal{A},\ 1\leq k\leq r)$
  satisfies the Properties~\eqref{enu:lem:agreement:e}--\eqref{enu:lem:agreement:h}.  Next, pick $k_t\in \mathcal{K}_t$ for each $1\leq t\leq s$, and
  define the following (a priori possibly unbounded) diagonal operator $\widehat{D} = (\widehat{d}_P^k : P\in \mathcal{A},\ 1\leq k\leq r)$ by
  \begin{equation}\label{eq:83}
    \widehat{d}_P^k
    = \widetilde{d}_P^{k_t},
    \qquad P\in \mathcal{A},\ 1\leq t\leq s,\ k\in \mathcal{K}_t.
  \end{equation}
  We claim that
  \begin{equation}\label{eq:84}
    \|\widehat{D} - \widetilde{D}\colon Z_V\to Z_V\|
    \leq \eta < \infty;
  \end{equation}
  in particular, $\widehat{D}\colon Z_V\to Z_V$ is bounded.  Now let $i\in \mathbb{N}_0$, $P\in \mathcal{AF}_i$, $1\leq t\leq s$, $k\in \mathcal{K}_t$
  be fixed.  Observe that, by~\eqref{eq:1} and since $k_t\in \mathcal{K}_t$, it follows from Property~\eqref{enu:lem:agreement:e} and~\eqref{eq:53}
  that
  \begin{equation*}
    |\widehat{d}_P^k - \widetilde{d}_P^k|
    = |\widetilde{d}_P^{k_t} - \widetilde{d}_P^k|
    \leq |A_i^{\widetilde{D},k_t}(x) - M^{\widetilde{D},k_t}(x)| + |A_i^{\widetilde{D},k}(x) - M^{\widetilde{D},k}(x)|
    \leq 2\eta_i,
  \end{equation*}
  for all $x\in P\in \mathcal{AF}_i$.  Thus, by \Cref{lem:schauder} and~\eqref{eq:85}, we obtain
  \begin{equation*}
    \|\widehat{D} - \widetilde{D}\colon Z_V\to Z_V\|
    \leq 4\lambda_{Z_V} \sum_{i=0}^{\infty} \sum_{P\in \mathcal{AF}_i} \sum_{k=1}^r \eta_i
    = 4\lambda_{Z_V} r \sum_{i=0}^{\infty} (\# \mathcal{AF}_i) \eta_i
    = \eta
    < \infty,
  \end{equation*}
  as claimed in~\eqref{eq:84}.

  Next, we define for each $1\leq t\leq s$ the diagonal operator $D_t = (d_{t,P}^k : P\in \mathcal{A},\ 1\leq k\leq r)$ by
  \begin{equation}\label{eq:87}
    d_{t,P}^k
    = \widehat{d}_P^{k_t},
    \qquad P\in \mathcal{A},\ 1\leq k\leq r.
  \end{equation}
  We will now prove that $D_t\colon Z_V\to Z_V$ is bounded for each $1\leq t\leq s$.  Let $1\leq t_0\leq s$ be fixed and define the martingale
  operator $\mathring{D} = (\mathring{d}_P^k : P\in \mathcal{A},\ 1\leq k\leq r)$ by
  \begin{equation}\label{eq:88}
    \mathring{d}_P^k
    = d_{t_0,[0,1)}^k
    = \widehat{d}_{[0,1)}^{k_{t_0}},
    \qquad P\in \mathcal{A},\ 1\leq k\leq r.
  \end{equation}

  First, we will verify that $\mathring{D}\colon Z_V\to Z_V$ is indeed bounded. Let $f = \sum_{P,k} a_P^k \varphi_P^k$ denote a finite sum and note
  that
  \begin{align*}
    \|\mathring{D} f\|_{Z_V}
    &= \bigl\| \sum_{P,k} \mathring{d}_P^k a_P^k \varphi_P^k \bigr\|_{Z_V}
      = \bigl\| \sum_{P\in \mathcal{A}}  \sum_{k=1}^r \widehat{d}_{[0,1)}^{k_{t_0}} a_P^k \varphi_P^k \bigr\|_{Z_V}\\
    &= |\widehat{d}_{[0,1)}^{k_{t_0}}|\cdot \|f\|_{Z_V}
      \leq  \|\widehat{D}\|\cdot \|f\|_{Z_V}.
  \end{align*}

  Secondly, applying \Cref{pro:variation-norm} to the trivial partition $\{1,\dots, r\}$ of $\{1,\dots, r\}$, recalling that
  $d_{t_0,P}^k = d_{t_0,P}^{k_{t_0}}$ for all $P\in \mathcal{A}$, $1\leq k\leq r$ (see~\eqref{eq:87}) and
  $\mathring{d}_P^{k_{t_0}} = \mathring{d}_{[0,1)}^{k_{t_0}}$ for all $P\in \mathcal{A}$ (see~\eqref{eq:88}) yields
  \begin{equation}\label{eq:90}
    \begin{aligned}
      &\|D_{t_0} - \mathring{D}\colon Z_V\to Z_V\|
        \leq 2\cdot \|D_{t_0} - \mathring{D}\|_{\{1,\dots, r\}}\\
      &\qquad \leq 2\cdot \bigl(
        |d_{t_0,[0,1)}^{k_{t_0}} - \mathring{d}_{[0,1)}^{k_{t_0}}|
        + \sum_{\substack{P\in \mathcal{A}\\P\subsetneq [0,1)}} |(d_{t_0,\pi(P)}^{k_{t_0}} - \mathring{d}_{\pi(P)}^{k_{t_0}}) - (d_{t_0,P}^{k_{t_0}} - \mathring{d}_P^{k_{t_0}})|
      \bigr)\\
      &\qquad = 2\cdot \sum_{\substack{P\in \mathcal{A}\\P\subsetneq [0,1)}} |d_{t_0,\pi(P)}^{k_{t_0}} - d_{t_0,P}^{k_{t_0}}|
      = 2\cdot \sum_{\substack{P\in \mathcal{A}\\P\subsetneq [0,1)}} |\widehat{d}_{\pi(P)}^{k_{t_0}} - \widehat{d}_P^{k_{t_0}}|.
    \end{aligned}
  \end{equation}
  By~\eqref{eq:83} and~\eqref{eq:1}, we see that $\widehat{d}_P^{k_{t_0}} = \widetilde{d}_P^{k_{t_0}} = A_i^{\widetilde{D},k_{t_0}}(x)$, whenever
  $x\in P\in \mathcal{AF}_i$, and thus, by~\eqref{eq:53} we obtain
  \begin{equation}\label{eq:89}
    |\widehat{d}_{\pi(P)}^{k_{t_0}} - \widehat{d}_P^{k_{t_0}}|
    \leq \eta_{i-1} + \eta_i,
    \qquad P\in \mathcal{AF}_i,\ i\in \mathbb{N}.
  \end{equation}
  Combining~\eqref{eq:90} with~\eqref{eq:89} and recalling the definition of $\eta$ (see~\eqref{eq:85}) yields
  \begin{equation*}
    \|D_{t_0} - \mathring{D}\colon Z_V\to Z_V\|
    \leq 2\cdot \sum_{i=1}^{\infty} (\#\mathcal{AF}_i) (\eta_{i-1} + \eta_i)
    < \infty.
  \end{equation*}
  In particular, the operators $D_t\colon Z_V\to Z_V$, $1\leq t\leq s$ are all bounded.

  Our next objective is to verify the prerequisites of the Vandermonde technique \Cref{thm:vantrick1} applied in the following setting: We put
  $S_t := D_t$, $g_t:= Q_{\mathcal{K}_t}$, $1\leq t\leq s$ and $T_j := \widehat{D}^j$, $0\leq j\leq s$.

  First, we recall that, as established above, the operator $D_t\colon Z_V\to Z_V$ is continuous for each $1\leq t\leq s$, verifying
  \Cref{thm:vantrick1}~\eqref{enu:th:vantrick1:i}.  Secondly, observe that the operators $D_t$, $1 \leq t \leq s$, commute with one another, which
  gives us \Cref{thm:vantrick1}~\eqref{enu:th:vantrick1:ii}.  Thirdly, note the identity
  \begin{equation}\label{eq:100}
    \widehat{D}
    = \sum_{t=1}^s D_t Q_{\mathcal{K}_t}.
  \end{equation}
  Since $\widehat{D}$ is a bounded operator, taking powers $j=0,\dots, s$ in the above identity shows that the operators $T_j\colon Z_V\to Z_V$,
  $0\leq j\leq s$ are all bounded, which establishes \Cref{thm:vantrick1}~\eqref{enu:th:vantrick1:iii}.  Moreover, note that
  \begin{equation}\label{eq:94}
    T_j
    = \Bigl( \sum_{t=1}^s D_t Q_{\mathcal{K}_t} \Bigr)^j
    = \sum_{t=1}^s D_t^j Q_{\mathcal{K}_t},
    \qquad 0\leq j\leq s,
  \end{equation}
  which shows that $T_j$ has the required algebraic form~\eqref{eq:77}.  The latter identity is a consequence of the following observations:
  \begin{itemize}
  \item $D_{t_1} Q_{\mathcal{K}_{t_1}} D_{t_0} Q_{\mathcal{K}_{t_0}} = D_{t_1} D_{t_0} Q_{\mathcal{K}_{t_1}} Q_{\mathcal{K}_{t_0}}$ for all
    $1\leq t_0, t_1\leq s$,
  \item $Q_{\mathcal{K}_{t_1}} Q_{\mathcal{K}_{t_0}} = 0$ for all $1\leq t_0\neq t_1 \leq s$, and
  \item $Q_{\mathcal{K}_{t}}$ is a projection for $1\leq t\leq s$.
  \end{itemize}
  The final ingredient for the application of the Vandermonde technique is to establish \Cref{thm:vantrick1}~\eqref{enu:th:vantrick1:iv}, i.e., that
  the operators $D_{t_0}- D_{t_1}$, $1\leq t_0\neq t_1\leq s$ are continuously invertible.  To accomplish this goal, we want to apply the additional
  part of \Cref{pro:variation-norm}.  Clearly, $D_{t_0}- D_{t_1} = (d_{t_0,P}^{k_{t_0}} - d_{t_1,P}^{k_{t_1}} : P\in \mathcal{A},\ 1\leq k\leq r)$ is
  a diagonal operator, and by~\eqref{eq:87} as well as~\eqref{eq:83}, we obtain
  \begin{equation}\label{eq:95}
    \begin{aligned}
      |(d_{t_0,[0,1)}^{k_{t_0}} - d_{t_1,[0,1)}^{k_{t_1}}) - (d_{t_0,P}^{k_{t_0}} - d_{t_1,P}^{k_{t_1}})|
      &= |(\widetilde{d}_{[0,1)}^{k_{t_0}} - \widetilde{d}_{[0,1)}^{k_{t_1}}) - (\widetilde{d}_P^{k_{t_0}} - \widetilde{d}_P^{k_{t_1}})|\\
      &\leq |\widetilde{d}_{[0,1)}^{k_{t_0}} - \widetilde{d}_P^{k_{t_0}}| + |\widetilde{d}_{[0,1)}^{k_{t_1}} - \widetilde{d}_P^{k_{t_1}}|.
    \end{aligned}
  \end{equation}
  Recalling that by~\eqref{eq:1}, $\widetilde{d}_P^{k_{t_0}} = A_i^{\widetilde{D},k_{t_0}}(x)$, whenever $x\in P\in \mathcal{AF}_i$,
  $i\in \mathbb{N}_0$, which combined with~\eqref{eq:53}, \eqref{eq:95} and~\eqref{eq:85} yields for $x\in P\in \mathcal{AF}_i$
  \begin{equation}\label{eq:64}
    \begin{aligned}
      |(d_{t_0,P_0}^{k_{t_0}} - d_{t_1,P_0}^{k_{t_1}}) - (d_{t_0,P}^{k_{t_0}} - d_{t_1,P}^{k_{t_1}})|
      &\leq |A_0^{\widetilde{D},k_{t_0}}(x) - A_i^{\widetilde{D},k_{t_0}}(x)| + |A_0^{\widetilde{D},k_{t_1}}(x) - A_i^{\widetilde{D},k_{t_1}}(x)|\\
      &\leq \eta_0 + \eta_i \leq \eta.
    \end{aligned}
  \end{equation}
  Similarly, we can show that for all $x\in [0,1)$, we have
  \begin{align*}
    |d_{t_0,[0,1)}^{k_{t_0}} - d_{t_1,[0,1)}^{k_{t_1}}|
    &= |A_0^{\widetilde{D},k_{t_0}}(x) - A_0^{\widetilde{D},k_{t_1}}(x)|\\
    &\geq |M^{\widetilde{D},k_{t_0}}(x) - M^{\widetilde{D},k_{t_1}}(x)|
      - |A_0^{\widetilde{D},k_{t_0}}(x) - M^{\widetilde{D},k_{t_0}}(x)|\\
    &\qquad - |A_0^{\widetilde{D},k_{t_1}}(x) - M^{\widetilde{D},k_{t_1}}(x)|.
  \end{align*}
  Appealing to~\eqref{eq:85} as well as the Properties~\eqref{enu:lem:agreement:f} and~\eqref{enu:lem:agreement:g} together with~\eqref{eq:53} and the
  above estimate yields
  \begin{equation}\label{eq:99}
    \begin{aligned}
      |d_{t_0,[0,1)}^{k_{t_0}} - d_{t_1,[0,1)}^{k_{t_1}}|
      &\geq |M^{\widetilde{D},k_{t_0}}(x) - M^{\widetilde{D},k_{t_1}}(x)| - 2\eta_0\\
      &\geq |c_{t_0} - c_{t_1}| - \tfrac{2\delta}{4} - \tfrac{\delta}{4}
        \geq \tfrac{\delta}{4}.
    \end{aligned}
  \end{equation}

  Observe that by~\eqref{eq:99}, \eqref{eq:64} and~\eqref{eq:85}, we have
  \begin{equation}\label{eq:65}
    \begin{aligned}
      |d_{t_0,P}^{k_{t_0}} - d_{t_1,P}^{k_{t_1}}|
      &\geq |d_{t_0,[0,1)}^{k_{t_0}} - d_{t_1,[0,1)}^{k_{t_1}}|
        - |d_{t_0,P}^{k_{t_0}} - d_{t_0,[0,1)}^{k_{t_0}}|
        - |d_{t_1,P}^{k_{t_1}} - d_{t_1,[0,1)}^{k_{t_1}}|\\
      &\geq \tfrac{\delta}{4} - 2\eta
        \geq \tfrac{\delta}{8},
    \end{aligned}
  \end{equation}
  and hence~\eqref{eq:pro:variation-norm:3} of \Cref{pro:variation-norm} is satisfied for
  $D_{t_0} - D_{t_1} = (d_{t_0,P}^{k_{t_0}} - d_{t_1,P}^{k_{t_1}} : P\in \mathcal{A},\ 1\leq k\leq r)$.

  By~\eqref{eq:64} and~\eqref{eq:99}, we obtain
  \begin{equation}\label{eq:96}
    |(d_{t_0,[0,1)}^k - d_{t_1,[0,1)}^k) - (d_{t_0,P}^k - d_{t_1,P}^k)|
    \leq \eta
    \leq \frac{4}{\delta}\eta |d_{t_0,[0,1)}^k - d_{t_1,[0,1)}^k|,
  \end{equation}
  where we obtain from~\eqref{eq:85} that $0 < \frac{4}{\delta}\eta < 1$.
  
  By the above, all the prerequisites for applying the additional part of \Cref{pro:variation-norm} have been met:
  \begin{itemize}
  \item We refer to the definition of $D_t$ (see~\eqref{eq:87}) for the general hypothesis of \Cref{pro:variation-norm} to see that
    $d_{t,P}^k = d_{t,P}^{k'}$ for all $P\in \mathcal{A}$ and $k\neq k'$.
  \item By~\eqref{eq:96}, we see that~\eqref{eq:pro:variation-norm:4} is satisfied for $D_{t_0} - D_{t_1}$ and $\frac{4}{\delta}\eta$.
  \item Note that by~\eqref{eq:99}, the inequality in~\eqref{eq:pro:variation-norm:3} is satisfied for $\frac{\delta}{8}$, and by~\eqref{eq:65} we
    also obtain that $d_{t_0,P}^{k_{t_0}} - d_{t_1,P}^{k_{t_1}} \neq 0$.
  \end{itemize}
  Finally, we are in a position to apply \Cref{pro:variation-norm}, which gives us the boundedness of $(D_{t_0} - D_{t_1})^{-1}$:
  \begin{equation}\label{eq:97}
    \|(D_{t_0} - D_{t_1})^{-1}\|_{\{1,\dots, r\}}
    \leq \frac{64}{(1-\frac{4}{\delta}\eta)^2 \delta^2}\cdot \|D_{t_0} - D_{t_1}\|_{\{1,\dots, r\}}
    < \infty.
  \end{equation}

  With~\eqref{eq:97} established, all the prerequisites for the Vandermonde technique \Cref{thm:vantrick1} have also been met:
  Recalling~\eqref{eq:100} and applying \Cref{thm:vantrick1} finally yields the boundedness of the projections $Q_{\mathcal{K}_t}$, $1\leq t\leq s$.
  Recall that there exist $1\leq t_0\neq t_1\leq s$ such that $k_0\in \mathcal{K}_{t_0}$ and $k_1\in \mathcal{K}_{t_1}$, so we complete the proof by
  putting $\mathcal{K} = \mathcal{K}_{t_0}$ and $\mathcal{L} = \{1,\dots, r\}\setminus \mathcal{K}$.
\end{proof}

\subsection{Reduction to canonical projections}
\label{sec:reduct-canon-proj}

In this final step of the stabilization procedure, we reduce any diagonal operator via factorization to a linear combination of canonical projections
associated with the finest bounded partition of the underlying space.  The coefficients in this linear combination can be computed from the data of
the original diagonal operator via linear bounded functionals on the space of diagonal operators.  A further analysis of these linear combinations
permits us to establish the diagonal primary factorization property for Janson-Sobolev and Janson-Sobolev-Hardy spaces.

In our context, the canonical projections play the role of the Capon projection in~\cite[Theorem~2.14]{MR4816119}
(see~\Cref{thm:picewise-constant,thm:main-prop-I}).  For the reader's convenience, we restate our first main theorem,
\Cref{mainthm:picewise-constant}, below.

\begin{thm}[\Cref{mainthm:picewise-constant}]\label{thm:picewise-constant}
  Let $Z_V = Z_V((\mathcal{F}_j),X,\boldsymbol{\sigma};\mathbb{R}^l)$ denote a Janson-Sobolev space or a Janson-Sobolev-Hardy space, and let
  $\mathcal{K}_1, \dots, \mathcal{K}_s\neq \emptyset$ denote the finest bounded partition of $\{1, \dots, r\}$ for the orthonormal basis $\mathcal{B}$ of
  $V$ with respect to $Z_V$.  Suppose that $D\colon Z_V\to Z_V$ is a bounded diagonal operator with respect to
  $(\varphi_P^{k,\mathcal{B}} : P\in \mathcal{A},\ 1\leq k\leq r)$, let $k_t\in \mathcal{K}_t$, $1\leq t\leq s$ and pick any non-principal ultrafilter $\mathcal{U}$ on $\mathbb{N}_0$.  Then $S\colon Z_V\to Z_V$ given by
  \begin{equation}\label{eq:172}
    S = \sum_{t=1}^{s} \lambda_{\mathcal{U}}^{k_t}(D) Q_{\mathcal{K}_t}^{\mathcal{B}}
  \end{equation}
  almost projectionally factors through $D$ with constant~$1^{+}$.
\end{thm}

\begin{proof}
  Let $\eta > 0$, pick a sequence of positive numbers $(\eta_P : P\in \mathcal{A})$ such that
  \begin{equation}\label{eq:132}
    \sum_{P\in \mathcal{A}} \eta_P
    \leq \eta/(8\lambda_{Z_V} r (q+1))
    \leq \eta /(8r (q+1))
  \end{equation}
  and define $\eta_j = \min_{P\in \mathcal{AF}_j} \eta_P > 0$, $j\in \mathbb{N}_0$.  Recall that $\lambda_{Z_V}$ denotes the basis constant of
  $(\varphi_P^{k,\mathcal{B}} : P\in \mathcal{A},\ 1\leq k\leq r)$ and note that $\eta_j\to 0$.

  For any fixed $0 < \kappa < \tfrac{1}{2}$, we combine \Cref{pro:stab-2} with \Cref{pro:stab-3}. Thus, there exists a bounded diagonal operator
  $\widehat{D} = (\widehat{d}_P^k : P\in \mathcal{A},\ 1\leq k\leq r)$ with respect to
  $(\varphi_P^{k,\mathcal{B}} : P\in \mathcal{A},\ 1\leq k\leq r)$ which almost projectionally factors through $D$ with constant $1^{+}$ and satisfies
  the estimate~\eqref{eq:pro:stab-3:1} in \Cref{pro:stab-3} for the sequence $(\eta_j)_{j=0}^{\infty}$ we chose in~\eqref{eq:132}.  In particular,
  $(A_j^{\widehat{D},k})_{j=0}^{\infty}$ converges uniformly on $[0,1)$ to $M^{\widehat{D},k}$ for all $1\leq k\leq r$.  Moreover, we have that
  \begin{equation}\label{eq:184}
    |\lambda_{\mathcal{U}}^k(D) - \lambda_{\mathcal{U}}^k(\widehat{D})|
    \leq \frac{6\kappa}{1-2\kappa} \|D\|,
    \qquad 1\leq k\leq r.
  \end{equation}
  We would like to emphasize that $\|D\|$ in the above estimate arises from applying \Cref{pro:stab-3} to the output, denoted by $D_1$, of
  \Cref{pro:stab-2}, together with the estimate $\|D_1\|\leq \|D\|$.

  Since $\mathcal{K}_1, \dots, \mathcal{K}_s\neq \emptyset$ is the finest bounded partition, \Cref{thm:agreement} forces
  \begin{equation}\label{eq:189}
    M^{\widehat{D},k} = M^{\widehat{D},k'} \text{almost everywhere on $[0,1)$}
  \end{equation}
  for all $k,k'\in \mathcal{K}_t$ and $1\leq t\leq s$.  Letting $j\to \infty$ in~\eqref{eq:pro:stab-3:1} yields
  \begin{equation}\label{eq:188}
    \bigl| A_i^{\widehat{D},k}(x) - M^{\widehat{D},k}(x) \bigr|
    \leq \eta_i,
    \qquad x\in [0,1),\ 1\leq k\leq r,\ i\in \mathbb{N}_0.
  \end{equation}
  Thus, we obtain
  \begin{equation*}
    \bigl| A_i^{\widehat{D},k}(x) - A_j^{\widehat{D},k'}(x) \bigr|
    \leq \bigl| A_i^{\widehat{D},k}(x) - M^{\widehat{D},k}(x) \bigr| + \bigl| M^{\widehat{D},k'}(x) - A_j^{\widehat{D},k'}(x) \bigr|
    \leq \eta_i + \eta_j
  \end{equation*}
  for all $x\in [0,1)$, $k,k'\in \mathcal{K}_t$, $1\leq t\leq s$ and $i,j\in \mathbb{N}_0$.  In other words,
  \begin{equation}\label{eq:133}
    |\widehat{d}_P^k - \widehat{d}_Q^{k'}|
    \leq \eta_P + \eta_Q,
    \quad P\in \mathcal{AF},\ Q\in \mathcal{AF},\ P\subset Q,\ k,k'\in \mathcal{K}_t,\ 1\leq t\leq s.
  \end{equation}

  Now, select $k_t\in \mathcal{K}_t$, $1\leq t\leq s$, and define $\overline{D} = (\overline{d}_P^k : P\in \mathcal{A},\ 1\leq k\leq r)$ by
  \begin{equation}\label{eq:218}
    \overline{d}_P^k
    = \widehat{d}_P^{k_t},
    \qquad P\in \mathcal{AF},\ k\in \mathcal{K}_t,\ 1\leq t\leq s.
  \end{equation}
  Hence, \Cref{lem:schauder} and~\eqref{eq:133} together with~\eqref{eq:132} give us
  \begin{equation}\label{eq:138}
    \begin{aligned}
      \|(\overline{D} - \widehat{D})\colon Z_V\to Z_V\|
      &\leq 2 \lambda_{Z_V} \sum_{P\in \mathcal{A}} \sum_{t=1}^{s}\sum_{k\in \mathcal{K}_t}  |\overline{d}_P^k - \widehat{d}_P^k|\\
      &\leq 4 \lambda_{Z_V} \sum_{P\in \mathcal{A}} \sum_{k=1}^{r} \eta_P
        \leq \frac{\eta}{2}.
    \end{aligned}
  \end{equation}

  Next, we define $\widetilde{D} = (\widetilde{d}_P^k : P\in \mathcal{A},\ 1\leq k\leq r)$ by
  \begin{equation}\label{eq:114}
    \widetilde{d}_P^k
    := \overline{d}_{[0,1)}^k
    = \widehat{d}_{[0,1)}^{k_t},
    \qquad P\in \mathcal{AF},\ k\in \mathcal{K}_t,\ 1\leq t\leq s,
  \end{equation}
  and note that $\widetilde{D} = \sum_{t=1}^{s} \widehat{d}_{[0,1)}^{k_t}\cdot Q_{\mathcal{K}_t}$.

  Using \Cref{pro:variation-norm} and combining~\eqref{eq:133} and~\eqref{eq:132} gives us
  \begin{align*}
    \|\widetilde{D} - \overline{D}\|_{\mathcal{K}_1,\dots, \mathcal{K}_s}
    &\leq \sum_{k=1}^r \Bigl(
      \sum_{\substack{P\in \mathcal{A}\\P\subsetneq [0,1)}} |(\widetilde{d}_{\pi(P)}^k - \overline{d}_{\pi(P)}^k) - (\widetilde{d}_P^k - \overline{d}_P^k)|
    \Bigr)\\
    &= \sum_{k=1}^r \sum_{\substack{P\in \mathcal{A}\\P\subsetneq [0,1)}} |\widehat{d}_{\pi(P)}^k - \widehat{d}_P^k|
    \leq \sum_{k=1}^r \sum_{\substack{P\in \mathcal{A}\\P\subsetneq [0,1)}} (\eta_{\pi(P)} + \eta_P)\\
    &\leq 2 r (q+1) \sum_{P\in \mathcal{A}} \eta_P
      \leq \frac{\eta}{4}.
  \end{align*}
  Inserting the above estimate into the first result of \Cref{pro:variation-norm} yields
  \begin{equation*}
    \|(\widetilde{D} - \overline{D})\colon Z_V\to Z_V\|
    \leq 2\cdot \|\widetilde{D} - \overline{D}\|_{\{1, \dots, r\}}
    \leq \frac{\eta}{2}.
  \end{equation*}

  To summarize, we combine~\eqref{eq:138} with the latter estimate and record that
  \begin{equation}\label{eq:185}
    \|(\widetilde{D} - \widehat{D})\colon Z_V\to Z_V\|
    \leq \eta,
  \end{equation}

  Now we analyze $\lambda_{\mathcal{U}}^k$ for the constructed diagonal operators above.  To this end, let $1\leq t\leq s$ and $k\in \mathcal{K}_t$
  and first observe that by \Cref{dfn:functionals}, \eqref{eq:189} and~\eqref{eq:218}, we obtain
  \begin{equation}\label{eq:190}
    \lambda_{\mathcal{U}}^k(\widehat{D})
    = \int M^{\widehat{D},k}
    = \int M^{\overline{D},k_t}
    = \int M^{\overline{D},k}
    = \lambda_{\mathcal{U}}^k(\overline{D}).
  \end{equation}
  Second, \eqref{eq:114} and~\eqref{eq:188} yield
  \begin{equation}\label{eq:187}
    \begin{aligned}
      |\lambda_{\mathcal{U}}^k(\widetilde{D}) - \lambda_{\mathcal{U}}^k(\overline{D})|
      &= \lim_{i\to \mathcal{U}} \Bigl| \int A_i^{\widetilde{D},k} - A_i^{\overline{D},k}\Bigl|
        = \lim_{i\to \infty} \Bigl| \int \widehat{d}_{[0,1)}^{k_t} - A_i^{\widehat{D},k_t} \Bigl|\\
      &= \Bigl| \int A_0^{\widehat{D},k_t} - M^{\widehat{D},k_t} \Bigl|
        \leq \eta.
    \end{aligned}
  \end{equation}
  Thus, using~\eqref{eq:114} we obtain
  \begin{equation}\label{eq:186}
    \begin{aligned}
      \|S - \widetilde{D}\|
      &= \Bigl\| \sum_{t=1}^{s} \lambda_{\mathcal{U}}^{k_t}(D) Q_{\mathcal{K}_t}^{\mathcal{B}} - \sum_{t=1}^{s} \lambda_{\mathcal{U}}^{k_t}(\widetilde{D}) Q_{\mathcal{K}_t}^{\mathcal{B}} \Bigr\|
        \leq \sum_{t=1}^{s} |\lambda_{\mathcal{U}}^{k_t}(D) - \lambda_{\mathcal{U}}^{k_t}(\widetilde{D})|\cdot \|Q_{\mathcal{K}_t}^{\mathcal{B}}\|.
    \end{aligned}
  \end{equation}
  Combining~\eqref{eq:190}, \eqref{eq:187} and~\eqref{eq:184} gives us
  \begin{equation}\label{eq:192}
    |\lambda_{\mathcal{U}}^{k_t}(D) - \lambda_{\mathcal{U}}^{k_t}(\widetilde{D})|
    \leq \eta + |\lambda_{\mathcal{U}}^{k_t}(D) - \lambda_{\mathcal{U}}^{k_t}(\widehat{D})|
    \leq \eta + \frac{6\kappa}{1-2\kappa} \|D\|.
  \end{equation}
  Inserting~\eqref{eq:192} into~\eqref{eq:186} yields
  \begin{equation}\label{eq:193}
    \|S - \widetilde{D}\|
    \leq \bigl(\eta + \frac{6\kappa}{1-2\kappa} \|D\|\bigr) \sum_{t=1}^{s} \|Q_{\mathcal{K}_t}^{\mathcal{B}}\|.
  \end{equation}

  Since $\eta, \kappa > 0$ were chosen arbitrarily, and all the canonical projections $Q_{\mathcal{K}_t}^{\mathcal{B}}$, $1\leq t\leq s$ are bounded,
  we see that $S$ almost projectionally factors through $\widetilde{D}$ with constant $1^+$.  We conclude the proof by recalling that $\widehat{D}$
  almost projectionally factors through $D$ and referring to~\eqref{eq:185}.
\end{proof}

In the following corollary, we identify subspaces $E,F$ of $Z_V$ via the linear functionals $\lambda_{\mathcal{U}}^k$ such that $\Id_E$ and $\Id_F$ factor through
$D$ and $\Id_{Z_V} - D$, respectively.
\begin{cor}\label{cor:factor-1}
  Let $Z_V = Z_V((\mathcal{F}_j),X,\boldsymbol{\sigma};\mathbb{R}^l)$ denote a Janson-Sobolev space or a Janson-Sobolev-Hardy space, and let
  $D\colon Z_V\to Z_V$ denote a bounded diagonal operator.  Pick any non-principal ultrafilter $\mathcal{U}$ on $\mathbb{N}_0$ and define
  \begin{equation}\label{eq:194}
    \begin{aligned}
      \mathcal{K}
      &= \mathcal{K}(D)
        = \{ 1\leq k\leq r : \lambda_{\mathcal{U}}^k(D) \neq 0\},\\
      \mathcal{L}
      &= \mathcal{L}(D)
        = \{ 1\leq k\leq r : \lambda_{\mathcal{U}}^k(D) \neq 1\}.
    \end{aligned}
  \end{equation}
  Define the constants $\gamma = \sum_{t=1}^s \|Q_{\mathcal{K}_t}\|$, where $\mathcal{K}_t$, $1\leq t\leq s$ is the finest bounded partition, and put
  \begin{equation}\label{eq:202}
    \alpha
    = \alpha(D)
    = \max_{k\in \mathcal{K}}|\lambda_{\mathcal{U}}^k(D)|^{-1} \gamma,
    \qquad
    \beta
    = \beta(D)
    = \max_{k\in \mathcal{L}}|1 - \lambda_{\mathcal{U}}^k(D)|^{-1} \gamma.
  \end{equation}
  The spaces $E$ and $F$ are given by
  \begin{equation}\label{eq:195}
    E
    = Q_{\mathcal{K}}(Z_V),
    \qquad
    F
    = Q_{\mathcal{L}}(Z_V).
  \end{equation}
  
  Then $\Id_E$ factors through $D$ with constant $\alpha^+$ and $\Id_F$ factors through $\Id_{Z_V} - D$ with constant $\beta^+$.  In particular, if
  $\mathcal{K} = \{1, \dots , r\}$, then $\Id_{Z_V}$ factors through $D$ with constant $\alpha^+$, and if $\mathcal{L} = \{1, \dots , r\}$, then
  $\Id_{Z_V}$ factors through $\Id_{Z_V} - D$ with constant $\beta^+$.
\end{cor}

\begin{proof}
  First, we note that since $\lambda_{\mathcal{U}}^k(\Id_{Z_V}) = 1$, $1\leq k\leq r$ and $\mathcal{L}(\Id_{Z_V} - D) = \mathcal{K}(D)$, we only need
  to prove the case where $\Id_E$ factors through $D$.

  By \Cref{mainthm:picewise-constant} the operator $S\colon Z_V\to Z_V$ given by
  \begin{equation}
    S
    = \sum_{t=1}^{s} \lambda_{\mathcal{U}}^{k_t}(D) Q_{\mathcal{K}_t}^{\mathcal{B}}
  \end{equation}
  almost projectionally factors through $D$ with constant~$1^{+}$, i.e., for any $\eta > 0$, there exist bounded operators $A,B\colon Z_V\to Z_V$ with
  $AB = \Id_{Z_V}$ and such that
  \begin{equation}\label{eq:200}
    \|A\|\cdot \|B\| \leq 1+\eta
    \qquad\text{and}\qquad
    \|S - ADB\|\leq \eta.
  \end{equation}
  Define the bounded operator $T\colon Z_V\to E$ by
  \begin{equation}\label{eq:196}
    T
    = \sum_{\{t : \mathcal{K}_t\subset \mathcal{K}\}} \bigl(\lambda_{\mathcal{U}}^{k_t}(D)\bigr)^{-1} Q_{\mathcal{K}_t},
  \end{equation}
  note that $\|T\|\leq \alpha$ and $TS|_E = \Id_E$ and observe that by~\eqref{eq:200}
  \begin{equation}\label{eq:197}
    \|\Id_E - TADB|_E\|
    = \|T(S|_E - ADB|_E)\|
    \leq \eta \|T\|.
  \end{equation}
  For $\eta < \|T\|^{-1}$, the Neumann series implies that $TADB|_E$ is invertible and
  \begin{equation}\label{eq:198}
    \|(TADB|_E)^{-1}\|
    \leq \frac{1}{1 - \eta \|T\|}.
  \end{equation}
  Thus, $\Id_E = A'DB'$, where $A' = (TADB|_E)^{-1}TA$, $B' = B|_E$.  Using~\eqref{eq:200}, \eqref{eq:198} and~\eqref{eq:196} yields
  \begin{equation}\label{eq:199}
    \|A'\|\cdot \|B'\|
    \leq \frac{1+\eta}{1 - \eta \|T\|} \|T\|
    \leq \frac{1+\eta}{1 - \eta \|T\|} \alpha.
  \end{equation}
  Since $\eta > 0$ can be chosen arbitrarily close to $0$ in \eqref{eq:199}, it follows that $\Id_E$ factors through $D$ with constant $\alpha^+$.
\end{proof}

We next repeat and then prove our second main result, \Cref{mainthm:main-prop-I}.
\begin{thm}[\Cref{mainthm:main-prop-I}]\label{thm:main-prop-I}
  Every Janson-Sobolev space or Janson-Sobolev-Hardy space $Z_V$ with Property~I has the $2^+$-diagonal primary factorization property.
\end{thm}

\begin{proof}
  Since $Z_V$ has Property~I (see \Cref{dfn:space-type}), there exists a basis $\mathcal{B}$ of $V$ such that the only bounded canonical projections
  $Q_{\mathcal{K}}^{\mathcal{B}}\colon Z_V\to Z_V$ with respect to $\mathcal{B}$ are either $Q_{\mathcal{K}}^{\mathcal{B}} = \Id_{Z_V}$ or
  $Q_{\mathcal{K}}^{\mathcal{B}} = 0$, i.e., the finest bounded partition of $\{1, \dots, r\}$ is $\{1, \dots, r\}$ itself.

  Using \Cref{thm:agreement} and \Cref{dfn:functionals} yields $\lambda:=\lambda_{\mathcal{U}}^k(D) = \lambda_{\mathcal{U}}^{k'}(D)$ for all
  $k,k'\in \{1, \dots, r\}$, and we note that $\lambda\geq 1/2$ or $1-\lambda \geq 1/2$.  Thus, the constants $\alpha(D) = |\lambda|^{-1}$ and
  $\beta(D) = |1-\lambda|^{-1}$ defined in~\eqref{eq:202} satisfy either $\alpha(D)\leq 2$ and $\mathcal{K}(D) = \{1, \dots, r\}$ or $\beta(D)\leq 2$
  and $\mathcal{L}(D) = \{1, \dots, r\}$.  Noting that $\gamma = 1$ and applying \Cref{cor:factor-1} gives us that $\Id_{Z_V}$ either factors through
  $D$ or $\Id_{Z_V} - D$ with constant $2^{+}$ as desired.
\end{proof}

For Janson-Sobolev and Janson-Sobolev-Hardy spaces equipped with the $L^1$-norm, we restate below our main theorem \Cref{mainthm:main-L1}, which gives
the following additional results beyond Property~I.
\begin{thm}\label{thm:main-L1}
  Let $Z_V(L^1) = Z_V((\mathcal{F}_j),L^1,\boldsymbol{\sigma};\mathbb{R}^l)$ denote a Janson-Sobolev space or a Janson-Sobolev-Hardy space.  Then
  $Z_V(L^1)$ has the diagonal primary factorization property.  Moreover, if $Z_V(L^1)$ has Property~IIB, then $Z_V(L^1)$ also has the primary
  factorization property.
\end{thm}

\begin{proof}
  By \Cref{thm:trichotomy}, the space $Z_V(L^1)$ satisfies one of the following: Property~I, Property~IIA, or Property~IIB.
  
  In the case where $Z_V(L^1)$ satisfies Property~I, the conclusion follows directly from \Cref{mainthm:main-prop-I}.
  
  If $Z_V(L^1)$ is either a Janson-Sobolev-Hardy space or a Janson-Sobolev space with Property~IIB, then by \Cref{thm:isomorphic-classification-H1}
  $Z_V(L^1)$ is isomorphic to $H^1$.  By \Cref{rem:H1}, $H^1$ has the primary factorization property, and consequently, so does $Z_V(L^1)$.

  If $Z_V(L^1)$ is a Janson-Sobolev space with Property~IIA, \Cref{thm:isomorphic-classification-L1} gives us that $Z_V(L^1)$ is isomorphic to
  $Z_{\operatorname{span}\{b^t\}}(L^1)$, which clearly is a space with Property~I.  Thus, by \Cref{mainthm:main-prop-I},
  $Z_{\operatorname{span}\{b^t\}}(L^1)$ and therefore also $Z_V(L^1)$ has the diagonal primary factorization property (see
  \Cref{rem:diagonal-primary-factorization-property}).
\end{proof}

Observe that in the above proof, if $Z_V(L^1)$ has either Property~IIA or Property~IIB, then the preceding argument is valid for any choice of basis
of $V$.

In \Cref{sec:an-example-non}, we gave an example of a Janson-Sobolev space $Z_V(L^1)$ satisfying neither Property~IIA nor Property~IIB.  However, it
follows from \Cref{thm:trichotomy}, established in the next section, that $Z_V(L^1)$ has Property~I.  Therefore, by \Cref{mainthm:main-prop-I},
$Z_V(L^1)$ has the diagonal primary factorization property.  In contrast, \Cref{thm:non-primary-janson-sobolev} implies that $Z_V(L^1)$ fails to
possess the primary factorization property.

%

\section{Plateau martingales}
\label{sec:plateau-martingales}

N.~Nadirashvili's constructed three bounded harmonic functions~\cite{MR1419004} in the unit disc under differential constraints with unbounded
variance along each ray from the origin to the boundary of the disc.  Later in~\cite{MR1757799}, P.~F.~X.~Müller published the following martingale
model of this construction.

A sequence of semi-continuous functions $G_j\colon \Omega\to \mathbb{R}^3$, $j\in \mathbb{N}_0$, with respect to the filtration
$(\mathcal{G}_j)_{j=0}^{\infty}$ is called a \emph{Plateau martingale} provided that
\begin{equation}\label{eq:154}
  \Delta_j^{(\mathcal{G}_n)} G_j
  = R_{j,0}g + R_{j,1}h,
\end{equation}
where $g,h\in \mathbb{R}^3$ satisfy
\begin{equation}\label{eq:157}
  |g| - |h|
  = 0
  = \langle g, h \rangle
\end{equation}
and the independent random variables $R_{j,m}\colon \Omega\to \mathbb{R}$, $m=0,1$ are also independent of $\mathcal{G}_j$, measurable with respect to $\mathcal{G}_{j+1}$, and such that
\begin{equation}\label{eq:158}
  \cond(R_{j,m} | \mathcal{G}_j)
  = 0
  \qquad\text{and}\qquad
  \cond(R_{j,m}^2 | \mathcal{G}_j)
  = 1.
\end{equation}

P.~F.~X.~Müller's model is the concrete realization of the extension of the notion of a conformal (complex valued) martingale to $\mathbb{R}^3$.  Recall that a
\emph{conformal martingale} is a complex valued martingale $G_j = F_{j,0} + i F_{j,1}$ satisfying
\begin{equation}\label{eq:159}
  \begin{aligned}
    \cond(\|\Delta_jF_{j,0}\|_{\ell^2_2}^2 | \mathcal{G}_{j-1})
    &= \cond(\|\Delta_jF_{j,1}\|_{\ell^2_2}^2 | \mathcal{G}_{j-1}),\\
    \cond(\langle\Delta_jF_{j,0}, \Delta_jF_{j,1}\rangle | \mathcal{G}_{j-1})
    &= 0.
  \end{aligned}
\end{equation}

We use~\eqref{eq:159} to now define a model for $q$-adic filtrations with Janson-type restrictions, which can be viewed as a discrete counterpart for
differential constraints.  First, we put $\mathcal{G}_0 = \mathcal{F}_0$ and $\mathcal{G}_j = \mathcal{F}_{2j}$ and
$\mathcal{F}_{j,m} = \mathcal{F}_{2j-1+m}$, $j\in \mathbb{N}$, $m=0,1$.  A martingale $(G_j)_{j=0}^{\infty}$ with respect to the filtration
$(\mathcal{G}_j)_{j=0}^{\infty}$ is called a \emph{Janson Plateau martingale} if $G_0 = 0$ and for all $j\in \mathbb{N}$
\begin{equation}\label{eq:160}
  \Delta_j G_j
  = F_{j,0} + F_{j,1},
\end{equation}
where $F_{j,m}\colon [0,1)\to \mathbb{R}^l$ belongs to $Z_V$, is $\mathcal{F}_{j,m}$-measurable, and satisfies
\begin{equation}\label{eq:161}
  \cond( F_{j,1} | \mathcal{F}_{j,0})
  = 0,
  \qquad
  \langle F_{j,0}, F_{j,1} \rangle_{\ell^2_{\ell}}
  = 0
\end{equation}
and
\begin{equation}\label{eq:162}
  \cond ( \|F_{j,m}\|_{\ell^2_l}^2 | \mathcal{G}_{j-1})
  = \alpha_j^2,
  \qquad m=0,1.
\end{equation}
The core features of a Janson Plateau martingale are illustrated in \Cref{fig:plateau}.

\begin{figure}
  \centering
  \begin{tikzpicture}
    \node[anchor=south west,inner sep=0] (img) {\includegraphics[width=.8\textwidth]{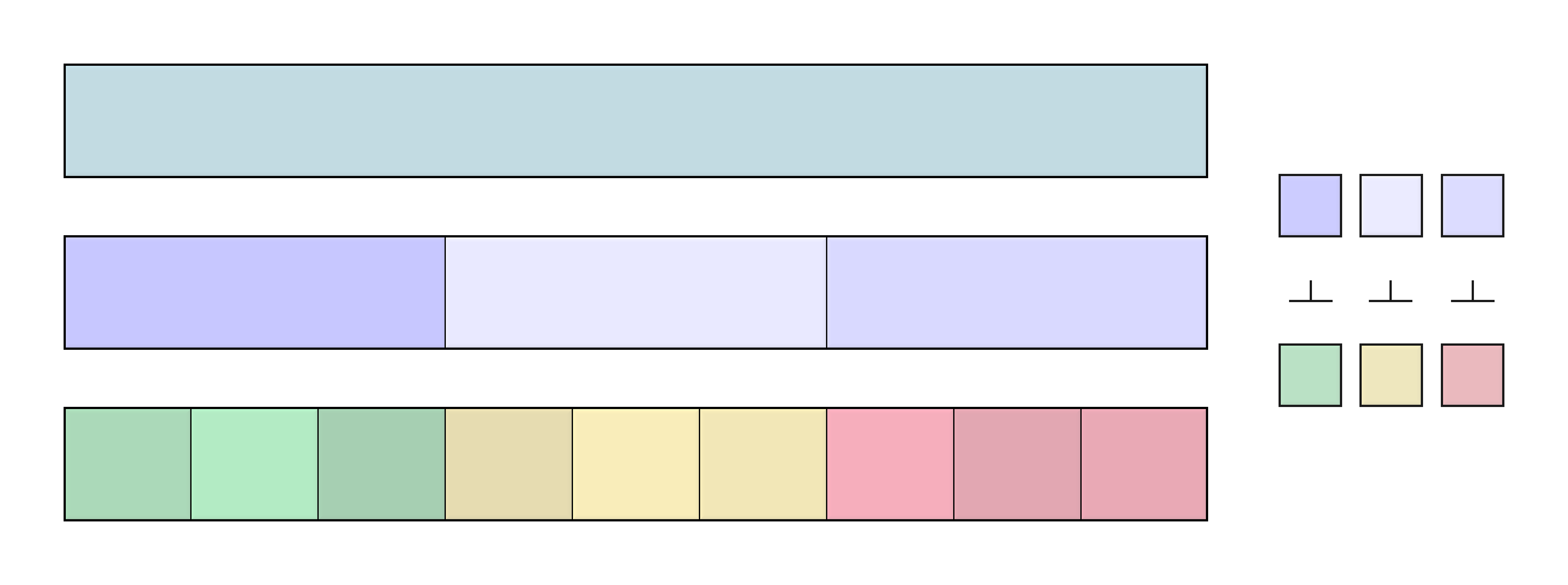}};
    \begin{scope}[x={(img.south east)},y={(img.north west)}]
      \node at (0.4,0.775) {{\small $P\in \mathcal{G}_{j-1}$}};%
      \node at (-0.03,0.775) {{\small $G_{j-1}$}};%
      \node at (-0.03,0.49) {{\small $F_{j,0}$}};%
      \node at (0.16,0.49) {{\small $Q\in \mathcal{F}_{j,0}$}};%
      \node at (-0.03,0.2) {{\small $F_{j,1}$}};%
    \end{scope}
  \end{tikzpicture}
  \caption{Depicted by the colors are the values of the functions $G_{j-1}$, $F_{j,0}$, and $F_{j,1}$, whose sum equals $G_j$, each restricted to the
    atom $P\in\mathcal{G}_{j-1}$.  On the right, we indicate that each value of $F_{j,0}$ represented in the top row of the small squares is orthogonal to all the
    corresponding values of $F_{j,1}$ represented by the colored squares in the second row.  For example, the value of $F_{j,0}$ on $Q$ is orthogonal
    to each of the values of $F_{j,1}$ represented by the green color tones.}
  \label{fig:plateau}
\end{figure}

\begin{dfn}\label{dfn:strong-column}
  Given $c > 0$, we say that $V$ is a \emph{$c$-strong column subspace} of $\mathbb{R}^{l\times q}_0$ if for all matrices $[v_1| \dots | v_q]\in V$ we have that
  $\|v_j\|_{\ell^2_l}\geq c \|v\|_V$ for all $1\leq j\leq q$.
\end{dfn}

We now construct an example of such a space.  Pick any $r\geq 1$, $q\geq 2$ and unit vectors $w^k = (w^k_j)_{j=1}^q$ with $\min_j |w_j^k| \geq c$ in
$\mathbb{R}^q_0$, $1\leq k\leq r$.  Now, we pick $l\geq r$ and define $v^k\in \mathbb{R}_0^{l\times q}$ to be the matrix whose $k$-th row is
$w^k/\sqrt{q}$ and is zero otherwise.  Note that $(v^k)_{k=1}^r$ is an orthonormal basis for $V = \spn\{ v^k : 1\leq k\leq r\}$.  Now let
$v = \sum_{k=1}^{r} a_k v^k \in V$ and observe that
\begin{equation}\label{eq:166}
  \Bigl\| \sum_{k=1}^r a_k v_j^k \Bigr\|_{\ell^2_l}^2
  = \sum_{i=1}^l \bigl|\sum_{k=1}^r a_k \langle v_j^k, e_i \rangle \bigr|^2
  = \frac{1}{q}\sum_{i=1}^l |a_i|^2 |w_j^i|^2
  \geq \frac{c^2}{q} \sum_{i=1}^l |a_i|^2
  = c^2 \|v\|_V^2.
\end{equation}

\begin{thm}\label{thm:janson-plateau}
  Suppose that $V$ is a $c$-strong column subspace of $\mathbb{R}^{l\times q}_0$ with dimension $r > q$.  Then there exists a Janson Plateau martingale
  $(G_j)_{j=0}^{\infty}$ which is closed in $Z_V$ such that
  \begin{equation}\label{eq:163}
    \sup_{j\in \mathbb{N}_0} \|G_j(x)\|_{\ell^2_l}
    \leq 1
    \qquad\text{and}\qquad
    \sum_{j=0}^{\infty} \|\Delta_jG_j(x)\|_{\ell^2_l}
    = \infty
  \end{equation}
  for all $x\in [0,1)$.
\end{thm}

In the proof of the theorem, we will use the following lemma.
\begin{lem}\label{lem:janson-plateau}
  Let $V$ denote a linear subspace of $\mathbb{R}^{l\times q}_0$ with dimension $r > q$ and put
  \begin{equation}\label{eq:column-space}
    \mathcal{C}
    = \spn\{ m_1, \dots , m_q : [m_1| \dots | m_q]\in V\}.
  \end{equation}
  Then for each $m\in \mathcal{C}$, we can find $n = [n_1| \dots | n_q]\in V$ such that $n_j\perp m$ for all $1\leq j\leq q$.
\end{lem}

\begin{proof}[Proof of \Cref{lem:janson-plateau}]
  First, recall that $(v^k)_{k=1}^r$ is an orthonormal basis for $V$.  Note that
  \begin{equation}\label{eq:209}
    \mathcal{C}
    = \spn\{ v_j^k : 1\leq j\leq q,\ 1\leq k\leq r\}
  \end{equation}
  and observe that $m\in \mathcal{C}$ if and only if there exist vectors $b^k = (b_j^k)_{j=1}^q\in \mathbb{R}^q$, $1\leq k\leq r$ such that
  \begin{equation}\label{eq:22}
    m
    = \sum_{k=1}^r \sum_{j=1}^q v_j^k b_j^k
    = \sum_{k=1}^r v^k b^k. 
  \end{equation}
  Next, given such $m\in \mathcal{C}$, observe that for all $n = [n_1| \dots | n_q]\in V$, we have that $n_j\perp m$, $1\leq j\leq q$ if and only if
  $m\in \ker n^t$.  Thus, the following conditions are equivalent:
  \begin{enumerate}[(a)]\label{enu:plateau-conditions}
  \item\label{enu:plateau-conditions:a} For all $m\in \mathcal{C}$ there exists $n = [n_1| \dots | n_q]\in V$ such that $n_j\perp m$ for all $1\leq j\leq q$.
  \item\label{enu:plateau-conditions:b} For all $b^1, \dots b^r\in \mathbb{R}^q$, there exist $a_1, \dots a_r\in \mathbb{R}$ such that
    \begin{equation}\label{eq:23}
      \Bigl(\sum_{k_0=1}^r a_{k_0} v^{k_0}\Bigr)^t
      \cdot \sum_{k_1=1}^r v^{k_1} b^{k_1}
      = 0.
    \end{equation}
  \end{enumerate}

  Using the above equivalence, we obtain that
  \begin{equation}\label{eq:211}
    \Bigl(\sum_{k_0=1}^r a_{k_0} v^{k_0}\Bigr)^t
    \cdot \sum_{k_1=1}^r v^{k_1} b^{k_1}
    = \sum_{k_0=1}^r a_{k_0}
    \sum_{k_1=1}^r (v^{k_0})^t\cdot v^{k_1} b^{k_1}
    = 0
  \end{equation}
  for some choice of $a_1, \dots, a_r\in \mathbb{R}$, since the vectors $\sum_{k_1=1}^r (v^{k_0})^t\cdot v^{k_1} b^{k_1}$, $1\leq k_0\leq r$ are in $\mathbb{R}^q$ and we have $r > q$.
\end{proof}

\begin{proof}[Proof of \Cref{thm:janson-plateau}]
  We will construct $(G_{j})_{j=0}^{\infty}$, $(F_{j,0})_{j=0}^{\infty}$, and $(F_{j,1})_{j=0}^{\infty}$ inductively.  Fix a sequence
  $(\alpha_j)_{j=0}^{\infty} \in \ell^2 \setminus \ell^1$ with $\|(\alpha_j)_{j=0}^{\infty}\|_{\ell^2} = \frac{1}{\sqrt{2q}}$ and let $G_0$ be the constant function equal to
  $0 \in \mathbb{R}^{l}$.  Assume we have already constructed $G_{j-1}\in Z_V((\mathcal{F}_j))$ which is measurable with respect to
  $\mathcal{G}_{j-1} = \mathcal{F}_{2j-2}$, we proceed as follows.  For each $x \in [0,1)$, there exists a unique atom $P \in \mathcal{AF}_{2j-2}$ such that
  $x \in P$.  The random variable $G_{j-1}$ is constant on $P$.  Note that by~\eqref{eq:column-space}, we have $G_{j-1}(x) \in \mathcal{C}$, and thus by
  \Cref{lem:janson-plateau}, there exists a matrix $n(P) = [n_1(P) | \dots | n_q(P)] \in V$ such that $G_{j-1}(x) \perp n_t(P)$ for
  $t \in \{1, \dots, q\}$.  We normalize $\|n\|_V = \alpha_j$ and define
  \begin{equation}\label{eq:165}
    F_{j,0}
    = \sum_{P \in \mathcal{AF}_{2j-2}} \sum_{R \in \mathcal{S}(P)} n_{\mathrm{ord}(R)}(P) \cha_R.
  \end{equation}
  Observe that $F_{j,0}$ is $\mathcal{F}_{2j-1}$-measurable and constant on each $R \in \mathcal{AF}_{2j-1}$.  Since
  $G_{j-1}(x) + F_{j,0}(x) \in \mathcal{C}$, by \Cref{lem:janson-plateau} we find $\tilde{n}(R) = [\tilde{n}_1(R) | \dots | \tilde{n}_q(R)] \in V$ such that
  $\tilde{n}_t(R) \perp (G_{j-1}(x) + F_{j,0}(x))$ for all $t\in \{1, \dots, q\}$.  By rescaling, we may assume $\|\tilde{n}\|_V = \alpha_j$, and we put
  \begin{equation}\label{eq:167}
    F_{j,1}
    = \sum_{R \in \mathcal{AF}_{2j-1}} \sum_{Q \in \mathcal{S}(R)} \tilde{n}_{\mathrm{ord}(Q)}(R) \cha_Q.
  \end{equation}
  Clearly, $F_{j,1}$ is $\mathcal{F}_{2j}$-measurable, and therefore, the function
  \begin{equation}\label{eq:156}
    G_{j}
    := G_{j-1} + F_{j,0} + F_{j,1}
  \end{equation}
  is $\mathcal{G}_j$-measurable.

  To summarize, we constructed $G_{j} = G_{j-1} + F_{j,0} + F_{j,1}$ such that
  \begin{equation}\label{eq:155}
    G_{j-1}(x) \perp F_{j,0}(x)
    \qquad\text{and}\qquad
    G_{j-1}(x) + F_{j,0}(x) \perp F_{j,1}(x)
  \end{equation}
  for all $x\in[0,1)$.  One can easily check that $(G_j)_{j=0}^{\infty}$ is indeed a Janson Plateau martingale.

  To compute $\|G_j(x)\|_{\ell^2_l}$ for $x \in [0,1)$, let $Q \in \mathcal{AF}_{2j}$ be such that $x \in Q \in \mathcal{S}(R)$ and
  $R \in \mathcal{S}(P)$.  Using~\eqref{eq:156} and~\eqref{eq:155}, we obtain
  \begin{equation}
    \begin{split}
      \|G_j(x)\|^2_{\ell^2_l}
      &= \|G_{j-1}(x) + F_{j,0}(x) + F_{j,1}(x)\|^2_{\ell^2_l}\\
      &= \|G_{j-1}(x) + F_{j,0}(x)\|^2_{\ell^2_l} + \|F_{j,1}(x)\|^2_{\ell^2_l}\\
      &= \|G_{j-1}(x)\|^2_{\ell^2_l} + \|F_{j,0}(x)\|^2_{\ell^2_l} + \|F_{j,1}(x)\|^2_{\ell^2_l}.
    \end{split}
  \end{equation}
  By~\eqref{eq:165} and~\eqref{eq:167}, we obtain
  \begin{equation}\label{eq:164}
    \|G_j(x)\|^2_{\ell^2_l}
    = \|G_{j-1}(x)\|^2_{\ell^2_l} + \|n_{\mathrm{ord}(R)}(P)\|^2_{\ell^2_l} + \|\tilde{n}_{\mathrm{ord}(Q)}(R)\|^2_{\ell^2_l}.
  \end{equation}
  Thus,
  \begin{equation}\label{eq:168}
    \|G_j(x)\|^2_{\ell^2_l}
    \leq \|G_{j-1}(x)\|^2_{\ell^2_l} + q \|n\|^2_V + q \|\tilde{n}\|^2_V
    \leq 2 q \sum_{t=1}^{j} \alpha_t^2
    \leq 2 q \sum_{t=1}^{\infty} \alpha_t^2
    =1.
  \end{equation}
  On the other hand, by the $c$-strong columns property of V and an analogous calculation, we have
  \begin{equation}
    \|\Delta_{j}^{(\mathcal{G}_j)}G_j(x)\|^2_{\ell^2_l}
    = \|n_{\mathrm{ord}(R)}(P)\|^2_{\ell^2_l} + \|\tilde{n}_{\mathrm{ord}(Q)}(R)\|^2_{\ell^2_l}
    \geq c^2\|n\|^2_V + c^2\|\tilde{n}\|^2_V
    = 2c^2 \alpha_j^2.
  \end{equation}
  This implies
  \begin{equation}
    \sum_{j \in \mathbb{N}} \|\Delta_{j}^{(\mathcal{G}_j)}G_j(x)\|
    \geq \sqrt{2} c \sum_{j \in \mathbb{N}} |\alpha_j|
    = +\infty.  
  \end{equation}

  Finally, fix $i < j\in \mathbb{N}_0$ and observe that the orthogonality in~\eqref{eq:155} gives us
  \begin{align}\label{eq:169}
    \|G_i - G_j\|_{Z_V}
    &= \Bigl\|\sum_{n=i+1}^j F_{n,0} + F_{n,1}\Bigr\|_{Z_V}
      = \Bigl\|
      x\mapsto \cond_{\boldsymbol{\sigma}}
      \Bigl\|
      \sum_{n=i+1}^j \sigma_{2n-1} F_{n,0} + \sigma_{2n} F_{n,1}
      \Bigr\|_{\ell^2_l}
      \Bigr\|_X\\
    &= \Bigl\|
      x\mapsto       
      \Bigl(\sum_{n=i+1}^j \|F_{n,0}\|_{\ell^2_l}^2 + \|F_{n,1}\|_{\ell^2_l}^2\Bigr)^{1/2}
      \Bigr\|_X.
  \end{align}
  Using similar steps as in the calculation of~\eqref{eq:168} and exploiting the monotonicity of the norm of $X$ yields
  \begin{equation}
    \|G_i - G_j\|_{Z_V}
    \leq \sqrt{2q} \Bigl(\sum_{n=i+1}^j \alpha_n^2\Bigr)^{1/2},
  \end{equation}
  and thus, $(G_j)_{j=0}^{\infty}$ is a Cauchy sequence in $Z_V$.
\end{proof}

%

\section*{Acknowledgements}
\label{sec:acknowledgements}

The authors declare that no artificial intelligence tools were used in the conception, derivation, verification, or development of the mathematical
content of this work.

We would like to thank M.~Wojciechowski for suggesting to investigate Janson-Sobolev and Janson-Sobolev-Hardy spaces, and P.~F~.~X.~Müller for
suggesting the problem concerning Plateau martingales.

%

\bibliographystyle{abbrv}%
\bibliography{bibliography}%

\end{document}